\documentclass[a4paper, 11pt]{amsart}
\usepackage[british]{babel}
\usepackage[utf8x]{inputenc}
\usepackage[T1]{fontenc}
\usepackage[
    backend=biber,
    style=alphabetic,
    maxnames=99,
    citestyle=alphabetic,
    maxalphanames=99,
    ]{biblatex}
\renewbibmacro*{volume+number+eid}{
    \printfield{volume}
    \setunit*{\addnbthinspace}
    \printfield{number}
    \setunit{\addcomma\space}
    \printfield{eid}}
  \DeclareFieldFormat[article]{number}{\mkbibparens{#1}}
\renewbibmacro{in:}{}

\usepackage[a4paper,top=3cm,bottom=3cm,left=3cm,right=3cm,marginparwidth=1.75cm]{geometry}

\usepackage{amsmath}
\numberwithin{equation}{section}
\usepackage{amsfonts}
\usepackage{bbm}
\usepackage{amssymb}
\usepackage{graphicx}
\usepackage{dsfont}
\usepackage[colorinlistoftodos]{todonotes}
\usepackage[colorlinks=true, allcolors=blue]{hyperref}
\usepackage{enumitem}
\usepackage{amsthm}
\usepackage{tikz-cd}
\usepackage{quiver}
\usepackage{mathrsfs}
\usepackage[dvipsnames]{xcolor}
\usepackage{hyperref}
\hypersetup{
    linkcolor=purple,
    urlcolor=purple,
    citecolor=purple,
    }

\usepackage{tikz}
\usetikzlibrary{knots}
\usetikzlibrary{matrix}
\usetikzlibrary{decorations.pathreplacing}
\usetikzlibrary{decorations.markings}
\usetikzlibrary{arrows}
\usetikzlibrary{calc}
\usetikzlibrary{shapes.misc}
\usetikzlibrary{fit}
\usepgflibrary{decorations.pathmorphing}
\usepgflibrary{shapes.geometric}
\usepackage{yhmath}
\usepackage{Cobordism}

\newcommand{\pa}[1]{\left( #1 \right)}

\newcommand{\set}[1]{\left\{ #1 \right\}}
\newcommand{\ol}[1]{\overline{#1}}

\newcommand{\Bord}[1]{\mathbf{Bord}_{1,2,3}^{\mathrm{#1}}}
\newcommand{\bord}[1]{\mathbf{Bord}_{2,3}^{\mathrm{#1}}}
\newcommand{\vect}{\mathbf{Vect}_k}
\newcommand{\Vect}{2\mathbf{Vect}_k}
\newcommand{\KV}{2\mathbf{Vect}_k^{\mathrm{KV}}}

\newcommand{\parin}{\partial_{\mathrm{in}}}
\newcommand{\parout}{\partial_{\mathrm{out}}}
\newcommand{\parIn}{\partial_{\mathrm{in}}'}
\newcommand{\parOut}{\partial_{\mathrm{out}}'}

\newcommand{\textcolour}[2]{\textcolor{#1}{#2}}

\newcommand\II{\ensuremath{\mathrm{II}}}

\newcommand\fixboundingbox{\path [use as bounding box, draw=none] (current bounding box.north west) rectangle (current bounding box.south east);}
\newcommand\selectpart[2][\selectcolour]{\fixboundingbox\begin{pgfonlayer}{selectionbox}\node [draw=red, fit=#2, inner sep=0.8*\cobordismlinewidth, #1, line width=\cobordismlinewidth] {};\end{pgfonlayer}}

\include{arrows}

\newenvironment{tz}[1][]{\begin{tikzpicture}[baseline={([yshift=-.8ex]current bounding box.center)},#1]}{\end{tikzpicture}}

\usepackage{etoolbox}

\makeatletter
\def\calign@preamble{%
   &\hfil\strut@
    \setboxz@h{\@lign$\m@th\displaystyle{##}$}%
    \ifmeasuring@\savefieldlength@\fi
    \set@field
    \hfil
    \tabskip\alignsep@
}
\let\cmeasure@\measure@
\patchcmd\cmeasure@{\divide\@tempcntb\tw@}{}{}{}
\patchcmd\cmeasure@{\divide\@tempcntb\tw@}{}{}{}
\patchcmd\cmeasure@{\ifodd\maxfields@
  \global\advance\maxfields@\@ne
  \fi}{}{}{}    
  
\makeatother

\def\smallbordisms{\scalecobordisms{0.5}\setlength\obscurewidth{0pt}}

\DeclareMathOperator{\Hom}{Hom}

\DeclareMathOperator{\id}{id}

\DeclareMathOperator{\tr}{tr}

\newtheorem{thm}{Theorem}[section]
\newtheorem{lem}[thm]{Lemma}
\newtheorem{prop}[thm]{Proposition}

\newtheorem{thmA}{Theorem}

\theoremstyle{definition}
\newtheorem{defn}[thm]{Definition}
\newtheorem{eg}[thm]{Example}
\newtheorem{constr}[thm]{Construction}

\theoremstyle{remark}
\newtheorem*{rk}{Remark}

\newcommand*{\newproofname}{Proof}
\newenvironment{proof*}[1][\newproofname]{\begin{proof}[#1]}{\end{proof}}

\title{All once-extended 3D TQFTs are Reshetikhin--Turaev theories}
\author{Glen Lim}
\address{Mathematical Institute, University of Oxford}
\email{glen.lim@maths.ox.ac.uk}

\begin{document}

\begin{abstract}
    We present a direct geometric construction which extends the Reshetikhin--Turaev TQFT to circles. We prove that this agrees with the generators-and-relations approach of Bartlett--Douglas--Schommer-Pries--Vicary, thus providing an alternative to the Cerf-theoretic part of their classification of once-extended 3-dimensional TQFTs, and deduce as a result that all once-extended 3-dimensional TQFTs arise via our construction.
\end{abstract}

\maketitle

\tableofcontents

\section{Introduction}

Formulated shortly after the axiomatic definition of topological quantum field theories (TQFTs) by Atiyah \cite{Atiyah-TQFT} and Segal \cite{Segal}, one of the earliest examples of TQFTs is the Reshetikhin--Turaev TQFT \cite{RT}. Based on Reshetikhin and Turaev's previous work on ribbon invariants \cite{RT-ribbon}, this is a geometric construction which made rigorous an approach of Witten \cite{Witten}, and may be summarised in modern language as follows: Given the data of a modular tensor category $\mathcal{C}$ and a choice of square root of its global dimension $\sqrt{D}$, it produces a symmetric monoidal functor $$rt_{\mathcal{C},\sqrt{D}}:\bord{sig}\rightarrow\vect.$$ from the \emph{signature bordism category} to the category of vector spaces. Here, the source category has as objects oriented surfaces and as morphisms oriented 3-dimensional bordisms with the extra data of an integer, which is thought of as the signature of a bounding 4-manifold.

A question that has been studied in various places in the literature is how one might extend this construction to circles, thus obtaining a once-extended TQFT in the sense of \cite{Lawrence}. Thus far, this has been done via distinct methods to the original approach of Reshetikhin and Turaev. The first complete construction was given by De Renzi \cite{dr-thesis}, who generalised the universal construction reformulation of $rt_{\mathcal{C},\sqrt{D}}$ given by Blanchet, Habegger, Masbaum and Vogel \cite{BHMV}.

Another approach is the one of Bartlett, Douglas, Schommer-Pries and Vicary \cite{BDSV4} using generators and relations. This involves the Cerf-theoretic construction of a finite presentation of $\Bord{sig}$, the proof of which will be given by combining results from \cite{BDSV3} with upcoming work of Bartlett, Douglas and Sytilidis \cite{pres-bord}. Using this, symmetric monoidal functors from $\Bord{sig}$ may be defined by specifying the images of the generators and checking that the relations are satisfied.

Via both of these, one then obtains from $\mathcal{C},\sqrt{D}$ a symmetric monoidal functor $$\Bord{sig}\rightarrow\Vect$$ from the \emph{signature bordism bicategory} to the bicategory of Cauchy-complete linear categories, linear functors and natural transformations.

In this paper, we present a direct geometric construction of a symmetric monoidal functor $$RT_{\mathcal{C},\sqrt{D}}:\Bord{sig}\rightarrow\Vect$$ which, like the original formulation of Reshetikhin and Turaev, is based on the ribbon invariants of \cite{RT-ribbon}. This may be summarised as:
\begin{thmA}[{Theorem \ref{thm:rt-functor}}]
    There exists a direct geometric construction which extends the Reshetikhin--Turaev TQFT to circles.
\end{thmA}

A significant advantage of this construction is the ease with which one may compute the images of 2-morphisms; these are given in terms of surgery presentations of 3-manifolds. One application of this involves the restriction to the \emph{half-signature bordism bicategory} $\Bord{sig/2}$ as presented in \cite{partA}; this is an index $2$ symmetric monoidal subbicategory of $\Bord{sig}$. By showing that $RT_{\mathcal{C},\pm\sqrt{D}}$ agree on this subbicategory, we prove:
\begin{thmA}[{Theorem \ref{thm:rt-functor2}}]
    Given a modular tensor category $\mathcal{C}$, we have a symmetric monoidal functor $RT_\mathcal{C}:\Bord{sig/2}\rightarrow\Vect$.
\end{thmA}

Another application involves the classification of once-extended 3-dimensional TQFTs. By giving a way to construct once-extended TQFTs directly from modular tensor categories, this provides an alternate approach to the Cerf-theoretic portion (i.e. \cite{haioun,filippos-thesis} and the upcoming \cite{pres-bord}) in the proof the classification theorems of Bartlett, Douglas, Schommer-Pries and Vicary. The interplay between the various results in this area is quite complex, so we refer the reader to the extended exposition of \textsection\ref{subsection:BDSV}, at the end of which our results regarding this are stated. Notably, our arguments will show in Theorem \ref{thmA:all-rt} that all once-extended 3-dimensional TQFTs arise as extended Reshetikhin--Turaev TQFTs, whence the title of this paper.

\subsection{Extending the Reshetikhin--Turaev TQFT to circles}

This paper is centred around a new direct construction which extends Reshetikhin--Turaev TQFT to circles. This is a problem which has studied in multiple places in the literature, and we provide a survey of the various approaches.

We first begin by discussing the original Reshetikhin--Turaev construction. Building upon their work on invariants of ribbon graphs \cite{RT-ribbon}, Reshetikhin and Turaev \cite{RT,Turaev-book} defined a symmetric monoidal functor $\bord{sig} \rightarrow \vect$. This is done by associating a ribbon graph with a single vertex and $g$ edges to each genus $g$ surface. Then, given the data of a modular tensor category (and a choice of square root of its global dimension), each connected surface is assigned a vector space of possible morphism labels of the single vertex of its associated ribbon graph. To each 3-dimensional bordism, one may then associate an embedded ribbon graph in a closed 3-manifold. The linear map assigned to this morphism can then be defined in terms of the invariant of this embedded ribbon graph.

The question of extending this construction to circles has been explored in various places in the literature. The earliest reference to this is a remark in the PhD thesis of Schommer-Pries \cite{csp-phd} that the Reshetikhin--Turaev construction can be extended to a once-extended TQFT by assigning the circle to the modular tensor category that one started with. The updated 2014 version of the thesis also references then-ongoing joint work with Bartlett, Douglas and Vicary \cite{BDSV1,BDSV2,BDSV3,BDSV4} whose aim was to classify once-extended 3-dimensional TQFTs in terms of modular tensor categories. This approach depends on Cerf-theoretic results which allow one to construct presentations of various bordism bicategories. The main result of this form, which we record as Theorem \ref{conj:pres-or}, regards the oriented bordism bicategory. Much of the proof of this has been completed by recent Cerf-theoretic work of Haïoun \cite{haioun} and Sytilidis \cite{filippos-thesis}; upcoming work of Bartlett, Douglas and Sytilidis \cite{pres-bord} will finish the proof of Theorem \ref{conj:pres-or}, so completing this approach. We refer the reader to \textsection\ref{subsection:BDSV} for more details. The status of Theorem \ref{conj:pres-or} notwithstanding, in the resulting classification theorem, the direction going from modular tensor categories to linear representations provides a construction which takes as input a modular tensor category (with a choice of square root of global dimension) and produces a once-extended TQFT. To the author's knowledge, it has not, however, been verified that this construction restricts to the Reshetikhin--Turaev TQFT. Nonetheless, this is the definition of the extended Reshetikhin--Turaev TQFT taken by Freed, Scheimbauer and Teleman \cite{FST} in their work in extending it to a point.

Another construction extending the Reshetikhin--Turaev TQFT to circles was given in the PhD thesis of Tsumura \cite{Tsumura}. Tsumura constructed a bicategory $\mathbf{Co}$ consisting of oriented 1-manifolds, \emph{connected} 2-dimensional oriented bordisms, and 3-dimensional oriented bordisms with corners. Explicitly, the 1-morphisms of $\mathbf{Co}$ are equipped with an identification to a ``standard'' connected surface with boundary. Each of these standard surfaces is in turn associated with a standard ribbon graph consisting of a single vertex, edges from the vertex to itself indexed by the genus of the surface, and edges with a free endpoint indexed by the boundary components. Then, a functor $\mathbf{Co}\rightarrow \Vect$ may be constructed in a manner analogous to \cite{RT,Turaev-book}. The main additional work that has to be done from here is to check the compatibility with composition of 1-morphisms and horizontal composition of 2-morphisms. As all 1-morphisms of $\mathbf{Co}$ are connected, one only has to consider a certain ``standard'' composition of standard surfaces. For compositions of these form, Tsumura constructed a composition law on the associated ribbon graphs, effectively contracting the vertices of the two graphs along a preferred edge to obtain another standard ribbon graph with a single vertex. Having established this, the compatibility with the horizontal composition of 2-morphisms could be established by generalising methods from \cite{Turaev-book}. Moreover, this is by construction compatible with the approach of Reshetikhin and Turaev. Thus, the functor $\mathbf{Co}\rightarrow \Vect$ extending the Reshetikhin--Turaev TQFT is defined.

However, it should be noted that this approach does not readily generalise to give a symmetric monoidal functor $\Bord{sig}\rightarrow\Vect$. In particular, it is possible for two 1-morphisms in $\Bord{sig}$ which each consist of many connected components to compose to a 1-morphism with a single connected component. Then, there would not be a distinguished way to combine the various ribbon graphs associated to each component into one with a single vertex. Moreover, given three composable 1-morphisms, any such approach could give very different resulting ribbon graphs when the order of composition of the 1-morphisms is changed.

To the author's knowledge, the first complete construction extending the Reshetkhin--Turaev TQFT to circles is in the PhD thesis of De Renzi \cite{dr-thesis}. (De Renzi in turn credits \cite{BDSV4} for providing a complete construction, though at the time, Theorem \ref{conj:pres-or} was far from being proven.) This is done via a bicategorical analogue of the universal construction reformulation of the Reshetikhin--Turaev TQFT given in \cite{BHMV}. Using this approach, the functoriality is automatic by definition, and most of the work is in showing that this construction is also compatible with the monoidal structure (disjoint unions). This approach was then used to construct non-semisimple extended TQFTs, generalising the construction of Kerler and Lyubashenko \cite{KL} which only dealt with connected surfaces. However, the way this construction is defined makes computing the images of 1- and 2-morphisms difficult. The images of the 1-morphisms in De Renzi's construction are, by definition, quotients of large vector spaces by large vector subspaces, and analogously, the images of 2-morphisms are quotients of large linear categories by large ideals. De Renzi provides computations for some simple 1- and 2-morphisms, but even those are rather involved.

Our approach here is motivated by the observation that the approaches of \cite{BDSV4} and \cite{Tsumura} are complementary: In \cite{BDSV4}, 1-morphisms are sent to vector spaces of ``internal string diagrams'', which are effectively the vector spaces of possible vertex labels of certain ribbon graphs (which are allowed to have more than one vertex). When set up in this way, the compatibility with composition of 1-morphisms (in the general case, without any connectivity assumptions) is clear. This being well-defined is however dependent on some difficult Cerf theory (in particular, Theorem \ref{conj:pres-or}). On the other hand, in \cite{Tsumura}, the well-definedness follows directly from the fact that the invariants of embedded ribbon graphs from \cite{RT-ribbon,RT} are well-defined. As such, our construction arises from combining the approach of \cite{BDSV4} at the object and 1-morphism level with the approach of \cite{Tsumura} (which in turn is heavily based on \cite{RT,Turaev-book}) at the 2-morphism level.

In this way, the images of 2-morphisms are easy to compute (requiring only a surgery presentation of the relevant 3-manifold), and those of the generating 2-morphisms are even easier. This allows us to make a direct comparison to the generators-and-relations approach of \cite{BDSV4}. On the other hand, this is also by construction easily comparable to the original formulation of Reshetikhin and Turaev \cite{RT}. Hence, it provides a way to bridge the gap between the generators-and-relations approach and the original Reshetikhin--Turaev formulation, the consequences of which we describe in the next subsection.

\subsection{The classification of once-extended 3-dimensional TQFTs} \label{subsection:BDSV}

The main work in the classification of once-extended 3-dimensional TQFTs is the project of Bartlett, Douglas, Schommer-Pries and Vicary. This was originally envisioned to consist of four papers \cite{BDSV1,BDSV2,BDSV3,BDSV4}, of which the second and the fourth have been published. The intended main results of the first will be proven in upcoming work of Bartlett, Douglas and Sytilidis \cite{pres-bord} by putting together results of Haïoun \cite{haioun} and Sytilidis \cite{filippos-thesis}, and the third paper, which is largely complete, will be forthcoming thereafter. In this subsection, we discuss this approach of classifying once-extended TQFTs by modular tensor categories and the state of the literature surrounding this.

There are several classification results of a similar flavour, for various 3-dimensional bordism bicategories. The exact results vary depend on the exact 3-dimensional bordism bicategory $\mathbf{Bord}$, but they are all statements of the form:
\begin{quote}
    Symmetric monoidal functors $\mathbf{Bord}\rightarrow\Vect$ up to equivalence are in bijection with modular tensor categories (possibly with some conditions or extra data) up to isomorphism.
\end{quote}
Following the methods of Bartlett--Douglas--Schommer-Pries--Vicary, each of these classification results is proven by constructing a finite presentation $\mathcal{X}$ of the relevant bordism bicategory. Then, for the symmetric monoidal bicategory $\mathbf{F}(\mathcal{X})$ generated by this presentation, symmetric monoidal functors $\mathbf{F}(\mathcal{X})\rightarrow\Vect$ and transformations between them may be defined by writing down a finite amount of data and performing some finite number of checks. On the other hand, the presentations are constructed in such a way so that there is a natural ``topological realisation'' functor $|-|:\mathbf{F}(\mathcal{X})\rightarrow\mathbf{Bord}$. By showing that $|-|$ is an equivalence of symmetric monoidal bicategories, the classification of linear representations of $\mathbf{Bord}$ can then be bootstrapped onto the corresponding classification result for $\mathbf{F}(\mathcal{X})$. In short, the approach boils down to the following diagram:
\[\begin{tikzcd}
	{\mathrm{Rep}(\mathbf{Bord})} && {\mathrm{Rep}(\mathbf{F}(\mathcal{X}))} && {\mathbf{MTC}}
	\arrow[tail reversed, from=1-1, to=1-3]
	\arrow[tail reversed, from=1-3, to=1-5]
\end{tikzcd}\]

For the oriented, signature, componentwise signature and $p_1$ bordism bicategories, the bijection on the right is proven purely algebraically in \cite{BDSV4}, using some computational results from \cite{BDSV2}. In separate work of the author \cite{partA}, we extend this approach to also handle the half-signature and componentwise half-signature bordism bicategories.

The other arrow has proven to be more difficult. The originally envisioned approach was to adapt the Cerf-theoretic methods of \cite{csp-phd} to prove that each presentation $\mathcal{X}$ was indeed a presentation of the corresponding bordism bicategory. Specifically, \cite{BDSV1} along with some algebraic reductions from \cite{BDSV2} was intended to prove this for the oriented bordism bicategory; we record this statement as Theorem \ref{conj:pres-or}. Using this, the analogous results for the signature, componentwise signature and $p_1$ bordism bicategories are shown in \cite{BDSV3}. In separate work of the author \cite{partA}, we also deduce the corresponding results for the half-signature and componentwise half-signature bordism bicategories from Theorem \ref{conj:pres-or}. Much of the Cerf-theoretic work towards proving Theorem \ref{conj:pres-or} has been done in recent work of Haïoun \cite{haioun} and Sytilidis \cite{filippos-thesis}. Upcoming work of Bartlett, Douglas and Sytilidis \cite{pres-bord} will finish the proof of Theorem \ref{conj:pres-or}, thus completing the classifications of the linear representations of the bordism bicategories mentioned above.

Our construction of the extended Reshetikhin--Turaev TQFT allows us to replace the left arrow in the diagram above, and so remove the dependence on Theorem \ref{conj:pres-or}. In comparison to showing that $|-|$ is an equivalence, it is far easier to show that it is essentially surjective, essentially full and full. This is a fact that has been evident to the experts in the area; we record a proof of this in the case of the oriented bordism bicategory in Proposition \ref{prop:modo-full} for the sake of completeness. (The analogous results for the other bordism bicategories can then be deduced from this one.) The bulk of Theorem \ref{conj:pres-or}, for which the difficult Cerf-theoretic work is required, would then be to show that $|-|$ is also faithful, i.e. that the set of relations of $\mathcal{X}$ is sufficiently large. Nevertheless, in weakening Theorem \ref{conj:pres-or} to Proposition \ref{prop:modo-full}, it follows formally (as proven in Proposition \ref{prop:pb-inj}) that the left arrow in the diagram above is replaced by an injection. In other words, we now have:
\[\begin{tikzcd}
	{\mathrm{Rep}(\mathbf{Bord})} && {\mathrm{Rep}(\mathbf{F}(\mathcal{X}))} && {\mathbf{MTC}}
	\arrow[hook, from=1-1, to=1-3]
	\arrow[tail reversed, from=1-3, to=1-5]
\end{tikzcd}\]
In order to complete the proof of the classification of linear representations of $\mathbf{Bord}$ from this picture, it then suffices to produce a construction which takes as input a modular tensor category (possibly with extra data or conditions) and outputs a symmetric monoidal functor $\mathbf{Bord}\rightarrow\Vect$, producing an arrow that goes from $\mathbf{MTC}$ to $\mathrm{Rep}(\mathbf{Bord})$ directly. One would then have to check that this construction is compatible with the generators-and-relations machinery of the right arrow.

This is where the extended Reshetikhin--Turaev TQFT fits in. As mentioned in the previous subsection, there is already an existing construction of this that is independent of generators and relations: the universal construction of De Renzi \cite{dr-thesis}. However, this is set up in an abstract way which makes the computation of the images of specific 1-morphisms and 2-morphisms rather involved. Through some non-trivial calculations, De Renzi computed the images of the pants and cup 1-morphisms (and the copants and cap are analogous), as well as 2-morphisms which, in the language of \textsection\ref{subsection:pres}, correspond to the 2-morphisms $\nu^\dag$ and $\id_{\tikztinycap}\star\mu\star\id_{\tikztinycup}$. Effectively, these computations show that applying the machinery of \cite{BDSV4} to De Renzi's construction would yield a modular tensor category which is equivalent as a monoidal linear category to the modular tensor category one started with (plus some computations for two more generating 2-morphisms). One expects that it should be possible to compute the images of the relevant generating 2-morphisms in a similar fashion, but we instead take a different approach.

In this paper, we present a direct construction of the extended Reshetikhin--Turaev TQFT in the same spirit as the original approach of Reshetikhin and Turaev \cite{RT} which produces a symmetric monoidal functor $\mathbf{Bord}\rightarrow\Vect$ from the data of a modular tensor category. From our construction, it is then not hard to compare this functor to the functor $\mathbf{F}(\mathcal{X})\rightarrow\Vect$ arising from the generators-and-relations approach. This is done in Proposition \ref{prop:mtc-rep-mtc}, allowing us to recover the classification results independently of the Cerf-theoretic result Theorem \ref{conj:pres-or}.

Viewing the Cerf-theoretic part of Theorem \ref{conj:pres-or} as a statement that the set of relations of $\mathcal{X}$ is large enough, the effective contribution of this result to the overall programme is that it allows one to show that the symmetric monoidal functor $\mathbf{Bord}\rightarrow\Vect$ constructed via the generators-and-relations approach is indeed well-defined. (It is, of course, also of independent interest as it allows one to construction symmetric monoidal functors valued in a general symmetric monoidal bicategory.) In our approach, this well-definedness comes directly from the well-definedness of the Reshetikhin--Turaev invariants of ribbon graphs \cite{RT-ribbon,RT}, thus allowing one to bypass the difficult Cerf-theoretic portion of the overall programme.

As a result, we obtain in \textsection\ref{subsection:classify} proofs of the following classification results, independently of Theorem \ref{conj:pres-or}.
\begin{thmA}[{\cite{BDSV4,partA,pres-bord}}] \label{thmA:classify}
    The following classification results hold.
    \begin{enumerate}
        \item Linear representations of the half-signature bordism bicategory $\Bord{sig/2}$ are classified by finite direct sums of modular tensor categories whose anomalies are equal.
        \item Linear representations of the oriented bordism bicategory $\Bord{or}$ are classified by finite direct sums of modular tensor categories with trivial anomaly.
        \item Linear representations of the signature bordism bicategory $\Bord{sig}$ are classified by finite direct sums of modular tensor categories whose anomalies are equal, along with a single choice of square root of this anomaly.
        \item Linear representations of the $p_1$ bordism bicategory $\mathbf{Bord}_{1,2,3}^{p_1}$ are classified by finite direct sums of modular tensor categories whose anomalies are equal, along with a single choice of $6^\mathrm{th}$ root of this anomaly.
        \item Linear representations of the componentwise half-signature bordism bicategory \linebreak $\Bord{csig/2}$ are classified by finite direct sums of modular tensor categories.
        \item Linear representations of the componentwise signature bordism bicategory $\Bord{csig}$ are classified by finite direct sums of modular tensor categories, along with a choice of square root of anomaly of each direct summand.
    \end{enumerate}
\end{thmA}

Implicit in our proofs of all of these is that all such linear representations may be written as an extended Reshetikhin--Turaev TQFT.
\begin{thmA} \label{thmA:all-rt}
    For each of the bordism bicategories in Theorem \ref{thmA:classify}, all of its linear representations arise as Reshetikhin--Turaev theories.
\end{thmA}
\begin{rk}
    In the case of $\Bord{sig/2}$, this is essentially the content of Proposition \ref{prop:mtc-rep-mtc}. The analogous results for the other bicategories may be deduced from this case.
\end{rk}

\subsection{Overview of this paper} The contents of this paper are as follows. In \textsection\ref{section:background}, we recall from \cite{partA} the definitions of the signature and half-signature bordism bicategories, and present models for these which are better suited to the extended Reshetikhin--Turaev construction. In \textsection\ref{section:rt}, we first go through some preliminaries regarding the invariants of various notions of ribbon graphs; this also serves the function of fixing the notation that will be used in subsequent sections. Then, we write down the construction for $RT_{\mathcal{C},\sqrt{D}}:\Bord{sig}\rightarrow\Vect$. In \textsection\ref{section:rt-functor}, we show that this construction indeed defines a symmetric monoidal functor. Along the way, we also check that this construction is indeed compatible with the Reshetikhin--Turaev TQFT $rt_{\mathcal{C},\sqrt{D}}$. At the end of the section, we check that upon restriction to $\Bord{sig/2}$, $RT_{\mathcal{C},\sqrt{D}}$ and $RT_{\mathcal{C},-\sqrt{D}}$ agree, and hence it makes sense to write $RT_\mathcal{C}:\Bord{sig/2}\rightarrow\Vect$. Finally, in \textsection\ref{section:classify}, we check that this $RT_\mathcal{C}$ is compatible with the generators-and-relations approach of \cite{partA}, and so obtain a proof of Theorem \ref{thmA:classify} that is independent of the Cerf-theoretic Theorem \ref{conj:pres-or}.

\subsection{Notations and conventions} Throughout this paper, $k$ denotes an algebraically closed field of arbitrary characteristic. All notions of linearity are over $k$. This is in particular allows us to use results from \cite{BK} on modular tensor categories over algebraically closed fields.

\subsection{Acknowledgements}

We are grateful to André Henriques for proposing this project, as well as continued guidance and support throughout this process. Thanks are also due to Christopher Douglas for many helpful discussions as well as his suggestions on the structure and framing of this paper. We also thank Gheehyun Nahm for some helpful discussions.

\section{The source and target bicategories} \label{section:background}

The extended Reshetikhin--Turaev TQFT defined here will be a symmetric monoidal functor from a suitable central extension of the oriented bordism bicategory to a suitable bicategorical analogue of the category of vector spaces. In this section, we specify our models for the source and target bicategories.

\subsection{The half-signature bordism bicategory}

The source category will be the half-signature bordism bicategory, as defined in \cite{partA}. We will first recall from there the construction of the signature and half-signature bordism bicategories, then proceed to present an alternate model which is more suited to the extended Reshetikhin--Turaev TQFT. This will be done in some detail, so as to fix the notation used in the rest of the paper. Much of this subsection is summarised from \cite[\textsection 3]{partA}, with the diagrams reproduced for clarity.

\begin{defn}
    Let $M, M'$ be closed oriented $(d-1)$-dimensional manifolds. An \emph{oriented bordism} from $M$ to $M'$ is an oriented $d$-manifold $X$ with an orientation-preserving diffeomorphism $\partial X \cong (-M) \sqcup M'$, where $-M$ is $M$ with the opposite orientation.

    Given such a bordism $X$, write $\parin X = M$ and $\parout X = M'$.
\end{defn}
\begin{defn}\label{defn:bordism-corners}
    Let $X, X'$ be oriented bordisms between closed oriented $(d-2)$-dimensional manifolds $M$ and $M'$. An \emph{oriented bordism with corners} from $X$ to $X'$ is an oriented $d$-dimensional manifold with corners $Y$ with orientation-preserving diffeomorphism $$\partial Y \cong \pa{(-X) \sqcup X'} \cup_{(-M) \sqcup M \sqcup (-M') \sqcup M'} \pa{I \times (-M) \sqcup I \times M'}.$$

    This has a \emph{horizontal boundary} which consists of the manifolds with boundary $\parin Y = X$ and $\parout Y = X'$, as well as a \emph{vertical boundary} which consists of the manifolds with boundary $\parIn Y = I \times M$ and $\parOut Y = I \times M'$.
\end{defn}

In the definition of the signature bordism bicategory given in \cite{BDSV3}, the horizontal and vertical composition of 2-morphisms is defined in terms of gluing $4$-manifolds along $3$-dimensional submanifolds of their boundary. This is omitted here; we will instead present in Definition \ref{defn:comp} an alternative formulation using Wall's invariant (Definition \ref{defn:Wall}).

\begin{defn}\label{defn:bordsig}
    The \emph{signature bordism bicategory} $\Bord{sig}$ is the symmetric monoidal bicategory with:
    \begin{itemize}
        \item Objects: Pairs $(S, D_S)$ where $S$ is a closed oriented 1-dimensional manifold and $D_S$ a 2-manifold with boundary $\partial D_S = S$. $D_S$ is oriented such that the orientation induced on $\partial D_S$ is the same as the orientation on $S$.
        \item 1-morphisms: Pairs $(\Sigma, H_\Sigma)$ where $\Sigma$ is a oriented 2-dimensional bordism, and $H_\Sigma$ is a 3-manifold with boundary $$\partial H_{\Sigma} = \ol{\Sigma}:= D_{\parin \Sigma} \cup_{\parin \Sigma} \Sigma \cup_{\parout \Sigma} \pa{-D_{\parout \Sigma}}.$$ $H_\Sigma$ is oriented such that the orientation induced on $\partial H_\Sigma$ is the same as the orientation on $\ol{\Sigma}$.
        \item 2-morphisms: Equivalence classes of pairs $(M,n)$ where $M$ is an oriented bordism with corners, and $n$ is an integer. Two pairs $(M,n)$ and $(M',n')$ are taken up to equivalent if and only if $M,M'$ are diffeomorphic relative to the boundary and $n=n'$.
    \end{itemize}
    The composition of 1-morphisms $(\Sigma, H_{\Sigma})$ and $(\Sigma', H_{\Sigma'})$ with  $\parout \Sigma = \parin \Sigma' = S$ is given by $$(\Sigma \cup_S \Sigma', H_{\Sigma} \cup_{D_S} H_{\Sigma'}).$$

    The vertical and horizontal composition of 2-morphisms are as given in Definition \ref{defn:comp}.

    The monoidal structure is given by taking disjoint unions.
\end{defn}

Given a 2-morphism $(M,n)$, the extra data associated to the objects and 1-morphisms allows us to construct a closed 3-manifold $\ol{M}$, which will be essential in the definition of the half-signature bordism bicategory as well as our construction of the extended Reshetikhin--Turaev TQFT.

\begin{defn} \label{defn:olm}
    Let $(M,n)$ be a 2-morphism in $\Bord{sig}$. Define $\ol{M}$ to be the closed 3-manifold obtained from $M$ by gluing $H_{\parin M}$ and $-H_{\parout M}$ onto its horizontal boundaries and $V_{\parIn M}:= I\times D_{\parin\parin M}$ and $-V_{\parOut M}:= -(I\times D_{\parout\parin M})$ onto its vertical boundaries. Figure \ref{fig:bordsig-2mor} shows a schematic for $\ol{M}$.\begin{figure}[hbt!]
        \includegraphics[scale=0.5]{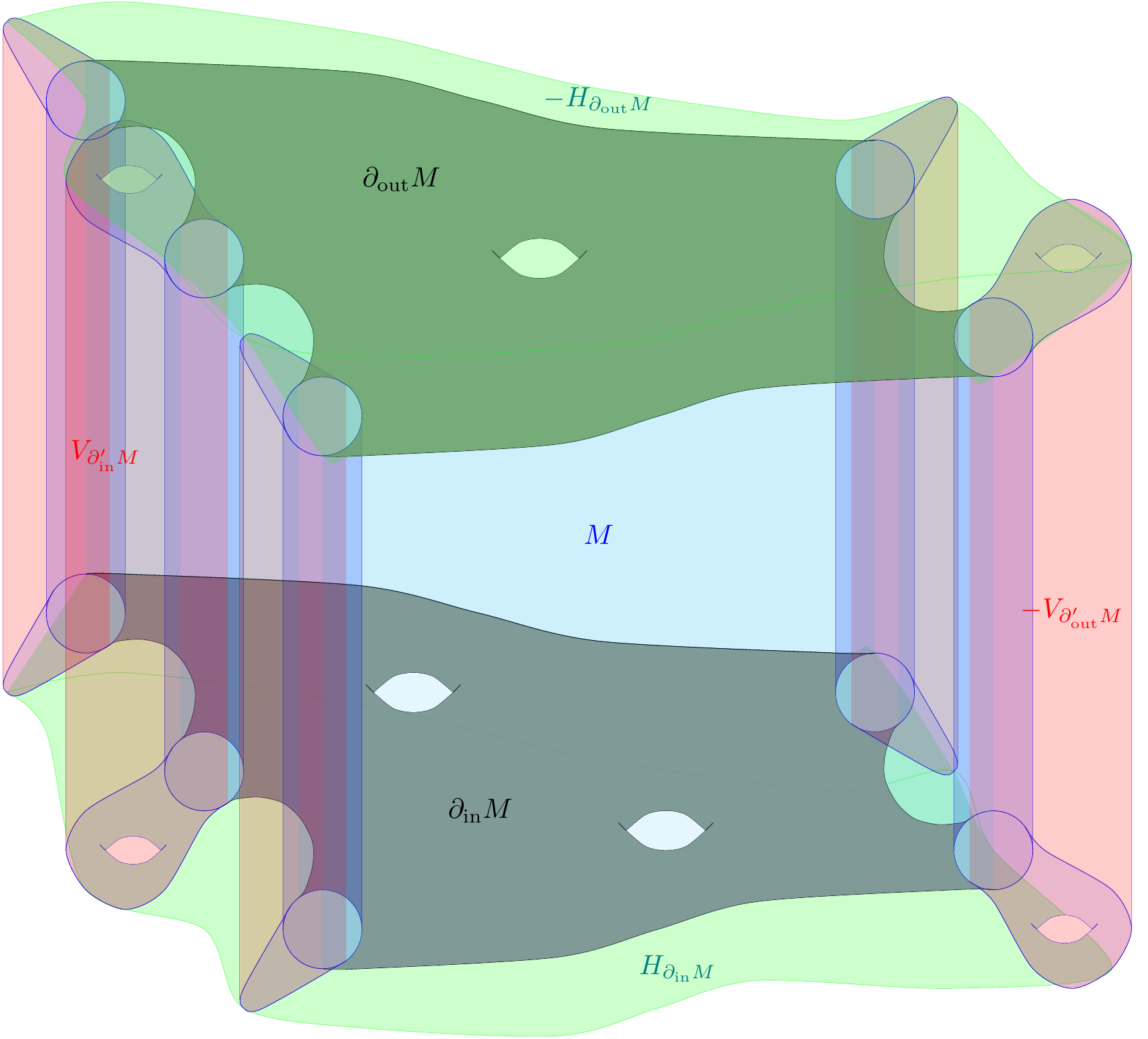}
        \caption{Schematic for $\ol{M}$.} \label{fig:bordsig-2mor}
    \end{figure}
\end{defn}
\begin{rk}
    The integer $n$ in Definition \ref{defn:bordsig} may be thought of as the signature of a 4-manifold $X$ with boundary $\ol{M}$. We refer the reader to \cite[\textsection 3]{partA} for more details.
\end{rk}

Our definition of the signature and half-signature bordism bicategories uses an invariant of Wall \cite{Wall}, which he used to study the signature of the $4k$-manifold created by gluing two $4k$-manifolds along $4k-1$-dimensional submanifolds of their boundary.

\begin{defn}\label{defn:Wall}
    Let $(V,\omega)$ be a finite-dimensional symplectic vector space and $A,B,C$ three Lagrangian subspaces. Then, we may define a bilinear form $\langle\cdot,\cdot\rangle$ on the vector space $\frac{A\cup(B+C)}{(A\cup B) + (A\cup C)}$ as follows: for $a,a'\in A\cup (B+C)$, let $a+b+c=a'+b'+c'=0$, where $b,b'\in B$ and $c,c'\in C$. Then, define $\langle a,a' \rangle := \omega(a,b')$.

    \emph{Wall's invariant}, also referred to in the literature as the \emph{Maslov index}, is the signature of $\langle \cdot,\cdot\rangle$, denoted $\sigma(V;A,B,C)$.
\end{defn}
This invariant has the following topological significance:
\begin{thm}[\cite{Wall}]\label{thm:wall}
    Let $M_-,H,M_+$ be $(4k-1)$-manifolds with common boundary $-\partial M_- = \partial H = \partial M_+ = \Sigma$ and $X_-, X_+$ be $4k$-manifolds with $\partial X_+ = M_+ \cup_\Sigma (-H), \partial X_- = X_0 \cup_\Sigma M_-$, and $X_0 = X_+ \cup_{H} X_-$, as in Figure \ref{fig:wall-diagram}. Then the signatures of $X_0$ and $X_\pm$ are related by $$\sigma(X_0) = \sigma(X_-) + \sigma(X_+) - \sigma(V; A, B, C)$$
    where $V = H_{2k-1}(\Sigma; \mathbb{R})$ and $A, B, C$ are the Lagrangian subspaces of $V$ associated to the inclusions $\Sigma \hookrightarrow M_-$, $\Sigma \hookrightarrow H$, $\Sigma \hookrightarrow M_+$ respectively.
    \begin{figure}[hbt!]
        \includegraphics[scale=0.5]{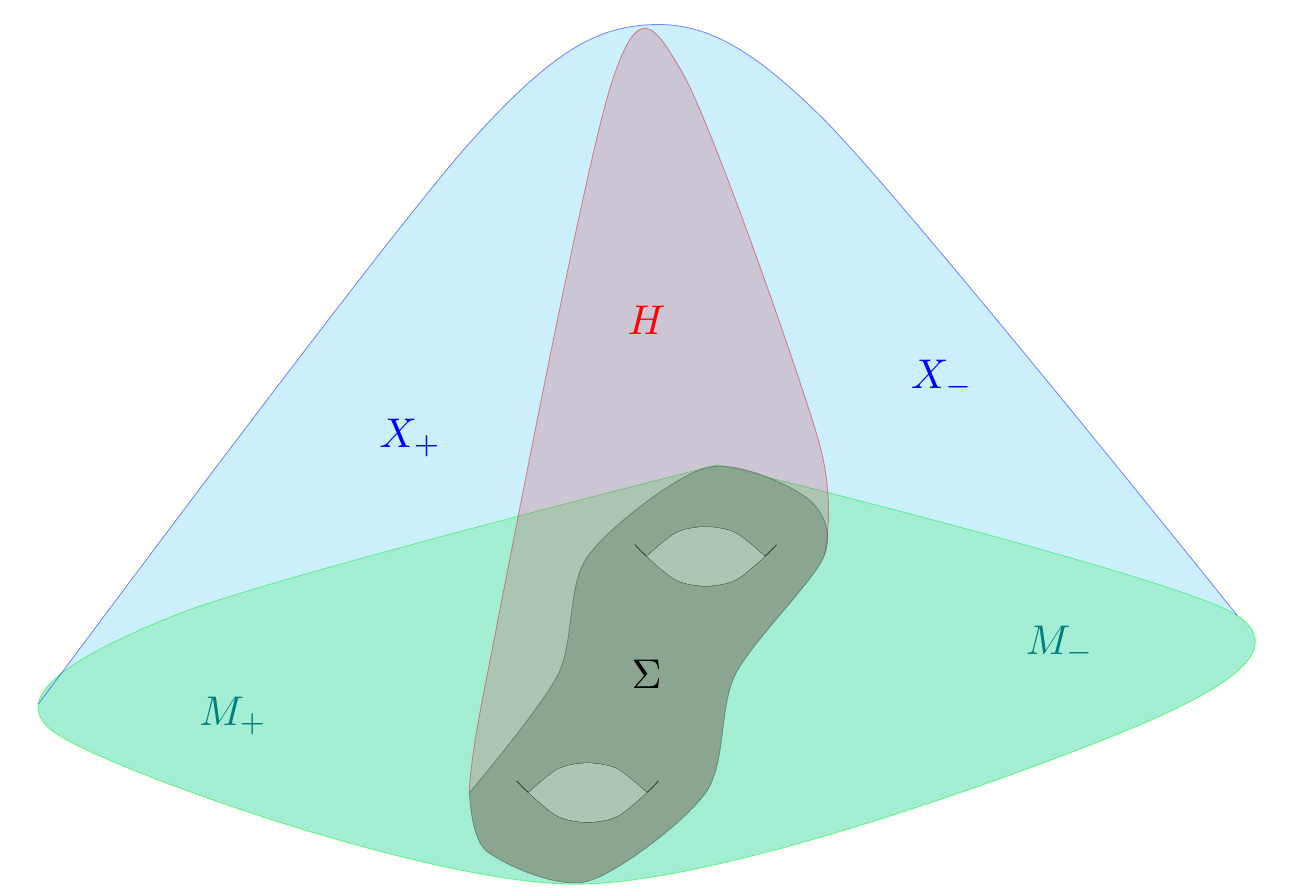}
        \caption{Gluing $X_\pm$ to form $X_0$ in Wall's Theorem.} \label{fig:wall-diagram}
    \end{figure}
\end{thm}

We are now able to write down the vertical and horizontal composition of 2-morphisms in $\Bord{sig}$.
\begin{defn} \label{defn:comp}
    The vertical and horizontal composition of 2-morphisms in $\Bord{sig}$ is given as follows.

    For two 2-morphisms $(M,n)$ and $(M',n')$ with, respectively, target and source $(\Sigma,H_\Sigma)$, their vertical composition is $$\pa{M \cup_\Sigma M', n + n' - \sigma(V; A, B, C)}$$ where
    \begin{equation} \label{eq:csig}
        \begin{split}
            V &= H_1(\ol{\Sigma}; \mathbb{R}) \\
            A &= \ker(H_1(\ol{\Sigma}; \mathbb{R}) \rightarrow H_1(\ol{M'}\setminus \mathring{H}_\Sigma; \mathbb{R})) \\
            B &= \ker(H_1(\ol{\Sigma}; \mathbb{R}) \rightarrow H_1(H_\Sigma; \mathbb{R})) \\
            C &= \ker(H_1(\ol{\Sigma}; \mathbb{R}) \rightarrow H_1(\ol{M}\setminus \mathring{H}_\Sigma; \mathbb{R}))
        \end{split}
    \end{equation}
    
    For two 2-morphisms $(M,n)$ and $(M',n')$ with $\parOut M \cong \parIn M' \cong \Sigma$, let $$\ol{\Sigma} = D_{\parout\parin M} \cup_{\parout\parin M} \Sigma \cup_{\parout\parout M} (-D_{\parout\parout M}).$$ Then, their horizontal composition is $$\pa{M \cup_\Sigma M', n + n' - \sigma(V; A, B, C)}$$ where
    \begin{equation} \label{eq:ctildesig}
        \begin{split}
            V &= H_1(\ol{\Sigma}; \mathbb{R}) \\
            A &= \ker(H_1(\ol{\Sigma}; \mathbb{R}) \rightarrow H_1(\ol{M'}\setminus \mathring{V}_\Sigma; \mathbb{R})) \\
            B &= \ker(H_1(\ol{\Sigma}; \mathbb{R}) \rightarrow H_1(V_\Sigma; \mathbb{R})) \\
            C &= \ker(H_1(\ol{\Sigma}; \mathbb{R}) \rightarrow H_1(\ol{M}\setminus \mathring{V}_\Sigma; \mathbb{R}))
        \end{split}
    \end{equation}
\end{defn}

Having defined $\Bord{sig}$, we may now define the half-signature bordism bicategory as a symmetric monoidal subbicategory.
\begin{defn}[{\cite{partA}}] \label{defn:sig/2}
    The \emph{half-signature bordism bicategory} $\Bord{sig/2}$ is the symmetric monoidal bicategory with 
    \begin{itemize}
        \item Objects and 1-morphism: as in $\Bord{sig}$ in Definition \ref{defn:bordsig}.
        \item 2-morphisms: Equivalence classes of pairs $(M,n)$ such that $n\in\mathbb{Z}$ has the same parity as 
        \begin{equation*}
            m(M) := b_1(\ol{M}) + b_0(\ol{M}) + \frac12 b_1(\ol{\parout M}) + b_0(\ol{\parout M}),
        \end{equation*}
        where $b_i$ denotes the $i^\text{th}$ Betti number, $b_i(X) := \dim H_i(X;\mathbb{R})$.
    \end{itemize}
    The composition of 1-morphisms, the vertical and horizontal composition of 2-morphisns, and the monoidal structure are as in in $\Bord{sig}$ in Definitions \ref{defn:bordsig} and \ref{defn:comp}.
\end{defn}

Now, we shall construct alternate models $\Bord{sig,dec}$ and $\Bord{sig/2,dec}$ for the signature and half-signature bordism bicategories. In these model, the objects are equipped with extra data which identifies them with a disjoint union of a certain generating object. Correspondingly, 1-morphisms are equipped with extra data identifying them with a composition of disjoint unions of certain generating 1-morphisms.

Specifically, let $(\mathbb{S},\mathbb{D})$ be the standard circle with the standard bounding disc. This is drawn diagrammatically as:
\begin{equation*}
    \begin{tikzpicture}
            \node[Cyl, top, height scale=0]  at (0,0) {};
    \end{tikzpicture}
\end{equation*} 

Moreover, consider the standard pants, copants, cup and cap, which are 1-morphisms between $(\emptyset,\emptyset)$, $(\mathbb{S},\mathbb{D})$ and $(\mathbb{S}\sqcup \mathbb{S}, \mathbb{D} \sqcup \mathbb{D})$:
\begin{equation*}
    \begin{tikzpicture}
        \node[Pants, top, bot] (A) at (0,0) {};
        \node[Cup, top] (C) at (4,0.1) {};
        \node[Copants, top, bot] (B) at (2,0) {};
        \node[Cap, bot] (D) at (6,-0.1) {};
    \end{tikzpicture}
\end{equation*}
These are all given the handlebody $H_\Sigma = D^3$, filling the ``insides'' of the pictures drawn above. Viewing these $D^3$s as being embedded in $\mathbb{R}^3$ in the obvious way, we may orient them according to the right-handed orientation on $\mathbb{R}^3$, which in turn induces an orientation on the respective surfaces.

Now, we consider a model of $\Bord{sig}$ whose objects are exactly disjoint unions of the standard $(\mathbb{S},\mathbb{D})$ and whose 1-morphisms are given by formal compositions of the 1-morphisms drawn above.

\begin{defn}\label{defn:bordsig-new}
    The \emph{signature bordism bicategory with decompositions} $\Bord{sig,dec}$ is the symmetric monoidal bicategory with
    \begin{itemize}
        \item Objects: Disjoint unions of copies of $(\mathbb{S},\mathbb{D})$.
        \item 1-morphisms: Formal compositions of disjoint unions of the standard pants, copants, cup and cap. If the underlying surface arising from such a composition is $\Sigma$, there is a natural choice of handlebody $H_\Sigma$ with $\partial H_\Sigma = \ol{\Sigma}$ that comes from gluing the handlebodies associated to each of the generating components.
        \item 2-morphisms: Equivalence classes of pairs $(M,n)$, where $(M,n)$ and $(M',n')$ are equivalent if and only if $M,M'$ are diffeomorphic relative to the boundary and $n=n'\in\mathbb{Z}$.
    \end{itemize}

    Horizontal and vertical composition of 2-morphisms is as in Definition \ref{defn:comp}. Moreover, note that for horizontal composition, the surface $\ol{\Sigma}$ in (\ref{eq:ctildesig}) always has genus $0$, and so the Wall's invariant term is always $0$.
\end{defn}
\begin{rk}
    For any of the presentations in \cite[\textsection 6]{partA} (most of which are reproduced from \cite{BDSV4}), one may consider the free quasistrict symmetric monoidal bicategory of the presentation, as defined in \cite[\textsection 2.12]{csp-phd} and \cite[Definition 62]{BDSV2}. The objects and 1-morphisms of the resulting symmetric monoidal bicategory will exactly be in bijection with those of $\Bord{sig,dec}$. In particular, the definitions of these presentations contain many examples of what 1-morphisms in $\Bord{sig,dec}$ could look like.
\end{rk}

\begin{defn} \label{defn:bordsigh-new}
    The \emph{half-signature bordism bicategory with decompositions} $\Bord{sig/2,dec}$ is the symmetric monoidal bicategory with the same objects and 1-morphisms as \linebreak $\Bord{sig,dec}$, but with the 2-morphisms being equivalence classes of pairs $(M,n)$ such that $n \equiv m(M) \pmod2$, where $m(M)$ is as in Definition \ref{defn:sig/2}. This is an index two subbicategory of $\Bord{sig}$.
\end{defn}
\begin{defn} \label{defn:bordor-new}
    The \emph{oriented bordism bicategory with decompositions} $\Bord{or,dec}$ is the symmetric monoidal bicategory with the same objects and 1-morphisms as $\Bord{sig,dec}$, but the 2-morphisms being equivalence classes of bordisms with corners $M$.
\end{defn}

By construction, we have equivalences of symmetric monoidal bicategories $\Bord{sig,dec}\rightarrow\Bord{sig}$, $\Bord{sig/2,dec}\rightarrow\Bord{sig/2}$ and $\Bord{or,dec}\rightarrow\Bord{or,exp}$ given by forgetting the extra structure on the objects and 1-morphisms. $\Bord{sig,dec}$ is a symmetric monoidal extension (in the sense of \cite[\textsection 2]{partA}) of $\Bord{or,dec}$ by $\mathbb{Z}$ with index $2$ subextension $\Bord{sig/2,dec}$.

\subsection{The target bicategory}

The target bicategory will be Kapranov--Voevodsky's \cite{KV} bicategory of 2-vector spaces $\KV$. While there are multiple possible bicategorical analogues of $\vect$, $\KV$ naturally admits a fully faithful functor into the others, as illustrated in \cite[Appendix~A]{BDSV4}.

\begin{defn}
    $\KV$ is the symmetric monoidal bicategory with 
    \begin{itemize}
        \item Objects: positive integers.
        \item 1-morphisms: matrices of finite-dimensional $k$-vector spaces.
        \item 2-morphisms: matrices of $k$-linear maps.
    \end{itemize}
    Composition of 1-morphisms and horizontal composition of 2-morphisms is given by matrix multiplication, with $\oplus$ and $\otimes$ playing the role of the usual addition and multiplication. Vertical composition of 2-morphisms is given by the entrywise composition.

    This has a symmetric monoidal structure given by multiplication on objects, and on the 1- and 2-morphisms, the pairwise tensor products of entries.
\end{defn}

This is equivalent to the bicategory of finite semisimple linear categories, with the integer $s$ corresponding to the linear category $\vect^{\oplus s}$. In particular, it admits a fully faithful functor into the other relevant bicategorical analogue of $\vect$:
\begin{defn}[{\cite[Definition 2.7]{BDSV4}}] \label{defn:rep}
    $\Vect$ is the symmetric monoidal bicategory with
    \begin{itemize}
        \item Objects: Cauchy-complete $k$-linear categories
        \item 1-morphisms: $k$-linear functors
        \item 2-morphisms: Natural transformations
    \end{itemize}
    The monoidal structure is given by the Cauchy completion of the enriched tensor product.
\end{defn}

\section{The extended Reshetikhin--Turaev TQFT} \label{section:rt}

In this section, we present our construction of the extended Reshetikhin--Turaev TQFT as a symmetric monoidal functor $RT_{\mathcal{C},\sqrt{D}}: \Bord{sig,dec} \rightarrow \KV$. In the next section, we will then check that this is compatible with vertical and horizontal composition, and that on the subbicategory $\Bord{sig/2,dec}$, the functors $RT_{\mathcal{C},\pm\sqrt{D}}$ agree, so we have a single well-defined functor $RT_\mathcal{C}:\Bord{sig/2,dec}\rightarrow\KV$.

Our construction here is based on combining the approach of Tsumura \cite{Tsumura} with that of Bartlett, Douglas, Schommer-Pries and Vicary \cite{BDSV4}. Specifically, our construction will resemble the ``internal string diagram'' formulation of the latter at the level of objects and 1-morphisms, and the former's approach involving Reshetkhin--Turaev ribbon invariants at the level of 2-morphisms. Before we proceed, we will first provide an extended overview of ribbon graphs and their invariants.

\subsection{Ribbon graphs and invariants}

We define the relevant notions of a ribbon graph, which is an essential ingredient in our definition of the extended Reshetikhin--Turaev TQFT. This comes in two different flavours: an abstract combinatorial one, and a topological one.

\begin{defn}
    An \emph{abstract ribbon graph} $\Gamma = (V,E,L)$ consists of 
    \begin{itemize}
        \item a set of \emph{vertices} $V$
        \item a set of directed \emph{edges} $E$, with \emph{head} and \emph{tail} maps $h,t:E\rightarrow V\sqcup\set{\infty}$
        \item a set of \emph{loops} $L$, which may be thought of as edges with no endpoints
        \item for each $v\in V$, a cyclic ordering of the set $E(v) := \set{E|h(E)=v}\sqcup\set{E|t(E)=v}$ of edges incident at $v$. $\set{E|h(E)=v}$ and $\set{E|t(E)=v}$ make up the \emph{incoming} and \emph{outgoing} edges at $v$, respectively.
        \item for each $v\in V$, a partition of $E(v)$ into two contiguous segments $E_s(v),E_t(v)$ with respect to the cyclic ordering. These form the \emph{source} and \emph{target} of $v$. If $E_s(v) = \emptyset$ (or respectively $E_t(v) = \emptyset$), then we additionally require a choice of breaking the cyclic ordering into a total ordering of $E_t(v)$ (or respectively $E_s(v)$). In all other cases, the cyclic ordering uniquely determines such a total ordering. We choose the total ordering on $E_t(v)$ such that it agrees with the cyclic ordering, and on $E_s(v)$ such that it is the reverse of the cyclic ordering.
    \end{itemize}
    Edges $e$ with $h(e) = \infty$ or $t(e) = \infty$ may be thought of as edges with a loose endpoint.

    Graphically, we may represent this as the disjoint union of a directed graph (possibly with parallel edges and edges whose head and tail are the same) and some loops, except each vertex is as depicted in Figure \ref{fig:ribbon-vertex}.

    \begin{figure}[hbt!]
        \begin{tikzpicture}
            \draw[draw=black, double=white, double distance=5pt] (-0.3,-0.3) to [out=down,in=up] (-1,-1);
            \draw[draw=black, double=white, double distance=5pt] (0.3,-0.3) to [out=down,in=up] (1,-1);
            \draw[draw=black, double=white, double distance=5pt] (0,-0.3) -- (0,-1);
            \draw[draw=black, double=white, double distance=5pt] (-0.2,0.3) to [out=up,in=down] (-0.7,1);
            \draw[draw=black, double=white, double distance=5pt] (0.2,0.3) to [out=up,in=down] (0.7,1);
            \draw[rounded corners] (-0.6,-0.3) rectangle (0.6,0.3);
        \end{tikzpicture}
        \caption{Vertex $v$ of a ribbon graph.} \label{fig:ribbon-vertex}
    \end{figure}
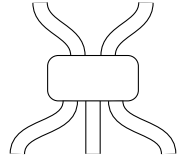

    In this diagram, the source of $v$ is taken to be the bottom side of the rectangle, while the target is the top side. The cyclic ordering is given by the locations of the endpoints of the edges on the perimeter of the rectangle in clockwise order. Thus, the orderings on $E_s(v)$ and $E_t(v)$ are given from left to right.

    When relevant, we will indicate the direction of an edge by drawing an arrow on the corresponding ribbon. In general, the direction of an edge is omitted when it may be arbitrarily chosen.
\end{defn}

We may view an abstract ribbon graph as a surface with boundary, with a partition into \emph{ribbons} (corresponding to the edges and loops) and \emph{coupons} (corresponding to the vertices). Explicitly, we may obtain this surface as the result of gluing some discs and bands: 
$$|\Gamma| := \pa{\bigsqcup_{v\in V}D^2 \sqcup \bigsqcup_{e \in E} [0,1]^2 \sqcup \bigsqcup_{\ell \in L} [0,1]\times S^1} \Big/ \sim$$
where the equivalence relation that encodes the gluing of the endpoints of edges to the discs as drawn in Figure \ref{fig:ribbon-vertex}. We may think of the coupons as 2-dimensional 0-handles, and the ribbons corresponding to the edges as 2-dimensional 1-handles. By smoothing out the corners of the gluing, we may view $|\Gamma|$ as a smooth surface.

The \emph{endpoints} of $|\Gamma|$ are the subspace comprising the loose endpoints of edges, i.e. $$\partial_\mathrm{end}|\Gamma| := \bigsqcup_{e \in h^{-1}(\infty)} [0,1] \times \set{0} \sqcup \bigsqcup_{e \in t^{-1}(\infty)} [0,1] \times \set{1}.$$ Let $E_{\mathrm{ext}}$ denote the \emph{external} edges whose head or tail (possibly both) is an endpoint, and $E_{\mathrm{int}}$ the \emph{internal} edges whose head and tail are both in $V$.

Given a 3-manifold with boundary $M$, we may now define the corresponding topological notion of a ribbon graph:

\begin{defn}
    An \emph{embedded ribbon graph} in a 3-manifold $M$ is an abstract ribbon graph $\Gamma$ with embedding $(|\Gamma|, \partial_\mathrm{end}|\Gamma|) \rightarrow (M,\partial M)$.
\end{defn}
\begin{rk}
    An embedding of the loops $L$ forms a framed link in $M$. We will thus also use $L$ to denote framed links in 3-manifolds.
\end{rk}

In particular, to each 1-morphism in $\Bord{sig,dec}$, we may associate a embedded ribbon graph. Recall that such 1-morphisms are formal compositions of the disjoint unions of the standard pants, copants, cup and cap, and so have a standard choice of associated handlebody. Let $(\Sigma,H_{\Sigma})$ denote the underlying surface and handlebody arising as a result of this composition. We shall construct an embedded ribbon graph in $H_{\Sigma}$.

\begin{eg}\label{eg:internal-graph}
    Each 1-morphism of $\Bord{sig,dec}$ comes with a decomposition into the standard generating 1-morphisms as well as the standard cylinder (which is the identity 1-morphism of the standard disc). To each of these pieces, we may assign an embedded ribbon graph in the associated handlebody, as drawn below. All edges in Figure \ref{fig:assoc-graph} are directed from bottom to top.

    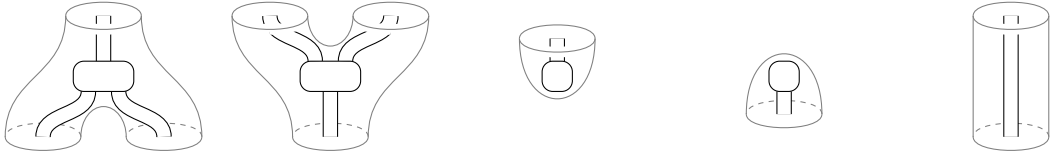
\begin{figure}[hbt!]
        \begin{tikzpicture}
            \begin{scope}[internal ribbon scope]
                \node[Pants, bot, top, scale=2, grey] (A) at (0,0) {};
                \draw[draw=black, double=white, double distance=5pt] (0,0.2) -- (A.belt);
                \draw[draw=black, double=white, double distance=5pt] (-0.2,-0.2) to [out=down,in=up] (A.leftleg);
                \draw[draw=black, double=white, double distance=5pt] (0.2,-0.2) to [out=down,in=up] (A.rightleg);
            \end{scope}
            \draw[rounded corners, draw=black] (-0.4,-0.2) rectangle (0.4,0.2);

            \begin{scope}[internal ribbon scope]
                \node[Copants, bot, top, scale=2, grey] (B) at (3,0) {};
                \draw[draw=black, double=white, double distance=5pt] (3,-0.2) -- (B.belt);
                \draw[draw=black, double=white, double distance=5pt] (2.8,0.2) to [out=up,in=down] (B.leftleg);
                \draw[draw=black, double=white, double distance=5pt] (3.2,0.2) to [out=up,in=down] (B.rightleg);
            \end{scope}
            \draw[rounded corners, draw=black] (2.6,-0.2) rectangle (3.4,0.2);

            \begin{scope}[internal ribbon scope]
                \node[Cup, top, scale=2, grey] (C) at (6,0.5) {}; 
                \draw[draw=black, double=white, double distance=5pt] (6,0.2) -- (C);
            \end{scope}
            \draw[rounded corners, draw=black] (5.8,-0.2) rectangle (6.2,0.2);

            \begin{scope}[internal ribbon scope]
                \node[Cap, bot, scale=2, grey] (D) at (9,-0.5) {}; 
                \draw[draw=black, double=white, double distance=5pt] (9,-0.2) -- (D);
            \end{scope}
            \draw[rounded corners, draw=black] (8.8,-0.2) rectangle (9.2,0.2);

            \begin{scope}[internal ribbon scope]
                \node[Cyl, bot, top, scale=2, grey] (E) at (12,0) {};
                \draw[draw=black, double=white, double distance=5pt] (E.bottom) -- (E.top); 
            \end{scope}
        \end{tikzpicture}
        \caption{Associated ribbon graph to each generating 1-morphism of $\Bord{sig}$.} \label{fig:assoc-graph}
    \end{figure}

    Then, for a general 1-morphism in $\Bord{sig}$, we may glue together the ribbon graphs embedded in each generating component along their endpoints, resulting in an embedded ribbon graph $\Gamma_\Sigma$ in $H_\Sigma$.
\end{eg}

\begin{defn}
    Given a 1-morphism in $\Bord{sig}$ with underlying surface and handlebody $(\Sigma,H_{\Sigma})$, call this embedded ribbon graph $\Gamma_\Sigma$ in $H_\Sigma$ the \emph{associated ribbon graph}.
\end{defn}

Note that the embedding of the ribbon graph $\Gamma$ in Example \ref{eg:internal-graph} naturally extends to an embedding of the \emph{thickening} $|\Gamma| \times [-1,1]$. If $H_\Sigma$ is embedded in a closed 3-manifold $M$, then it is ambiently isotopic to this embedded thickened graph.

In order to define invariants of a ribbon graph, we first need to define the notion of a colouring. These will depend on the choice of a semisimple ribbon category $\mathcal{C}$.

\begin{defn}[{\cite[Definition 8.10.1]{EGNO}}]
    A \emph{ribbon category} $\mathcal{C}$ is a rigid balanced braided monoidal category whose twist is compatible with the duals.
\end{defn}
\begin{defn}\label{defn:ss-ribbon}
    A \emph{semisimple ribbon category} is a ribbon linear category which is linearly equivalent to $\vect^{\oplus s}$ for some $s$ and whose unit $\mathbbm{1}$ is simple.
\end{defn}

Let $s$ be the number of equivalence classes of simple objects of $\mathcal{C}$, and choose a set of representatives $\mathbb{X} := \set{\mathbbm{1}=X_1,X_2,\ldots,X_s}$. These form the set of \emph{colours}. For each $i$, let $X_i^*$ denote the dual of $X_i$. Throughout, we identify $X_i^{**}$ with $X_i$ via the canonical isomorphism in \cite[Proposition 8.10.6]{EGNO}.

\begin{defn}
    A \emph{ribbon-colouring} of an abstract ribbon graph $\Gamma = (V,E,L)$ is a assignment of objects of $\mathbb{X}$ to the edges and loops of $\Gamma$, i.e. a function $K: E\cup L \rightarrow \mathbb{X}$.
\end{defn}

A ribbon-colouring determines the vector space of possible labellings of each vertex in the ribbon graph, namely the space of morphisms in $\mathcal{C}$ which type-check with the colours of the incoming and outgoing strands.

\begin{defn}
    Let $K$ be a ribbon-colouring of $\Gamma = (V,E,L)$, and $v\in V$. We may assign a vector space to this as follows: let $X_1,\ldots,X_a$ be the colours of the source edges of $v$, $Y_1,\ldots,Y_b$ the colours of the target edges, which both ordered according to the orderings on $E_s(v),E_t(v)$. Define $\epsilon_i \in \set{\pm1}$ for each $1\le i\le a$ to be $+1$ if the $i^\text{th}$ source edge is incoming, and $-1$ otherwise. Define $\varepsilon_j \in \set{\pm1}$ for each $1\le j \le b$ to be $+1$ if the $j^\text{th}$ target edge is outgoing and $-1$ otherwise. Then, define the \emph{space of labels of $v$} to be
    $$\mathcal{V}_K(\Gamma,v) := \Hom_\mathcal{C}\pa{\bigotimes_{i=1}^a X_i^{\epsilon_i}, \bigotimes_{j=1}^b Y_j^{\varepsilon_j}}$$ 
    where each $X^{+1}$ is taken to be $X$ and $X^{-1}$ is taken to be the dual $X^*$.

    The \emph{space of labels} of $\Gamma$ is the vector space $$\mathcal{V}_K(\Gamma) := \bigotimes_{v\in V} \mathcal{V}_K(\Gamma,v).$$
\end{defn}

More generally, we may extend this definition to ribbon graphs with partial ribbon-colourings, by taking the direct sum over all possible ribbon-colourings that the partial ribbon-colouring extends to:
\begin{defn}\label{defn:partial-space-labels}
    Let $\mathbf{K}$ be a \emph{partial ribbon-colouring}, i.e. a partial function $E\cup L \dashrightarrow \mathbb{X}$. The \emph{space of labels} is the vector space $$\mathcal{V}_\mathbf{K}(\Gamma) := \bigoplus_{K \in \mathbb{K}} \mathcal{V}_{K}(\Gamma)$$ where $\mathbb{K}$ is the set of ribbon-colourings of $\Gamma$ extending $\mathbf{K}$.
\end{defn}

By starting with a ribbon-colouring $K$ and then choosing a label for each of the coupons (i.e. an element $f_v \in \mathcal{V}_K(\Gamma,v)$ for each vertex $v$), we recover a notion of ribbon graph as described in \cite{RT-ribbon, RT, Turaev-book}:

\begin{defn}
    A \emph{labelled ribbon graph} is an abstract ribbon graph $\Gamma$ with a ribbon-colouring $K$ and a label $f_v \in \mathcal{V}_K(\Gamma,v)$ for each $v\in V$. Given a 3-manifold $M$, A \emph{labelled embedded ribbon graph} is an embedded ribbon graph whose underlying abstract ribbon graph is labelled.
\end{defn}

In particular, the case $M = \mathbb{R}^2 \times [0,1]$ recovers the notion of a coloured ribbon graph as defined in \cite{RT-ribbon}, which are referred to as $v$-coloured ribbon graphs in \cite{Turaev-book}. These are the ribbon graphs on which it makes sense to define the \emph{operator invariant}, which may be obtained as follows: Apply a suitable ambient isotopy such that upon projecting onto $\mathbb{R}\times[0,1]$, all coupons are locally as drawn in Figure \ref{fig:ribbon-vertex} (with the source on the bottom and target on the top), and all crossings of ribbons are generic. Then, evaluating the resulting coloured ribbon diagram yields the operator invariant. Suppose that the endpoints of such a ribbon graph $\Gamma$ in $\mathbb{R}^2\times\set{0}$ are coloured $X_1,\ldots,X_a$, with these ordered from left to right under the projection, the endpoints in $\mathbb{R}^2\times\set{1}$ are likewise $Y_1,\ldots,Y_b$, and let $\epsilon_i,\varepsilon_j$ where $1\le \epsilon_i \le a, 1\le \varepsilon_j \le b$ for each $i,j$ encode the directions of these strands. Then, the operator invariant is a morphism in $\mathcal{C}$ denoted $$F(\Gamma, K, \set{f_v}): \bigotimes_{i=1}^a X_i^{\epsilon_i}\rightarrow \bigotimes_{j=1}^b Y_j^{\varepsilon_j}.$$

\begin{eg}
    Let $X\in\mathbb{X}$. Then, its dimension $\dim X := \tr(\id_X)$ is exactly the operator invariant of the unknot with blackboard framing, coloured $X$. 
\end{eg}

If $\Gamma = \Gamma' \sqcup L$ consists of a labelled ribbon graph $(\Gamma',K)$ along with some uncoloured loops, we may define the operator invariant on this by taking the weighted sum over all possible colours of the uncoloured loops: $$F(\Gamma,K,\set{f_v}) := \sum_{K_L \in \mathbb{K}_L} \dim(K_L) F(\Gamma, K\sqcup K_L, \set{f_v})$$ where the sum is taken over all colourings $K_L$ of the uncoloured loops and $$\dim(K_L) := \prod_{\ell\in L} \dim(K_L(\ell))$$ is the product of the dimensions of the colours of the components of $L$.

Note that by \cite[Lemma II.2.2.1]{Turaev-book}, this operator invariant does not depend on the choice of orientations of the uncoloured loops. Moreover, Reshetikhin and Turaev showed:
\begin{thm}[\cite{RT-ribbon}]
    The operator invariant $F$ is invariant on labelling-preserving isotopies. Moreover, $F$ is functorial, i.e. gluing endpoints of ribbon graphs corresponds to the composition of morphisms in $\mathcal{C}$. $F$ is also monoidal, i.e taking disjoint unions of ribbon graphs corresponds to taking a tensor product in $\mathcal{C}$.
\end{thm}

This machinery allows us to define certain invariants of semisimple ribbon categories which will later be useful:
\begin{defn} \label{defn:cat-consts}
    For a semisimple ribbon category $\mathcal{C}$, let its \emph{global dimension} be the operator invariant of an uncoloured untwisted unknot, $$D := \sum_{i=1}^s \dim(X_s)^2.$$ Additionally, let the \emph{Gauss sums} $$p_\pm := \sum_{i=1}^s \theta_{X_s}^{\pm 1}\dim(X_s)^2$$ be the operator invariants of an uncoloured unknot with a single twist (right-handed for $p_+$, left-handed for $p_-$).

    Let the \emph{anomaly} be the ratio $\frac{p_+}{p_-}$.
\end{defn}
\begin{defn} \label{defn:s-matrix}
    For a semisimple ribbon category $\mathcal{C}$, let its \emph{$S$-matrix} be the $s\times s$ matrix whose $(i,j)$-entry is the operator invariant of a Hopf link (with blackboard framing) with one component coloured by $i$ and the other by $j$, i.e. $$S_{ij} := \tr(\beta_{ij}\beta_{ji}).$$
\end{defn}

Additionally, note that the operator invariant $F(\Gamma,K,\set{f_v})$ is linear in each $f_v$. Thus, by extending linearly, the operator invariant is still well-defined when $\set{f_v}$ is replaced by a general element $\mathbf{f}$ of the space of labels $\mathcal{V}_K(\Gamma)$. In other words, $F(\Gamma,K,-)$ defines a linear map from $\mathcal{V}_K(\Gamma)$ to the relevant hom space of $\mathcal{C}$ (which depends on the colours of the external edges). By taking direct sums, given a partial ribbon-colouring $\mathbf{K}$ which colours all the external edges, $F(\Gamma,\mathbf{K},-)$ defines a linear map from $\mathcal{V}_{\mathbf{K}}(\Gamma)$ a corresponding hom space of $\mathcal{C}$, and it thus makes sense to write $F(\Gamma,\mathbf{K},\mathbf{f})$ where $\mathbf{f} \in \mathcal{V}_{\mathbf{K}}(\Gamma)$.

The construction of $RT_{\mathcal{C},\sqrt{D}}$ will come from additional structure which arises from letting $\mathcal{C}$ be a modular tensor category.
\begin{defn}[{\cite[Definition 3.1.1]{BK}}]
    A \emph{modular tensor category} is a semisimple ribbon category $\mathcal{C}$ whose $S$-matrix is invertible.
\end{defn}

In the case where $\mathcal{C}$ is a modular tensor category, Reshetikhin and Turaev extended the operator invariant to an invariant $\tau$ of labelled embedded ribbon graphs in closed 3-manifolds. This was then used in \cite{RT} to define a (projective) symmetric monoidal functor $\bord{or} \rightarrow \mathbf{Vect}_k$. Our construction of $RT_{\mathcal{C},\sqrt{D}}$ will also use this invariant $\tau$ in a similar way. At the level of closed surfaces and bordisms between closed surfaces, this agrees with the Reshetikhin--Turaev construction, as demonstrated later in \textsection\ref{subsection:rt-compare}.

The 3-manifold invariant depends on an extra choice of square root of the global dimension $D$, denoted $\sqrt{D}$, and is written in terms of $\sqrt{D}$ and $p_-$.
\begin{defn}[\cite{RT}]
    Let $M$ be a closed connected 3-manifold and $\Gamma$ a labelled embedded ribbon graph in $M$ with ribbon-colouring $K$ and vertex labels $\set{f_v}$. The invariant $\tau(M,\Gamma,K,\set{f_v}) \in k$ is defined as follows: 

    Let $L$ be a framed link in $S^3$ (viewed as an embedded $S^1\times I$) with $\# L$ components, and $(\Gamma,K,\set{f_v})$ a labelled embedded ribbon graph in $S^3$ such that performing surgery on $L$ yields the pair $(M,|\Gamma|)$. Then, we have an embedding of $|\Gamma\cup L|$ in $S^3$; call this a \emph{surgery presentation} of the pair $(M,|\Gamma|)$. $\Gamma\cup L$ is now a labelled ribbon graph with some uncoloured loops. Additionally, let $X_L$ be the trace of the framed link $L$ (the 4-manifold with boundary $M$ given by attaching 2-handles to $D^4$ along the framed link $L$). Define $\sigma(X_L)$ to be the signature of $X_L$.

    The invariant on the labelled embedded ribbon graph $\Gamma$ is a scalar multiple of the operator invariant of $\Gamma\cup L$: $$\tau(M,\Gamma,K,\set{f_v}) := p_-^{\sigma(X_L)}\sqrt{D}^{-\sigma(X_L)-\# L-1} F(\Gamma \cup L,K,\set{f_v}).$$

    In general, if $M$ is not connected, say $(M,\Gamma) = \bigsqcup (M_i,\Gamma_i)$ where each $M_i$ is connected, then we define $$\tau(M,\Gamma,K,\set{f_v}) := \prod\tau(M_i,\Gamma_i,K|_{\Gamma_i},\set{f_v})$$ to be the product of $\tau$ over all the components.
\end{defn}

As shown in \cite{RT}, such a surgery link $L$ always exists, and $\tau$ is independent of the choice of $L$, as well as the choice of orientation of each of the link components. As is the case with the operator invariant, we may extend this definition linearly to make sense of $\tau(M,\Gamma,\mathbf{K},\mathbf{f})$ where $\mathbf{K}$ is a partial ribbon-colouring and $\mathbf{f} \in \mathcal{V}_\mathbf{K}(\Gamma)$.

In general, if the choices of ribbon-colouring and vertex labels are clear, we drop the $K$ and $\set{f_v}$ from the notation and simply write $\tau(M,\Gamma)$.

\begin{rk}
    As in \cite{Turaev-book}, this has normalisation $\tau(S^3,\emptyset) = \frac1{\sqrt{D}}$; this is the case where $\Gamma \cup L$ is the empty diagram.
\end{rk}

We record some basic properties of $\tau$ which may be easily verified; see \cite[pp. 81-82]{Turaev-book}.

\begin{lem}\label{lem:tau-properties}
    The following properties of $\tau$ hold.
    \begin{enumerate}
        \item $\tau(S^1\times S^2,\emptyset) = 1$.
        \item Given labelled embedded ribbon graphs $\Gamma$ in $M$ and $\Gamma'$ in $M'$, $$\tau(M\#M',\Gamma\cup\Gamma') = \sqrt{D}\tau(M,\Gamma)\tau(M',\Gamma') = \sqrt{D}\tau(M\sqcup M', \Gamma\cup\Gamma').$$
    \end{enumerate}
\end{lem}

We now record some ribbon graph transformations which are compatible with the invariants $F$ and $\tau$.

\begin{lem}\label{lem:cyc-shift}
    Let $\Gamma$ be a ribbon graph with partial ribbon-colouring $\mathbf{K}$, and $v\in V$. Let $\Gamma'$ be the same ribbon graph, but with a different choice of partition into source and target for each vertex. Then, $|\Gamma|$ and $|\Gamma'|$ are diffeomorphic by identifying coresponding coupons and ribbons, and $\mathbf{K}$ induces a corresponding partial ribbon-colouring $\mathbf{K}'$ of $\Gamma'$.
    \begin{enumerate}
        \item There is a canonical isomorphism $\phi: \mathcal{V}_\mathbf{K}(\Gamma) \rightarrow \mathcal{V}_{\mathbf{K}' }(\Gamma')$.
        \item Given any embedding of $\Gamma$ into a 3-manifold $M$, the corresponding embedding of $\Gamma'$ satisfies $$\tau(M,\Gamma,\mathbf{K},\mathbf{f}) = (M,\Gamma',\mathbf{K}',\phi(\mathbf{f}))$$ for each $\mathbf{f}\in\mathcal{V}_\mathbf{K}(\Gamma)$.
    \end{enumerate}
\end{lem}
\begin{proof}
    The isomorphism is given by composing with evaluation and coevaluation maps as in \cite[Exercise I.1.8.1]{Turaev-book}. For example, an isomorphism $\Hom(X_1,X_2\otimes X_3) \rightarrow \Hom(X_2^* \otimes X_1, X_3)$ is given by sending $f$ to $(\id_{X_2^*} \otimes f)\circ (\mathrm{ev}_{X_2} \otimes \id_{X_3})$. That these define a canonical isomorphism $\phi$ is part of the proof of the isotopy invariance of $F$; see in particular \cite[Figures I.4.16-17]{Turaev-book}. By construction, $\phi$ satisfies (2) as it is a composition of isomorphisms which satisfy (2).
\end{proof}

Henceforth, we may thus ignore the choice of source and target of each vertex.

In the definition of $RT_{\mathcal{C},\sqrt{D}}$, we will be considering ribbon graphs $\Gamma$ with partial ribbon-colourings which leave all the internal edges uncoloured. The spaces of labels of these are large direct sums of tensor products of vector spaces. By contracting each $\Gamma$ along an uncoloured subtree, we may obtain a new ribbon graph whose space of labels is canonoically identified with that of $\Gamma$. In particular, if we contract $\Gamma$ along a spanning subtree, the space of labels of the resulting graph is simply the space of labels of a single vertex.

Suppose $\Gamma$ has a ribbon subgraph $T$ whose underlying graph is a tree. Then $|T|$ is a subsurface of $|\Gamma|$ diffeomorphic to a disc. We may construct a ribbon graph $\Gamma/T$ such that $|\Gamma/T|$ is exactly $|\Gamma|$, but with $|T|$ viewed as one disc associated to a single vertex.

\begin{defn}
    Let $\Gamma = (V,E,L)$ be a ribbon graph, and $T = (V_T,E_T,\emptyset)$ be a ribbon subgraph whose underlying graph is a tree. The \emph{contraction} $\Gamma/T = (V/V_T,E\setminus E_T,L)$ is the ribbon graph with the cyclic order at the vertex $V_T/V_T$ induced by the cyclic order of the edges incident at the disc $|T|$. 
\end{defn}

By construction, $|\Gamma|$ and $|\Gamma/T|$ are diffeomorphic by identifying $|T|\subset|\Gamma|$ with the disc corresponding to the vertex $V_T/V_T$.

If all the edges of $T$ are uncoloured, then contracting along $T$ does not change the space of labels:

\begin{lem}\label{lem:tree}
    Let $\Gamma$ be a ribbon graph with partial ribbon-colouring $\mathbf{K}$. Let $T$ be a ribbon subtree whose edges are uncoloured. Then $\mathbf{K}$ induces a (partial) ribbon-colouring $\mathbf{K}_{/T}$ of $\Gamma/T$.

    \begin{enumerate}
        \item There is a canonical ``composition'' isomorphism $\bigcirc:\mathcal{V}_{\mathbf{K}}(\Gamma)\rightarrow\mathcal{V}_{\mathbf{K}_{/T}}(\Gamma/T)$.
        \item If $\Gamma$ is in addition an embedded ribbon graph in a manifold $M$, the corresponding embedded ribbon graph $\Gamma/T$ in $M$ satisfies $$\tau(M,\Gamma,\mathbf{K},\mathbf{f}) = \tau(M,\Gamma/T,\mathbf{K}_{/T},\bigcirc(\mathbf{f}))$$ for each $\mathbf{f}\in\mathcal{V}_\mathbf{K}(\Gamma)$.
    \end{enumerate}
\end{lem}
\begin{proof}
    First consider the case where $T$ is a single edge $e$. Suppose in addition that all edges are coloured but $e$. We may choose the source and target of the two vertices involved such that $h(e)$ has the $e$ as its only source edge, and $t(e)$ has the $e$ as its only target edge. Then, the isomorphism in question is given by the composition map $$\bigoplus_{Z\in\mathbb{X}}\Hom(X_1\otimes\cdots\otimes X_a, Z)\otimes \Hom(Z, Y_1\otimes\cdots\otimes Y_b) \rightarrow  \Hom(X_1\otimes\cdots\otimes X_a, Y_1\otimes\cdots\otimes Y_b).$$ This is an isomorphism as $\mathcal{C}$ is semisimple. Moreover, this satisfies (2), by construction. By taking direct sums, this defines the isomorphism in the case when $T$ has a single edge $e$ but there are uncoloured edges other than $e$.

    In general, we may contract $T$ one edge at a time to produce an isomorphism. By construction, this isomorphism satisfies (2). Now, note that (2) uniquely determines this linear map, and hence the isomorphism is canonical.
\end{proof}

Finally, given an abstract ribbon graph $\Gamma$, it will be useful to construct an abstract ribbon graph $-\Gamma$ whose spaces of labels are dual to that of $\Gamma$.

\begin{defn}
    Let $\Gamma = (V,E,L)$ be a ribbon graph. The \emph{reverse graph} $-\Gamma = (V,E,L)$ is the ribbon graph with $h,t:E\rightarrow V\cup\set{\infty}$ interchanged, and the cyclic ordering at each vertex reversed.
\end{defn}

Note $|\Gamma|$ and $|-\Gamma|$ are diffeomorphic via an orientation-reversing diffeomorphism, so if $\Gamma$ is an embedded ribbon graph in a 3-manifold $M$, then we correspondingly have $-\Gamma$ embedded in $-M$.

Additionally, any partial ribbon-colouring $\mathbf{K}$ of $\Gamma$ induces a corresponding partial ribbon-colouring of $-\Gamma$, given by colouring the relevant ribbons in the exact same way.

\begin{lem}\label{lem:rev-dual}
    Let $\mathbf{K}$ be a partial ribbon-colouring of $\Gamma$. Then, $\mathcal{V}_\mathbf{K}(-\Gamma)$ is canonically isomorphic to the linear dual of $\mathcal{V}_\mathbf{K}(\Gamma)$.
\end{lem}
\begin{proof}
    It suffices to prove this for ribbon-colourings, and then take direct sums. Let $v\in V$. Then, it suffices to show that the spaces of labels of $v$ in the respective graphs are dual to each other.

    In $\Gamma$, let all strands at $v$ be target strands, and in $-\Gamma$, let them all be source strands. Then, the two vector spaces are now $$\Hom_{\mathcal{C}}\pa{\mathbbm{1}, \bigotimes_{i=1}^a X_i^{\epsilon_i}} \text{ and } \Hom_{\mathcal{C}}\pa{\bigotimes_{i=1}^a X_i^{\epsilon_i}, \mathbbm{1}}$$ for some $X_i$ and $\epsilon_i$.

    These two vector spaces are dual to each other, with the evaluation map given by composition; here, we are identifying $\Hom_{\mathcal{C}}(\mathbbm{1},\mathbbm{1})$ with the base field $k$.
\end{proof}
\begin{rk}
    This isomorphism is unique up to multiplication by a scalar. The purpose of providing a proof here is to fix a normalisation.
\end{rk}
\begin{rk}
    For a general $v$ with both target and source strands, then under the isomorphism defined in Lemma \ref{lem:cyc-shift}, the evaluation map corresponds to composition (in either order), followed by taking the trace; see also \cite[Lemma II.4.2.3]{Turaev-book}.
\end{rk}

We let $\langle\cdot,\cdot\rangle_\Gamma: \mathcal{V}_\mathbf{K}(\Gamma)\otimes \mathcal{V}_\mathbf{K}(-\Gamma) \rightarrow k$ denote the evaluation map in Lemma \ref{lem:rev-dual}. For an element of the form $\bigotimes_{v\in V} f_v \otimes \bigotimes_{v\in V} g_v \in \mathcal{V}_K(\Gamma)\otimes \mathcal{V}_K(-\Gamma) \subseteq \mathcal{V}_\mathbf{K}(\Gamma)\otimes \mathcal{V}_\mathbf{K}(-\Gamma)$ where $K$ is a ribbon-colouring extending $\mathbf{K}$, we have $$\left\langle\bigotimes_{v\in V} f_v, \bigotimes_{v\in V} g_v\right\rangle_\Gamma := \prod_{v\in V}\tr(f_vg_v) = \prod_{v\in V}\tr(g_vf_v).$$

\subsection{Defining $RT_{\mathcal{C},\sqrt{D}}$} \label{subsection:def-rt}

We now have all the tools that we need in order to define the extended Reshetikhin--Turaev TQFT $RT_{\mathcal{C},\sqrt{D}}:\Bord{sig,dec}\rightarrow \KV$. Let $\mathcal{C}$ be a modular tensor category, and $\sqrt{D}$ be a choice of square root of its global dimension. Recall that we have chosen representatives $\mathbb{X} = \set{\mathbbm{1}=X_1,\ldots,X_s}$ of the equivalence classes of simple objects of $\mathcal{C}$.

We now define $RT_{\mathcal{C},\sqrt{D}}$ levelwise, starting from the objects. This will resemble the functor in \cite[\textsection 4]{BDSV4} at the level of objects and 1-morphisms, but will differ at the 2-morphism level.

To the standard circle, we assign the object of $\KV$ given by the integer $s$. In general, the disjoint union of $n$ circles is assigned $s^n$. This is taken to index the set $\mathbb{X}^n$ via the lexicographical ordering, and can be interpreted as a choice of assignment of some $X\in\mathbb{X}$ to each circle.

In this way, a morphism $s^m \rightarrow s^n$ in $\KV$ may be taken to be a matrix of vector spaces whose indexing set is pairs of the form $\pa{X_{i_1}\boxtimes\cdots\boxtimes X_{i_m},X_{j_1}\boxtimes\cdots\boxtimes X_{j_n}}$. Denote such a matrix as $\pa{V_\mathbf{i}^\mathbf{j}}$ where $\mathbf{i} = (i_1,\ldots,i_m)$ and $\mathbf{j} = (j_1,\ldots,j_n)$.

On each of the generating 1-morphisms, define:
\begin{align*}
    RT_{\mathcal{C},\sqrt{D}}(\tikztinypants) &= \pa{\Hom_\mathcal{C}(X_{i_1}\otimes X_{i_2}, X_j)_{i_1i_2}^j}
    \\
    RT_{\mathcal{C},\sqrt{D}}(\tikztinycopants) &= \pa{\Hom_\mathcal{C}(X_i, X_{j_1}\otimes X_{j_2})_i^{j_1j_2}}
    \\
    RT_{\mathcal{C},\sqrt{D}}(\tikztinycup) &= \pa{\Hom_\mathcal{C}(\mathbbm{1}, X_j)^j}
    \\
    RT_{\mathcal{C},\sqrt{D}}(\tikztinycap) &= \pa{\Hom_\mathcal{C}(X_i, \mathbbm{1})_i}
\end{align*}

Now, a general 1-morphism $\Sigma$ is a formal composition of disjoint unions of these generating 1-morphism. We wish to send $\Sigma$ to the corresponding composition of morphisms in $\KV$. This may be computed as follows:

Let $\Gamma_\Sigma$ be the associated ribbon graph. Then, each choice of indices $\mathbf{i}$ and $\mathbf{j}$ corresponds exactly to a choice of colouring of the endpoints of $\Gamma_\Sigma$. This determines at most one colouring of the external edges: if an external edge $e$ has either endpoint $X$, then colour the whole by $X$. There is a possibility that an external edge has both ends coloured, in which case there could be no possible colouring if the colours are incompatible.

In the general case, this defines a partial ribbon-colouring of $\Gamma$; call it $\mathbf{K}_{\mathbf{i},\mathbf{j}}$. Then, we may set the $(\mathbf{i},\mathbf{j})$-entry of $RT_{\mathcal{C},\sqrt{D}}(\Sigma)$ to be the space of labels of $\mathbf{K}_{\mathbf{i},\mathbf{j}}$. In the case that $\mathbf{K}_{\mathbf{i},\mathbf{j}}$ is not well-defined, set this entry to be the zero vector space.

In other words, let $\mathbb{K}_{\mathbf{i},\mathbf{j}}$ denote the set of ribbon-colourings of $\Gamma$ extending $\mathbf{K}_{\mathbf{i},\mathbf{j}}$ (In the case where $\mathbf{K}_{\mathbf{i},\mathbf{j}}$ does not exist, then set $\mathbb{K}_{\mathbf{i},\mathbf{j}} = \emptyset$.) Then define $$RT_{\mathcal{C}
,\sqrt{D}}(\Sigma) := \pa{\pa{\bigoplus_{K\in \mathbb{K}_{\mathbf{i},\mathbf{j}}} \mathcal{V}_K(\Gamma_\Sigma)}_{\mathbf{i}}^\mathbf{j}}$$

It is clear that this agrees with the definition given above on the generating 1-morphisms. As the direct sum of tensor products implements the composition law of matrices of vector spaces, this is indeed the desired composition.

Finally, we define $RT_{\mathcal{C},\sqrt{D}}$ on 2-morphisms. This will be done in the spirit of the original Reshetikhin--Turaev construction, and is similar to the approach of \cite{Tsumura}, after accounting for changes at the 1-morphism level.

First, given a 2-morphism $(M,n)$, recall from Definition \ref{defn:olm} that we may obtain a closed 3-manifold $\ol{M}$ from $M$ by gluing in the handlebodies $H_{\parin M}$, $-H_{\parout M}$ as well as the solid cylinders $V_{\parIn M} = I \times D_{\parin\parin M}$ and $-V_{\parOut M} = -I \times D_{\parout\parin M}$. The two handlebodies come with associated embedded ribbon graphs $\Gamma_{\parin M}$ and $-\Gamma_{\parout M}$. Now, take $-\Gamma_{\parout M}$ and then join its endpoints without twisting through the central axes of the solid cylinders $V_{\parIn M}$ and $-V_{\parOut M}$ to the endpoints of $\Gamma_{\parin M}$; the reversal ensures that the directions of the edges match up. In this way, we have now constructed an embedded ribbon graph $\Gamma_M$ in the closed 3-manifold $\ol{M}$.

Now, for a 2-morphism $(M,n): \Sigma \rightarrow \Sigma'$ in $\Bord{sig}$, $RT_{\mathcal{C},\sqrt{D}}(\Sigma)$ and $RT_{\mathcal{C},\sqrt{D}}(\Sigma')$ are both matrices of vector spaces with the same indexing set. As with the 1-morphisms, the choice of the index $(\mathbf{i},\mathbf{j})$ determines at most one colouring of the edges of $\Gamma_M$ which go through the solid cylinders, so defining at most one partial ribbon-colouring $\mathbf{K}_M$ of $\Gamma_M$. Let $\mathbb{K}_M$ be the set of ribbon-colourings of $\Gamma_M$ extending $\mathbf{K}_M$.

We want to define a linear map $$RT_{\mathcal{C},\sqrt{D}}(M,n)_\mathbf{i}^\mathbf{j}:\bigoplus_{K\in \mathbb{K}_{\mathbf{i},\mathbf{j}}} \mathcal{V}_K(\Gamma_\Sigma) \rightarrow \bigoplus_{K'\in \mathbb{K}_{\mathbf{i},\mathbf{j}}'} \mathcal{V}_{K'}(\Gamma_{\Sigma'}).$$ By duality and Lemma \ref{lem:rev-dual}, this is equivalent to defining a linear map $$(RT_{\mathcal{C},\sqrt{D}}(M,n)_\mathbf{i}^\mathbf{j})_{K,K'}: \mathcal{V}_K(\Gamma_\Sigma)\otimes\mathcal{V}_{K'}(-\Gamma_{\Sigma'}) \rightarrow \mathbb{C}$$ for each $K,K'$.

Notice that choosing $K$ and $K'$ is exactly the same as choosing a ribbon-colouring $K_M\in\mathbb{K}_M$, and for this $K_M$, we have $\mathcal{V}_{K_M}(\Gamma_M) \cong \mathcal{V}_K(\Gamma_\Sigma)\otimes\mathcal{V}_{K'}(-\Gamma_{\Sigma'})$. In other words, we want to define a linear map $\mathcal{V}_{K_M}(\Gamma_M) \rightarrow \mathbb{C}$ for each $K_M\in\mathbb{K}_M$.

This is done by using a rescaled version of the operator invariant. Let $$\dim_{\mathrm{int}}(K') := \prod_{e \in E_\mathrm{int}(\Gamma_{\Sigma'})}\dim K'(e)$$ be the product of the dimensions of the colours of the internal edges of $\Gamma_{\Sigma'}$. Now, our desired linear map is defined as follows: for each $\mathbf{f}\in\mathcal{V}_{K_M}(\Gamma_M)\cong \mathcal{V}_K(\Gamma_\Sigma)\otimes\mathcal{V}_{K'}(-\Gamma_{\Sigma'})$, set 
\begin{equation}\label{eq:def-rt}
    (RT_{\mathcal{C},\sqrt{D}}(M,n)_\mathbf{i}^\mathbf{j})_{K,K'}(\mathbf{f}):=\pa{\frac{\sqrt{D}}{p_-}}^n \sqrt{D}^{\chi\pa{\ol{\Sigma'}}/2} \dim_{\mathrm{int}}(K')\tau(\ol{M},\Gamma_M,K_M,\mathbf{f}).
\end{equation}
\begin{rk}
    The $\dim_\mathrm{int}(K')$ term is necessary for the linear map constructed in this way to be compatible with the isomorphisms given by contracting $\Gamma_M$ along subtrees; see Lemma \ref{lem:contract-tree-rt}. One may alternatively get rid of this term by choosing a different normalisation in Lemma \ref{lem:rev-dual}.
\end{rk}

Then, $RT_{\mathcal{C},\sqrt{D}}(M,n)_\mathbf{i}^\mathbf{j}$ is the sum of the $(RT_{\mathcal{C},\sqrt{D}}(M,n)_\mathbf{i}^\mathbf{j})_K^{K'}: \mathcal{V}_K(\Gamma_\Sigma)\rightarrow \mathcal{V}_{K'}(\Gamma_{\Sigma'})$ which are the unique linear maps satisfying 
\begin{equation} \label{eq:rt-nodual}
    \left\langle (RT_{\mathcal{C},\sqrt{D}}(M,n)_\mathbf{i}^\mathbf{j})_K^{K'}(\mathbf{f}), \mathbf{g} \right\rangle_{\Gamma_{\Sigma'}} = (RT_{\mathcal{C},\sqrt{D}}(M,n)_\mathbf{i}^\mathbf{j})_{K,K'}(\mathbf{f}\otimes\mathbf{g})
\end{equation}
for all $\mathbf{f} \in \mathcal{V}_K(\Gamma_\Sigma),\mathbf{g}\in \mathcal{V}_{K'}(\Gamma_{\Sigma'})$.

\section{Functoriality of $RT_{\mathcal{C}}$} \label{section:rt-functor}

The aim of this section is to prove that the construction presented in the previous section indeed defines a symmetric monoidal functor $RT_\mathcal{C}:\Bord{sig/2,dec} \rightarrow \KV$. We shall first prove in Theorem \ref{thm:rt-functor} that $RT_{\mathcal{C},\sqrt{D}}$ indeed defines a symmetric monoidal functor $\Bord{sig,dec}\rightarrow \KV$. Then, we will verify in Proposition \ref{prop:rt-well-def} that on the subbicategory $\Bord{sig/2,dec}\subseteq \Bord{sig,dec}$, $RT_{\mathcal{C},\sqrt{D}}$ and $RT_{\mathcal{C},-\sqrt{D}}$ agree, and so we have a symmetric monoidal functor $RT_\mathcal{C}$ that does not depend on the choice of square root. Along the way, we will also check that our construction agrees with the Reshetikhin--Turaev TQFT as defined in \cite{RT}.

By construction, $RT_{\mathcal{C},\sqrt{D}}$ is compatible with the monoidal structure at the object and 1-morphism level, as well as with the composition of 1-morphisms. In order to prove Theorem~\ref{thm:rt-functor}, it thus remains to check that at the 2-morphism level, $RT_{\mathcal{C},\sqrt{D}}$ is compatible with the monoidal structure, horizontal and vertical composition, as well as identities.

The compatibility with the monoidal structure is clear, as each term in (\ref{eq:def-rt}) is multiplicative under disjoint unions. Proving the compatibility with vertical composition (Proposition \ref{prop:vert-comp}) and horizontal composition (Proposition \ref{prop:hor-comp}) will occupy most of this section. This is done in a similar way to \cite{Turaev-book,Tsumura}, but accounting for the more general ribbon graphs associated to the 1-morphisms. Then, we will finally check in Lemma \ref{lem:rt-id} that identity 2-morphisms are sent to identity 2-morphisms.

\subsection{Some string calculus results}

Before we proceed, we first record some useful results of the string calculus of modular tensor categories. Throughout, ribbon graphs are coloured with the objects of a modular tensor category $\mathcal{C}$.

Given a labelled embedded ribbon graph in $\mathbb{R}^2 \times [0,1]$, we obtain the operator invariant $F(\Gamma,K,\set{f_v}): X \rightarrow Y$, a morphism in $\mathcal{C}$.  Post-composition by this morphism yields a linear map:

\begin{defn}
    Let $\Gamma$ be a labelled embedded ribbon graph, possibly with some uncoloured loops, in $\mathbb{R}^2 \times [0,1]$. Then let $$F_*(\Gamma,K,\set{f_v}): \Hom(\mathbbm{1},X) \rightarrow \Hom(\mathbbm{1},Y)$$ be the linear map given by post-composition with its operator invariant.
\end{defn}

Let $Y_1,\ldots,Y_a$ be any $a$ objects in $\mathcal{C}$. Then, as presented in \cite[\textsection IV.2.2]{Turaev-book}, the tensor product of morphisms defines a canonical inclusion $$\varphi: \bigotimes_{p=1}^a \Hom(\mathbbm{1},Y_p) \rightarrow \Hom\pa{\mathbbm{1},\bigotimes_{p=1}^a Y_p}$$ and moreover the dual of the linear map $$\bigotimes_{p=1}^a \Hom(Y_p,\mathbbm{1}) \rightarrow \Hom\pa{\bigotimes_{p=1}^a Y_p,\mathbbm{1}}$$ defines, under the identification in Lemma \ref{lem:rev-dual}, a canonical projection $$\psi: \Hom\pa{\mathbbm{1},\bigotimes_{p=1}^a Y_p} \rightarrow \bigotimes_{p=1}^a \Hom(\mathbbm{1},Y_p)$$
onto this subspace.

\begin{lem}[{\cite[Lemma IV.2.2.1]{Turaev-book}}]\label{lem:tens-hom}
    The composition $\varphi\psi$ is equal to $D^{-a}$ times $F_*$ of the following labelled embedded ribbon graph in $\mathbb{R}^2\times[0,1]$:

    \begin{equation*}
        \begin{tikzpicture}
            \begin{knot}[clip width=2]
                \strand[draw=blue, double=white, double distance=5pt] (0,-1.7) -- (0,1.7);
                \strand[draw=black, double=white, double distance=5pt] (-1,0) to [out=down, in=down] (1,0) to [out=up, in=up] (-1,0);
                \flipcrossings{1}
            \end{knot}
            \draw[draw=blue, ->] (0,1) -- (0,1.4);
            \node[] at (-0.3,1.2) {$\textcolour{blue}{Y_1}$};
            \begin{knot}[clip width=2]
                \strand[draw=red, double=white, double distance=5pt] (3,-1.7) -- (3,1.7);
                \strand[draw=black, double=white, double distance=5pt] (2,0) to [out=down, in=down] (4,0) to [out=up, in=up] (2,0);
                \flipcrossings{1}
            \end{knot}
            \draw[draw=red, ->] (3,1) -- (3,1.4);
            \node[] at (2.7,1.2) {$\textcolour{red}{Y_2}$};
            \node[] at (5,0) {$\cdots$};
            \begin{knot}[clip width=2]
                \strand[draw=Green, double=white, double distance=5pt] (7,-1.7) -- (7,1.7);
                \strand[draw=black, double=white, double distance=5pt] (6,0) to [out=down, in=down] (8,0) to [out=up, in=up] (6,0);
                \flipcrossings{1}
            \end{knot}
            \node[] at (6.3,1.2) {$\textcolour{Green}{Y_a}$};
            \draw[draw=Green, ->] (7,1) -- (7,1.4);
        \end{tikzpicture}
    \end{equation*}
    If some $Y_i$ is the tensor product of multiple simple objects, the central coloured ribbon is taken to represent parallel ribbons, each coloured with the relevant simple object.
\end{lem}

In the case where one of the $Y_i$ is the tensor product of two simple objects, this may be simplified even further with the following result:

\begin{lem}[{\cite[Corollary 3.1.10]{BK}}]\label{lem:BK}
    In a modular tensor category, the operator invariants of the following two labelled embedded ribbon graphs in $\mathbb{R}^2\times[0,1]$ are equal:

    \begin{equation*}
        \begin{tikzpicture}
            \begin{knot}[clip width=2]
                \strand[draw=blue, double=white, double distance=5pt] (-0.5,-1.7) -- (-0.5,1.7);
                \strand[draw=Green, double=white, double distance=5pt] (0.5,-1.7) -- (0.5,1.7);
                \strand[draw=black, double=white, double distance=5pt] (-1,0) to [out=down, in=down] (1,0) to [out=up, in=up] (-1,0);
                \flipcrossings{1,3}
            \end{knot}
            \draw[draw=blue, ->] (-0.5,1) -- (-0.5,1.4);
            \draw[draw=Green, <-] (0.5,1) -- (0.5,1.4);
            \node[] at (-0.8,1.2) {$\textcolour{blue}{i}$};
            \node[] at (0.8,1.2) {$\textcolour{Green}{j}$}; 
            \node[] at (2,0) {$=$};
            \node[] at (3,0) {$\delta_{\textcolour{blue}{i}\textcolour{Green}{j}}\frac{D}{\dim \textcolour{blue}{i}}$};
            \draw[draw=blue, double=white, double distance=5pt] (4,-1.7) to [out=up, in=left] (4.5,-0.7) to [out=right, in=up] (5,-1.7);
            \draw[draw=blue, double=white, double distance=5pt] (4,1.7) to [out=down, in=left] (4.5,0.7) to [out=right,in=down] (5,1.7);
            \node[] at (4.5,0.4) {$\textcolour{blue}{i}$};
            \node[] at (4.5,-0.4) {$\textcolour{blue}{i}$};
            \draw[draw=blue, <-] (4.3,0.75) .. controls (4.5,0.65) .. (4.7,0.75);
            \draw[draw=blue, ->] (4.3,-0.75) .. controls (4.5,-0.65) .. (4.7,-0.75);
        \end{tikzpicture}
    \end{equation*}
\end{lem}

\subsection{Reformulating $RT_{\mathcal{C},\sqrt{D}}$ for 2-morphisms}Before proving functoriality, we will first present a slight reformulation of $RT_{\mathcal{C},\sqrt{D}}$ which is more tractable. Let $(M,n)$ be a 2-morphism in $\Bord{sig,dec}$. In this subsection, we further assume that $M$ is connected.

Let $|\Gamma\cup L|$ be a surgery presentation of $(\ol{M},|\Gamma|)$. By applying suitable ambient isotopies, we may obtain a diagram from which it is easier to read off the linear maps defined above:

First, we may replace each $\mathcal{V}_{\mathbf{K}_{\mathbf{i},\mathbf{j}}}(\Gamma_\Sigma)$ with an isomorphic vector space. In each connected component of $\Sigma$, pick a spanning tree of the associated ribbon graph. All of these spanning trees' edges are uncoloured, since the only edges that are coloured by $\mathbb{K}_{\mathbf{i},\mathbf{j}}$ are external ones. Let $\tilde{\Gamma}_\Sigma$ be the result of contracting along these trees, and $\tilde{\mathbf{K}}_{\mathbf{i},\mathbf{j}}$ the induced colouring. Then, as proven in Lemma \ref{lem:tree}, this provides a canonical isomorphism between $\mathcal{V}_{\mathbf{K}_{\mathbf{i},\mathbf{j}}}(\Gamma_\Sigma)$ and $\mathcal{V}_{\tilde{\mathbf{K}}_{\mathbf{i},\mathbf{j}}}(\tilde{\Gamma}_\Sigma)$.

Let $\tilde{K}$ and $\tilde{K}'$ be ribbon-colourings of $\tilde{\Gamma}_\Sigma$ and $\tilde{\Gamma}_{\Sigma'}$ which extend $\tilde{\mathbf{K}}_{\mathbf{i},\mathbf{j}}$ and $\tilde{\mathbf{K}}_{\mathbf{i},\mathbf{j}}'$ respectively. Then, we wish to define a linear map $$(RT_{\mathcal{C},\sqrt{D}}(M,n)_\mathbf{i}^\mathbf{j})_{\tilde{K}}^{\tilde{K}'}: \mathcal{V}_{\tilde{K}}(\tilde{\Gamma}_\Sigma) \rightarrow \mathcal{V}_{\tilde{K}'}(\tilde{\Gamma}_{\Sigma'}).$$

We may do this exactly as before: let $\tilde{\Gamma}_M$ denote the contraction of $\Gamma_M$ by the chosen spanning trees of the components of $\Gamma_\Sigma$ and $\Gamma_{\Sigma'}$. Then, choosing $\tilde{K}$ and $\tilde{K}'$ is exactly the same as choosing a ribbon-colouring $\tilde{K}_M$ of $\tilde{\Gamma}_M$, and we may view $\tilde{\Gamma}_M$ as embedded in $M$ as in Lemma \ref{lem:tree}. A similarly rescaled version of the operator invariant $\tau(M,\tilde{\Gamma}_M,\tilde{K}_M,-)$ as in (\ref{eq:def-rt}) defines the linear map $(RT_{\mathcal{C},\sqrt{D}}(M,n)_\mathbf{i}^\mathbf{j})_{\tilde{K},\tilde{K}'}: \mathcal{V}_{\tilde{K}}(\tilde{\Gamma}_\Sigma)\otimes\mathcal{V}_{\tilde{K}'}(\tilde{\Gamma}_{\Sigma'})\rightarrow k$; note in particular that the $\dim_{\mathrm{int}}(K')$ term is replaced by $\dim_{\mathrm{int}}(\tilde{K}')$. Then, we may define $(RT_{\mathcal{C},\sqrt{D}}(M,n)_\mathbf{i}^\mathbf{j})_{\tilde{K}}^{\tilde{K}'}$ as in (\ref{eq:rt-nodual}), but adding tildes as appropriate.

Now, we want to check that this new definition corresponds to the original definition of $RT_{\mathcal{C},\sqrt{D}}(M)$, upon applying the isomorphism in Lemma \ref{lem:tree}. This will require us to study the interaction between the isomorphisms defined in Lemmas \ref{lem:tree} and \ref{lem:rev-dual}.

Let $\Gamma$ be a ribbon graph with partial ribbon-colouring $\mathbf{K}$, and $T$ a ribbon subtree whose ribbons are uncoloured. Then, one might expect the following diagram to commute:
\[\begin{tikzcd}
	{\mathcal{V}_\mathbf{K}(\Gamma) \otimes \mathcal{V}_\mathbf{K}(-\Gamma)} && \\
	&& {k} \\
    {\mathcal{V}_{\mathbf{K}_{/T}}(\Gamma/T) \otimes \mathcal{V}_{\mathbf{K}_{/T}}(-\Gamma/T)} &&
	\arrow["{\langle\cdot,\cdot\rangle_\Gamma}", from=1-1, to=2-3]
	\arrow["{\bigcirc\otimes\bigcirc}", from=1-1, to=3-1]
	\arrow["{\langle\cdot,\cdot\rangle_{\Gamma/T}}", from=3-1, to=2-3]
\end{tikzcd}\]
Here, the vertical arrow is the isomorphism from Lemma \ref{lem:tree} and the other two arrows are the evaluation maps from Lemma \ref{lem:rev-dual}. This is, however, not exactly right; it turns out to be off by a scalar that depends on the colours of the edges of $T$.

Fix a ribbon-colouring $K$ that extends $\mathbf{K}$. Recall that by Definition \ref{defn:partial-space-labels}, $\mathcal{V}_\mathbf{K}(\Gamma) \otimes \mathcal{V}_\mathbf{K}(-\Gamma)$ is a direct sum of vector spaces. $\mathcal{V}_{K}(\Gamma) \otimes \mathcal{V}_{K}(-\Gamma)$ is one of the summands, and so is a vector subspace. In this subspace, the above diagram fails to commute by a factor of $$\dim(K|_T) = \prod_{e \in E_T} \dim(K(e)),$$ the product of the dimensions of the colours of the edges of $T$. On the other hand, on the subspaces of the form $\mathcal{V}_{K}(\Gamma) \otimes \mathcal{V}_{\ol{K}}(-\Gamma)$ where $K\ne \ol{K}$ extend $\mathbf{K}$, both routes give the zero linear map. In other words, we have the following technical lemma:

\begin{lem}\label{lem:dual-contract}
    Let $K,\ol{K}$ be ribbon-colourings that extend $\mathbf{K}$, so we have an inclusion $\iota: \mathcal{V}_{K}(\Gamma) \otimes \mathcal{V}_{\ol{K}}(-\Gamma) \rightarrow \mathcal{V}_\mathbf{K}(\Gamma) \otimes \mathcal{V}_\mathbf{K}(-\Gamma)$ of vector spaces.
    \begin{enumerate}
        \item If $K=\ol{K}$, the following diagram commutes:
        \begin{equation}\label{eq:dual-contract}
            \begin{tikzcd}
                {\mathcal{V}_{K}(\Gamma) \otimes \mathcal{V}_{K}(-\Gamma)} && {k}\\
                \\
                {\mathcal{V}_{\mathbf{K}_{/T}}(\Gamma/T) \otimes \mathcal{V}_{\mathbf{K}_{/T}}(-\Gamma/T)} && {k}
                \arrow["{\langle\cdot,\cdot\rangle_\Gamma\circ\iota}", from=1-1, to=1-3]
                \arrow["{(\bigcirc\otimes\bigcirc)\circ\iota}", from=1-1, to=3-1]
                \arrow["{\langle\cdot,\cdot\rangle_{\Gamma/T}}", from=3-1, to=3-3]
                \arrow["{\cdot\dim(K|_T)^{-1}}", from=1-3, to=3-3]
            \end{tikzcd}
        \end{equation}
        \item If $K\ne\ol{K}$, then we instead have $$\langle\cdot,\cdot\rangle_{\Gamma/T} \circ (\bigcirc\otimes\bigcirc) \circ \iota = \langle\cdot,\cdot\rangle_\Gamma\circ\iota = 0.$$
    \end{enumerate}
    
\end{lem}
\begin{proof}
    It suffices to prove this in the case where $T$ is a single edge $e$ and $\mathbf{K}$ colours all edges except for $e$. Then, we may take direct sums for general partial ribbon-colourings $\mathbf{K}$ and induct over edge contractions for a general tree $T$. Let $K,\ol{K}$ extend $\mathbf{K}$ by colouring the last edge with colours $i,j$. Note that if $i\ne j$, then we automatically have $\langle\cdot,\cdot\rangle_\Gamma = 0$, by definition.

    Let the endpoints of $e$ be $v=h(e),v'=t(e)$. Comparing the clockwise and anticlockwise routes in (\ref{eq:dual-contract}), the terms corresponding to all vertices except for $v,v'$ cancel out. We are thus left with showing that for all $f\in \mathcal{V}_K(\Gamma,v), g\in \mathcal{V}_{\ol{K}}(-\Gamma,v)$ and $f'\in \mathcal{V}_K(\Gamma,v'),g'\in \mathcal{V}_{\ol{K}}(-\Gamma,v')$ we have $$\tr((f'f)(gg')) = \delta_{ij}\frac1{\dim i}\tr(fg)\tr(f'g').$$

    Indeed, consider the following equality of operator invariants of ribbon diagrams:
    \begin{equation*}
        \begin{tikzpicture}
            \begin{knot}[clip width=2]
                \strand[draw=blue, double=white, double distance=5pt] (-1,-1.8) -- (-1,1.8);
                \strand[draw=Green, double=white, double distance=5pt] (1,-1.8) -- (1,1.8);
                \strand[draw=red, double=white, double distance=5pt] (-1.4,2.2) to [out=up, in=left] (0,3.6) to [out=right, in=up] (1.4,2.2);
                \strand[draw=red, double=white, double distance=5pt] (-0.6,2.2) to [out=up, in=left] (0,2.8) to [out=right, in=up] (0.6,2.2);
                \strand[draw=violet, double=white, double distance=5pt] (-0.6,-2.2) to [out=down, in=left] (0,-2.8) to [out=right, in=down] (0.6,-2.2);
                \strand[draw=violet, double=white, double distance=5pt] (-1.4,-2.2) to [out=down, in=left] (0,-3.6) to [out=right, in=down] (1.4,-2.2);
                \strand[draw=black, double=white, double distance=5pt] (-1.5,0) to [out=down, in=down] (1.5,0) to [out=up, in=up] (-1.5,0);
                \flipcrossings{1,3}
            \end{knot}
            \node[] at (0,3.3) {$\vdots$};
            \node[] at (0,-3.1) {$\vdots$};
            \draw[rounded corners] (-1.6,-2.2) rectangle (-0.4,-1.8) node[pos=.5] {$f$};
            \draw[rounded corners] (-1.6,1.8) rectangle (-0.4,2.2) node[pos=.5] {$f'$};
            \draw[rounded corners] (0.4,-2.2) rectangle (1.6,-1.8) node[pos=.5] {$g$};
            \draw[rounded corners] (0.4,1.8) rectangle (1.6,2.2) node[pos=.5] {$g'$};
            \draw[draw=blue, ->] (-1,1.1) -- (-1,1.5);
            \draw[draw=Green, <-] (1,1.1) -- (1,1.5);
            \node[] at (-1.3,1.3) {$\textcolour{blue}{i}$};
            \node[] at (1.3,1.3) {$\textcolour{Green}{j}$};
            \node[] at (-2,0) {$\frac1{D}$};
            \node[] at (-2.5,0) {$=$};
            \node[] at (2.5,0) {$=$};

            \draw[draw=blue, double=white, double distance=5pt] (-6,-1.8) -- (-6,1.8);
            \draw[draw=Green, double=white, double distance=5pt] (-4,-1.8) -- (-4,1.8);
            \draw[draw=red, double=white, double distance=5pt] (-6.4,2.2) to [out=up, in=left] (-5,3.6) to [out=right, in=up] (-3.6,2.2);
            \draw[draw=red, double=white, double distance=5pt] (-5.6,2.2) to [out=up, in=left] (-5,2.8) to [out=right, in=up] (-4.4,2.2);
            \draw[draw=violet, double=white, double distance=5pt] (-5.6,-2.2) to [out=down, in=left] (-5,-2.8) to [out=right, in=down] (-4.4,-2.2);
            \draw[draw=violet, double=white, double distance=5pt] (-6.4,-2.2) to [out=down, in=left] (-5,-3.6) to [out=right, in=down] (-3.6,-2.2);
            \node[] at (-5,3.3) {$\vdots$};
            \node[] at (-5,-3.1) {$\vdots$};
            \draw[rounded corners] (-6.6,-2.2) rectangle (-5.4,-1.8) node[pos=.5] {$f$};
            \draw[rounded corners] (-6.6,1.8) rectangle (-5.4,2.2) node[pos=.5] {$f'$};
            \draw[rounded corners] (-4.6,-2.2) rectangle (-3.4,-1.8) node[pos=.5] {$g$};
            \draw[rounded corners] (-4.6,1.8) rectangle (-3.4,2.2) node[pos=.5] {$g'$};
            \draw[draw=blue, ->] (-6,1.1) -- (-6,1.5);
            \draw[draw=Green, <-] (-4,1.1) -- (-4,1.5);
            \node[] at (-6.3,1.3) {$\textcolour{blue}{i}$};
            \node[] at (-3.7,1.3) {$\textcolour{Green}{j}$};

            \draw[draw=red, double=white, double distance=5pt] (3.6,2.2) to [out=up, in=left] (5,3.6) to [out=right, in=up] (6.4,2.2);
            \draw[draw=red, double=white, double distance=5pt] (4.4,2.2) to [out=up, in=left] (5,2.8) to [out=right, in=up] (5.6,2.2);
            \draw[draw=violet, double=white, double distance=5pt] (4.4,-2.2) to [out=down, in=left] (5,-2.8) to [out=right, in=down] (5.6,-2.2);
            \draw[draw=violet, double=white, double distance=5pt] (3.6,-2.2) to [out=down, in=left] (5,-3.6) to [out=right, in=down] (6.4,-2.2);
            \draw[draw=blue, double=white, double distance=5pt] (4,1.8) to [out=down, in=left] (5,0.8) to [out=right, in=down] (6,1.8);
            \draw[draw=blue, double=white, double distance=5pt] (4,-1.8) to [out=up, in=left] (5,-0.8) to [out=right, in=up] (6,-1.8);
            \node[] at (5,3.3) {$\vdots$};
            \node[] at (5,-3.1) {$\vdots$};
            \draw[rounded corners] (3.4,-2.2) rectangle (4.6,-1.8) node[pos=.5] {$f$};
            \draw[rounded corners] (3.4,1.8) rectangle (4.6,2.2) node[pos=.5] {$f'$};
            \draw[rounded corners] (5.4,-2.2) rectangle (6.6,-1.8) node[pos=.5] {$g$};
            \draw[rounded corners] (5.4,1.8) rectangle (6.6,2.2) node[pos=.5] {$g'$};
            \node[] at (5,0.5) {$\textcolour{blue}{i}$};
            \node[] at (5,-0.5) {$\textcolour{blue}{i}$};
            \draw[draw=blue, <-] (4.7,0.85) .. controls (5,0.75) .. (5.3,0.85);
            \draw[draw=blue, ->] (4.7,-0.85) .. controls (5,-0.75) .. (5.3,-0.85);
            \node[] at (3.1,0) {$\frac{\delta_{\textcolour{blue}{i}\textcolour{Green}{j}}}{\dim\textcolour{blue}{i}}$};
        \end{tikzpicture}
    \end{equation*}
    For the first equality, note that we may apply an isotopy to the middle diagram to obtain the ribbon diagram on the left-hand side along with an uncoloured unknot. The latter has operator invariant $D$. For the second equality, we are using Lemma \ref{lem:BK}.
\end{proof}

Recall that $\tilde{\Gamma}_M$ is the contraction of $\Gamma_M$ by the chosen spanning trees of the components of $\Gamma_\Sigma,\Gamma_{\Sigma'}$. We may now show that the definition of $RT_{\mathcal{C},\sqrt{D}}$ using $\tilde{\Gamma}_M$ agrees with the definition using $\Gamma_M$.

\begin{lem}\label{lem:contract-tree-rt}
    The definitions of $RT_{\mathcal{C},\sqrt{D}}(M)$ using $\Gamma_M$ and $\tilde{\Gamma}_M$ coincide. More precisely, let $\tilde{K}$ and $\tilde{K}'$ be ribbon-colourings of $\tilde{\Gamma}_\Sigma$ and $\tilde{\Gamma}_{\Sigma'}$ which extend $\tilde{\mathbf{K}}_{\mathbf{i},\mathbf{j}}$ and $\tilde{\mathbf{K}}_{\mathbf{i},\mathbf{j}}'$ respectively. These induce partial ribbon-colourings $\mathbf{K},\mathbf{K}'$ of $\Gamma_\Sigma,\Gamma_{\Sigma'}$ that extend $\mathbf{K}_{\mathbf{i},\mathbf{j}},\mathbf{K}_{\mathbf{i},\mathbf{j}}'$ respectively. Let $\mathbb{K},\mathbb{K}'$ denote the sets of ribbon-colourings of $\Gamma_\Sigma,\Gamma_{\Sigma'}$ that extend $\mathbf{K},\mathbf{K}'$ respectively. Then, the following diagram commutes:
    \[\begin{tikzcd}
        {\mathcal{V}_{\mathbf{K}}(\Gamma_\Sigma)} & {\displaystyle\bigoplus_{K \in \mathbb{K}} \mathcal{V}_K(\Gamma_\Sigma)} &&& {\displaystyle\bigoplus_{K' \in \mathbb{K}'} \mathcal{V}_{K'}(\Gamma_{\Sigma'})} & {\mathcal{V}_{\mathbf{K}'}(\Gamma_{\Sigma'})} \\
        \\
        {\mathcal{V}_{\tilde{K}}(\tilde{\Gamma}_\Sigma)} &&&&& {\mathcal{V}_{\tilde{K}'}(\tilde{\Gamma}_{\Sigma'})}
        \arrow["{\sum_{K,K'}(RT_{\mathcal{C},\sqrt{D}}(M)_\mathbf{i}^\mathbf{j})_{K}^{K'}}", from=1-2, to=1-5]
        \arrow["{\bigcirc}", from=1-1, to=3-1]
        \arrow["{\bigcirc}", from=1-6, to=3-6]
        \arrow["{(RT_{\mathcal{C},\sqrt{D}}(M)_\mathbf{i}^\mathbf{j})_{\tilde{K}}^{\tilde{K}'}}", from=3-1, to=3-6]
        \arrow[equals, from=1-1, to=1-2]
        \arrow[equals, from=1-5, to=1-6]
    \end{tikzcd}\]
\end{lem}
\begin{proof}
    Any choice of $K\in\mathbb{K},K'\in\mathbb{K}'$ determines a ribbon-colouring $K_M$ of $\Gamma_M$ which extends the partial ribbon-colouring $\mathbf{K}_M$ given by $\mathbf{K},\mathbf{K}'$. We thus have an inclusion linear map $\iota:\mathcal{V}_{K_M}(\Gamma_M) \rightarrow \mathcal{V}_{\mathbf{K}_M}(\Gamma_M)$.
    
    By Lemma \ref{lem:tree}(2), the following diagram commutes ($T'$ denotes the chosen spanning trees of $\Gamma_{\Sigma'}$):
    \[\begin{tikzcd}
        {\mathcal{V}_{K_M}(\Gamma_M)} &&&& {k} \\
        \\
        {\mathcal{V}_{\tilde{K}_M}(\tilde{\Gamma}_M)} &&&& {k}
        \arrow["{\dim_{\mathrm{int}}(K')\tau(M,\Gamma_M,K_M,-)}", from=1-1, to=1-5]
        \arrow["{\bigcirc\circ\iota}", from=1-1, to=3-1]
        \arrow["{\cdot \dim(K'|_{T'})^{-1}}", from=1-5, to=3-5]
        \arrow["{\dim_{\mathrm{int}}(\tilde{K}')\tau(M,\tilde{\Gamma}_M,\tilde{K}_M,-)}", from=3-1, to=3-5]
    \end{tikzcd}\]
    Thus, we have
    \begin{equation} \label{eq:contract-tree-intermediate}
        (RT_{\mathcal{C},\sqrt{D}}(M,n)_\mathbf{i}^\mathbf{j})_{\tilde{K},\tilde{K}'}((\bigcirc\mathbf{f})\otimes(\bigcirc\mathbf{g})) = \frac1{\dim(K'|_{T'})}(RT_{\mathcal{C},\sqrt{D}}(M,n)_\mathbf{i}^\mathbf{j})_{K,K'}(\mathbf{f}\otimes\mathbf{g})
    \end{equation}
    for each $\mathbf{f}\otimes\mathbf{g}\in\mathcal{V}_K(\Gamma_\Sigma)\otimes\mathcal{V}_{K'}(-\Gamma_{\Sigma'})\cong\mathcal{V}_{K_M}(\Gamma_M)$. Note here that under the isomorphism $\mathcal{V}_{\tilde{K}}(\tilde{\Gamma}_\Sigma)\otimes\mathcal{V}_{\tilde{K}'}(-\tilde{\Gamma}_{\Sigma'})\cong\mathcal{V}_{\tilde{K}_M}(\tilde{\Gamma}_M)$, $(\bigcirc\mathbf{f})\otimes(\bigcirc\mathbf{g})$ is sent to $\bigcirc(\mathbf{f}\otimes\mathbf{g})$.

    On the other hand, by Lemma \ref{lem:dual-contract}, we have $$\left\langle \bigcirc(RT_{\mathcal{C},\sqrt{D}}(M,n)_\mathbf{i}^\mathbf{j})_K^{K'}(\mathbf{f}), \bigcirc\mathbf{g}\right\rangle_{\Gamma/T} = \frac{\delta_{K'\ol{K'}}}{\dim(K'|_{T'})}\left\langle (RT_{\mathcal{C},\sqrt{D}}(M,n)_\mathbf{i}^\mathbf{j})_K^{K'}(\mathbf{f}), \mathbf{g}\right\rangle_{\Gamma}$$ for each $\mathbf{f}\in\mathcal{V}_K(\Gamma_\Sigma),\mathbf{g}\in\mathcal{V}_{\ol{K'}}(\Gamma_{\Sigma'})$ for each $K\in\mathbb{K}$ and $K',\ol{K'}\in\mathbb{K}'$.

    Summing over $K'\in\mathbb{K}'$ and applying (\ref{eq:rt-nodual}) (and the version with tildes) to compare the result to (\ref{eq:contract-tree-intermediate}), we thus obtain that $$\left\langle\bigcirc\sum_{K'\in\mathbb{K}'}(RT_{\mathcal{C},\sqrt{D}}(M,n)_\mathbf{i}^\mathbf{j})_K^{K'}(\mathbf{f}),\bigcirc\mathbf{g}\right\rangle_{\Gamma/T} = \left\langle(RT_{\mathcal{C},\sqrt{D}}(M,n)_\mathbf{i}^\mathbf{j})_{\tilde{K}}^{\tilde{K}'}(\bigcirc\mathbf{f}),\bigcirc\mathbf{g}\right\rangle_{\Gamma/T}$$ for all $\mathbf{f}\in\mathcal{V}_K(\Gamma_\Sigma),\mathbf{g}\in\mathcal{V}_{\ol{K'}}(\Gamma_{\Sigma'})$. By varying $\ol{K'}\in\mathbb{K}'$, we may replace $\bigcirc\mathbf{g}$ with an arbitrary element of $\mathcal{V}_{\tilde{K}'}(\tilde{\Gamma}_{\Sigma'})$, and hence we have $$\bigcirc\sum_{K'\in\mathbb{K}'}(RT_{\mathcal{C},\sqrt{D}}(M,n)_\mathbf{i}^\mathbf{j})_K^{K'}(\mathbf{f}) = (RT_{\mathcal{C},\sqrt{D}}(M,n)_\mathbf{i}^\mathbf{j})_{\tilde{K}}^{\tilde{K}'}(\bigcirc\mathbf{f})$$ for each $\mathbf{f}\in\mathcal{V}_K(\Gamma_\Sigma) \subseteq \mathcal{V}_\mathbf{K}(\Gamma_\Sigma)$ for all $K\in\mathbb{K}$, which is what we wanted.
\end{proof}

So, instead of the surgery presentation $|\Gamma_M\cup L|$, we may contract along spanning trees and consider $|\tilde{\Gamma}_M\cup L|$ where the embedding of $\tilde{\Gamma}_M$ is given by that of $\Gamma_M$. Given the embedded surface $|\tilde{\Gamma}_M\cup L|$, we shall now apply an ambient isotopy to convert it into a form from which $RT_{\mathcal{C},\sqrt{D}}$ can easily be read off.

\begin{constr}\label{constr:emb}
    Let $V = V_\mathrm{in} \sqcup V_\mathrm{out}$ be the vertices of $\tilde{\Gamma}_M$; $V_\mathrm{in}$ corresponds exactly to the connected components of $\parin M$, and likewise for $V_\mathrm{out}$ and $\parout M$. $\Gamma_M\setminus V$ may be viewed as a ribbon subgraph whose edges all have source and target $\infty$.

    Pick an ordering $(v_1,\ldots,v_a)$ of $V_\mathrm{in}$ and $(v_1',\ldots,v_b')$ of $V_\mathrm{out}$. For each $v_i$ and each $v_j'$, fix a way to break the cyclic ordering of its edges into a total ordering (corresponding to letting $s(v_i) = \emptyset, t(v_j'). = \emptyset$ for each $i,j$).

    Now, we may apply an ambient isotopy such that the embedded surface $|\tilde{\Gamma}_M\cup L| \subset \mathbb{R}^3 \subset S^3$ satisfies the following conditions:
    \begin{itemize}
        \item The coupons corresponding to each vertex are embedded in $\mathbb{R}^2\times\set{0}$.
        \item $|(\tilde{\Gamma}_M\setminus V)\cup L|$ is embedded in $\mathbb{R}\times [0,1]\times\mathbb{R}$, with the endpoints corresponding to the edges incident at $V_\mathrm{in}$ in $\mathbb{R}\times\set{0}\times\set{0}$ and those corresponding to edges incident at $V_\mathrm{out}$ in $\mathbb{R}\times\set{1}\times\set{0}$.
        \item The order of the endpoints in $\mathbb{R}\times\set{0}\times\set{0}$ and $\mathbb{R}\times\set{1}\times\set{0}$ is given by the ordering of the $v_i$ and $v_j'$, followed by the ordering of the edges at each vertex.
        \item The projection given by forgetting the last coordinate gives a valid ribbon diagram.
    \end{itemize}
    \begin{figure}[hbt!]
        \begin{tikzpicture}
            \begin{knot}[consider self intersections, end tolerance=2pt, clip radius=5pt, clip width=2]
                \strand[draw=blue, double=white, double distance=5pt] (-0.6,-3) to [out=up, in=down] (-2.2,0) to [out=up,in=left] (-1,0.5) to [out=right,in=up] (-0.3,0) to [out=down,in=up] (0.3,-1) to [out=down,in=up] (-0.3,-2) to [out=down,in=up] (0.2,-3);
                \strand[draw=blue, double=white, double distance=5pt] (0.6,-3) to [out=up, in=down] (1.2,0) to [out=up,in=up] (0.3,0) to [out=down,in=up] (-0.3,-1) to [out=down,in=up] (0.3,-2) to [out=down,in=up] (-0.2,-3);
                \strand[draw=blue, double=white, double distance=5pt] (-3.4,-3) to [out=up,in=left] (-2,1.5) to [out=right,in=up] (-1.5,0) to [out=down,in=up] (-2.6,-3);
                \strand[draw=blue, double=white, double distance=5pt] (2.6,-3) to [out=up,in=down] (2,-2) to [out=up,in=left] (4,0) to [out=right,in=right] (3,-1.5) to [out=left,in=left] (2,0) to [out=right,in=up] (4,-2) to [out=down,in=up] (3.4,-3);
                \strand[draw=red, double=white, double distance=5pt] (-2.1,3) to [out=down,in=left] (-2,1) to [out=right,in=down] (-0.9,3);
                \strand[draw=red, double=white, double distance=5pt] (-1.7,3)  to [out=down,in=left] (-1.7,1.8) to [out=right,in=down] (-1.3,3);
                \strand[draw=red, double=white, double distance=5pt] (1.3,3) to [out=down,in=left] (1.5,1) to [out=right,in=down] (1.7,3);
                \strand[draw=black, double=white, double distance=5pt] (3.2,-2.2) to [out=down,in=left] (4.2,-2.7) to [out=right,in=down] (5.2,-2.2) to [out=up,in=right] (4.2,-1.7) to [out=left,in=up] (3.2,-2.2);
                \strand[draw=black, double=white, double distance=5pt] (0.5,0.5) to [out=down,in=left] (1.5,-0.5) to [out=right,in=down] (2.5,0.9) to [out=up,in=right] (1.5,2.2) to [out=left,in=right] (0,2.5) to [out=left,in=up] (0.5,0.5);
                \strand[draw=black, double=white, double distance=5pt] (-1.55,2.5) to [out=down,in=left] (-0.5,2.1) to [out=right,in=down] (0.5,2.5) to [out=up,in=right] (-0.5,2.8) to [out=left,in=up] (-1.55,2.5);
                \strand[draw=Green, double=white, double distance=5pt] (-3,-3) to [out=up,in=left] (-1.5,-2) to [out=right,in=down] (-0.9,-1) to [out=up,in=down] (-0.9,0.8) to [out=up,in=left] (0.3,1.5) to [out=right,in=down] (0.9,3);
                \strand[draw=Green, double=white, double distance=5pt] (2.1,3) to [out=down,in=up] (3,2.2) to [out=down,in=left] (3.5,1.5) to [out=right,in=down] (5,2) to [out=up,in=up] (3.9,0.4) to [out=down,in=up] (3,-1) to [out=down,in=up] (2.8,-2) to [out=down,in=up] (3,-3);
                \flipcrossings{3,6,9,10,12,14,15,17,18,22,24,26,27,31,32}
            \end{knot}
            \draw[rounded corners] (-0.8,-3.6) rectangle (0.8,-3);
            \draw[rounded corners] (-3.8,-3.6) rectangle (-2.2,-3);
            \draw[rounded corners] (2.2,-3.6) rectangle (3.8,-3);
            \draw[rounded corners] (-2.3,3) rectangle (-0.7,3.6);
            \draw[rounded corners] (0.7,3) rectangle (2.3,3.6);
        \end{tikzpicture}
        \caption{Embedding of $|\tilde{\Gamma}_M\cup L|$.} \label{fig:gamma-m}
    \end{figure}
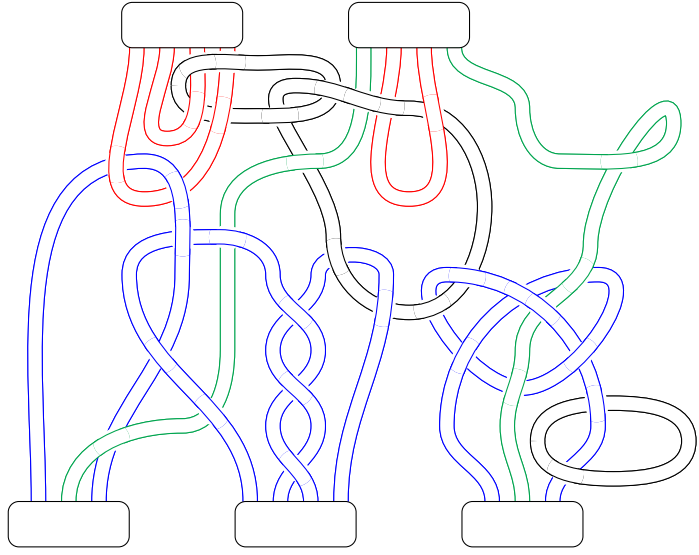
    An example is shown in Figure \ref{fig:gamma-m}. In this example, the components of $\Sigma=\parin M$ are two genus $1$ surfaces with one boundary component and one closed genus $2$ surface. The bottom coupons correspond to the vertices of $\Gamma_\Sigma$, and the blue ribbons to its internal edges. The green ribbons correspond to its external edges, which are joined up with the corresponding external edges of $\Gamma_{\Sigma'}$. The uncoloured (black) ribbons form the surgery link $L$. Note that the incoming edges to the vertices in $V_\mathrm{out}$ are ordered \textit{oppositely} to the cyclic ordering, as they are taken to be the sources of the respective vertices.
\end{constr}

When presented in this way, this provides an explicit description for $(RT_{\mathcal{C},\sqrt{D}}(M)_\mathbf{i}^\mathbf{j})_{\tilde{K}}^{\tilde{K}'}$, in terms of $F_*$ applied to $(\tilde{\Gamma}_M\setminus V)\cup L$. Indeed, note that the choices of $\mathbf{i},\mathbf{j}$ determine a colouring for the green ribbons, while $\tilde{K},\tilde{K}'$ provide a colouring for the blue and red ribbons respectively, and so these make $(\tilde{\Gamma}_M\setminus V)\cup L$ the union of a labelled ribbon graph and some uncoloured loops.

We seek to define $$(RT_{\mathcal{C},\sqrt{D}}(M)_\mathbf{i}^\mathbf{j})_{\tilde{K}}^{\tilde{K}'}: \bigotimes_{p=1}^a \Hom\pa{\mathbbm{1},\bigotimes_{q=1}^{\mathrm{deg}(v_p)} X_{i_{p,q}}^{\epsilon_{i_{p,q}}}} \rightarrow \bigotimes_{p=1}^b \Hom\pa{\mathbbm{1},\bigotimes_{q=1}^{\mathrm{deg}(v_p')} X_{j_{p,q}}'^{\epsilon_{j_{p,q}}'}}$$ where the $X_{i_{p,q}},{\epsilon_{i_{p,q}}}$ and $X_{j_{p,q}}',{\epsilon_{j_{p,q}}'}$ encode the colours and directions of the ribbons at the top and bottom edges respectively. The source and target of this map are, however, not the same vector spaces as the source and target of $F_*((\tilde{\Gamma}_M\setminus V)\cup L, \tilde{K}_M, \emptyset)$. To make up for this discrepancy, we have to use the canonical inclusion and projection from Lemma \ref{lem:tens-hom} in the case where 
\begin{equation} \label{eq:def-yp}
    Y_p = \bigotimes_{q=1}^{\mathrm{deg}(v_p)} X_{i_{p,q}}^{\epsilon_{i_{p,q}}}, \quad Y_p' = \bigotimes_{q=1}^{\mathrm{deg}(v_p')} X_{j_{p,q}}'^{\epsilon_{j_{p,q}}'}
\end{equation}
respectively.

\begin{lem}\label{lem:rt-formula}
    Consider the linear map $$\Phi := F_*((\tilde{\Gamma}_M\setminus V)\cup L, \tilde{K}_M, \emptyset): \Hom\pa{\mathbbm{1},\bigotimes_{j=1}^a Y_p} \rightarrow \Hom\pa{\mathbbm{1},\bigotimes_{j=1}^b Y_p'}$$ along with the canonical inclusion $$\varphi: \bigotimes_{j=1}^a \Hom\pa{\mathbbm{1},Y_p} \rightarrow \Hom\pa{\mathbbm{1},\bigotimes_{j=1}^a Y_p}$$ and canonical projection $$\psi': \Hom\pa{\mathbbm{1},\bigotimes_{j=1}^b Y_p'} \rightarrow \bigotimes_{j=1}^b\Hom\pa{\mathbbm{1}, Y_p'}$$ given in Lemma \ref{lem:tens-hom}. These satisfy
    \begin{equation}\label{eq:rt-reform}
        (RT_{\mathcal{C},\sqrt{D}}(M,n)_\mathbf{i}^\mathbf{j})_{\tilde{K}}^{\tilde{K}'} = \pa{\frac{p_-}{\sqrt{D}}}^{\sigma(X_L)-n} \sqrt{D}^{\chi\pa{\ol{\Sigma'}}/2-\# L-1}\dim_\mathrm{int}(\tilde{K}')\psi'\Phi\varphi.
    \end{equation}
\end{lem}
\begin{proof}
    Let $\mathbf{f} = f_1\otimes\cdots\otimes f_a \in \bigotimes_{p=1}^a \Hom(\mathbbm{1},Y_p)$ and $\mathbf{g} = g_1\otimes\cdots\otimes g_b \in \bigotimes_{p=1}^b \Hom(Y_p',\mathbbm{1})$. Choosing these makes $\tilde{\Gamma}_M$ into a labelled embedded ribbon graph in $\ol{M}$. Then by definition of $\tau$, this exactly satisfies $$\tau(\ol{M},\tilde{\Gamma}_M,\tilde{K}_M,\mathbf{f}\otimes\mathbf{g}) = p_-^{\sigma(X_L)}\sqrt{D}^{-\sigma(X_L)-\# L -1}F(\tilde{\Gamma}_M \cup L,\tilde{K}_M,\mathbf{f}\otimes\mathbf{g})$$ where $L$ is the (uncoloured) surgery link. Then, the version of (\ref{eq:rt-nodual}) with tildes may be rewritten as
    \begin{align*}
        &\left\langle(RT_{\mathcal{C},\sqrt{D}}(M,n)_\mathbf{i}^\mathbf{j})_{\tilde{K}}^{\tilde{K}'}(\mathbf{f}),\mathbf{g}\right\rangle_{\Gamma_{\Sigma'}} \\ &\quad= \pa{\frac{p_-}{\sqrt{D}}}^{\sigma(X_L)-n}\sqrt{D}^{\chi(\ol{\Sigma'})/2-\#L-1} \dim_\mathrm{int}(\tilde{K}') F(\tilde{\Gamma}_M\cup L, \tilde{K}_M,\mathbf{f}\otimes\mathbf{g}).
    \end{align*}
    Hence, in order to prove (\ref{eq:rt-reform}), it suffices to show that 
    \begin{equation} \label{eq:reform-toshow}
        \left\langle\psi'\Phi\varphi(\mathbf{f}),\mathbf{g}\right\rangle_{\Gamma_{\Sigma'}} = F(\tilde{\Gamma}_M\cup L, \tilde{K}_M,\mathbf{f}\otimes\mathbf{g})
    \end{equation}
    for all $\mathbf{f}\in \bigotimes_{p=1}^a \Hom(\mathbbm{1},Y_p),\mathbf{g}\in \bigotimes_{p=1}^b \Hom(Y_p',\mathbbm{1})$.

    By the definition of $\psi$, the left-hand side of (\ref{eq:reform-toshow}) may be rewritten as $ \left\langle\Phi\varphi(\mathbf{f}), \psi'^*(\mathbf{g})\right\rangle_{\Gamma_{\Sigma'}}$ where $\psi'^*$ is the inclusion $\bigotimes_{j=1}^b \Hom(Y_p',\mathbbm{1})\rightarrow \Hom\pa{\bigotimes_{j=1}^b Y_p',\mathbbm{1}}$. When written in this way, this is exactly $F(\tilde{\Gamma}_M\cup L, \tilde{K}_M,\mathbf{f}\otimes\mathbf{g})$, and so (\ref{eq:reform-toshow}) holds for all $\mathbf{f},\mathbf{g}$, which is what we wanted.
\end{proof}

\subsection{Comparison to the Reshetikhin--Turaev TQFT} \label{subsection:rt-compare}

Using the results of the previous subsection, we are now able to compare our constructed $RT_{\mathcal{C},\sqrt{D}}$ to the (projective) TQFT constructed in \cite{RT,Turaev-book}. After replacing $\bord{or}$ with $\bord{sig}$ to resolve the projectivity, the Reshetikhin--Turaev TQFT is a symmetric monoidal functor $rt_{\mathcal{C},\sqrt{D}}: \bord{sig,param}\rightarrow \vect$ whose source category $\bord{sig,param}$ is a model of $\bord{sig}$ whose objects are \emph{parametrised surfaces}, i.e. surfaces with a chosen identification to a certain \emph{standard surface}. Effectively, this equips each object $(\Sigma,H_\Sigma)$ with an embedded ribbon graph $\Gamma(g)$ in each genus $g$ connected component of $H_\Sigma$ which consists of a single vertex $v$ and $g$ edges $e_1,\ldots,e_g$. The cyclic order at $v$ is given by $t(e_1),h(e_1),\ldots,t(e_g),h(e_g)$.

For our model $\Bord{sig,dec}$ as given in Definition \ref{defn:bordsig-new}, the hom category from $(\emptyset,\emptyset)$ to itself is given by $$\bord{sig,dec} := \Hom_{\Bord{sig,dec}}((\emptyset,\emptyset),(\emptyset,\emptyset)).$$ This admits an equivalent symmetric monoidal subcategory $\bord{sig,dec'}$ whose objects are compositions of a cap, $g$ many copants-pants pairs, followed by a cup. (Explicitly, these are the 1-morphisms in $\Bord{sig,dec}$ from $(\emptyset,\emptyset)$ to itself in the normal form as defined in the proof of Lemma \ref{lem:all-diffeos}.) For such a 1-morphism in $\Bord{sig,dec}$, contracting the associated ribbon graph along the spanning tree that contains all the edges going through the left legs of all the pants and copants yields the ribbon graph $\Gamma(g)$. This determines an equivalence of symmetric monoidal categories $G: \bord{sig,dec'} \rightarrow \bord{sig,param}$.

Then, the reformulation of $RT_{\mathcal{C},\sqrt{D}}$ in Lemma \ref{lem:rt-formula} corresponds to the formula at the end of \cite[\textsection IV.2.4]{Turaev-book}, and hence we have $RT_{\mathcal{C},\sqrt{D}} = rt_{\mathcal{C},\sqrt{D}} \circ G$.
\begin{rk}
    There is another definition of $rt_{\mathcal{C},\sqrt{D}}$ given in \cite[\textsection IV.1.8]{Turaev-book} in terms of a dual formulation analogous to (\ref{eq:def-rt}). This differs from (\ref{eq:def-rt}) by a $\dim_\mathrm{int}$ term, likely because of a difference of choice of normalisation in Lemma \ref{lem:rev-dual}.
\end{rk}

More generally, the Reshetikhin--Turaev TQFT is also defined on surfaces decorated by marked arcs and 3-dimensional bordisms with labelled embedded ribbon graphs whose endpoints coincide with the marked arcs. For a genus $g$ 1-morphism $\Sigma$ in $\Bord{sig,dec}$ in normal form (in the sense of the proof of Lemma \ref{lem:all-diffeos}) from $(\emptyset,\emptyset)$ to $(\mathbb{S},\mathbb{D})^{\sqcup n}$, contracting the assoiated ribbon graph along the spanning tree through all edges going through the left legs of all the pants and copants yields the ribbon graph which Reshetikhin and Turaev associate to a standard surface decorated with $n$ disjoint arcs \cite[Figure IV.1.1]{Turaev-book}. Given a choice of index $\mathbf{j} = (j_1,\ldots,j_n)$, $RT_{\mathcal{C},\sqrt{D}}(\Sigma)^\mathbf{j}$ is canonically identified with the vector space that Reshetikhin and Turaev assign to the closed genus $g$ surface with $n$ arcs marked by objects $X_{j_1},\ldots,X_{j_n}$. For a 2-morphism $(M,n)$ between two such 1-morphisms, our definition of $RT_{\mathcal{C},\sqrt{D}}(M,n)^\mathbf{j}$ then agrees with the linear map which Reshetikhin and Turaev assign to the cobordism $(M\cup_{\parOut M} (-V_{\parOut M}),n)$ with embedded ribbon graph given by the ribbons through the central axes of $-V_{\parOut M}$, coloured by $j_1,\ldots,j_n$.

\subsection{Vertical composition of 2-morphisms}

Having reformulated $RT_{\mathcal{C},\sqrt{D}}$ in terms of $F_*$ of a given ribbon graph as in Lemma \ref{lem:rt-formula}, we may proceed along similar lines to the proof of \cite[Lemma 2.1.2]{Turaev-book} to show that $RT_{\mathcal{C},\sqrt{D}}$ is compatible with vertical composition. 

Consider two vertically composable 2-morphisms in $\Bord{sig,dec}$, $(M,n):\Sigma\rightarrow\Sigma'$ and $(M',n'):\Sigma'\rightarrow\Sigma''$. Here, we use $\Sigma$, etc. to refer to the underlying surface for the respective 1-morphisms, though we understand these to also have extra structure as described in Definition \ref{defn:bordsig-new}. Let $M'' = M\cup_{\Sigma'} M'$ be the underlying 3-manifold for their composition.

First, we shall reduce to the case where $M,M'$ are connected. This will be done by replacing the $M,M'$ with the connect sum of their components, which scales both sides of the equation by the same factor:
\begin{lem}\label{lem:add-1handle}
    Let $(M,n)$ be a 2-morphism in $\Bord{sig}$. Let $M_+$ be $M$ with a 1-handle attached, i.e. by taking the connect sum of two components, or by replacing one component with its connect sum with $S^1\times S^2$. Then $$RT_{\mathcal{C},\sqrt{D}}(M_+,n) = \sqrt{D}RT_{\mathcal{C},\sqrt{D}}(M,n).$$
\end{lem}
\begin{proof}
    We look at how (\ref{eq:def-rt}) changes when $M$ is replaced with $M_+$. The only term which is affected is the $\tau(\ol{M},\Gamma_M)$ term. If we replace two connected components with their connect sum, then by Lemma \ref{lem:tau-properties}(2), $\tau(\ol{M},\Gamma_M)$ is multiplied by a factor of $\sqrt{D}$. For the other case, by both parts of Lemma \ref{lem:tau-properties}, we have $\tau(\ol{M}\#(S^1\times S^2),\Gamma_M) = \sqrt{D}\tau(\ol{M},\Gamma_M)$.
\end{proof}

Henceforth, we may thus assume that $M,M'$ are connected, so $M''$ is as well.

Choose spanning trees for each of the connected components of $\Gamma_\Sigma,\Gamma_{\Sigma'},\Gamma_{\Sigma''}$, then we may apply the arguments in the previous subsections to compute the composition $$(RT_{\mathcal{C},\sqrt{D}}(M',n')\circ RT_{\mathcal{C},\sqrt{D}}(M,n))_\mathbf{i}^\mathbf{j} = RT_{\mathcal{C},\sqrt{D}}(M',n')_\mathbf{i}^\mathbf{j}\circ RT_{\mathcal{C},\sqrt{D}}(M,n)_\mathbf{i}^\mathbf{j}.$$ Recall here that vertical composition of 2-morphisms in $\KV$ is done entrywise. As before, we may fix ribbon-colourings $\tilde{K},\tilde{K}''$ of $\tilde{\Gamma}_\Sigma,\tilde{\Gamma}_{\Sigma''}$ which extend $\tilde{\mathbf{K}}_{\mathbf{i},\mathbf{j}},\tilde{\mathbf{K}}_{\mathbf{i},\mathbf{j}}''$. Then, it suffices to look at 
\begin{equation}\label{eq:2vect-comp}
    \begin{split}
        ((RT_{\mathcal{C},\sqrt{D}}(M',n')&\circ RT_{\mathcal{C},\sqrt{D}}(M,n))_\mathbf{i}^\mathbf{j})_{\tilde{K}}^{\tilde{K}''} \\ &= \sum_{\tilde{K}'\in\tilde{\mathbb{K}}_{\mathbf{i},\mathbf{j}}'} (RT_{\mathcal{C},\sqrt{D}}(M',n')_\mathbf{i}^\mathbf{j})_{\tilde{K}'}^{\tilde{K}''}\circ (RT_{\mathcal{C},\sqrt{D}}(M,n)_\mathbf{i}^\mathbf{j})_{\tilde{K}}^{\tilde{K}'}
    \end{split}
\end{equation}
where the sum is over all ribbon-colourings $\tilde{K}'$ of $\tilde{\Gamma}_{\Sigma'}$ which extend $\tilde{\mathbf{K}}_{\mathbf{i},\mathbf{j}}'$. We shall now construct a ribbon diagram whose operator invariant realises this composition.

\begin{constr}\label{constr:vert-glue}
    As in Construction \ref{constr:emb}, we may apply an ambient isotopy so that $|\tilde{\Gamma}_M\cup L|$ and $|\tilde{\Gamma}_{M'}\cup L'|$ end up in the form depicted in Figure \ref{fig:gamma-m}. In doing so, we have to pick an ordering of the vertices, as well as an ordering of the edges incident at each vertex. Note that in this case, $V_\mathrm{out}$ is canonically identified with $V_{\mathrm{in}}'$; they both correspond to the vertices of $\pm\tilde{\Gamma}_{\Sigma'}$. At this stage, we insist that the orderings chosen for these two sets are the same. Moreover, we need to pick an ordering of the edges incident at each vertex. We require the orderings for each of the vertices in $V_\mathrm{out}$ and $V_{\mathrm{in}}'$ to agree. (Recall here that the ordering of edges for each vertex in $V_\mathrm{out}$ is opposite to the cyclic ordering coming from $-\Gamma_{\Sigma'}$, which is in turn opposite to the cyclic ordering coming from $\Gamma_{\Sigma'}$. It is thus possible for this to agree with the ordering of edges at the corresponding vertex of $V_\mathrm{in}'$.) Then by construction, the orders of the external edges of $|\tilde{\Gamma}_M\cup L|$ on $\mathbb{R}\times\set{1}\times\set{0}$ as well as the external edges on of $|\tilde{\Gamma}_{M'}\cup L'|$ on $\mathbb{R}\times\set{0}\times\set{0}$ match up, and it would then make sense to concatenate the resulting ribbon diagrams.

    This concatenation is, however, not exactly what we want; we also need to insert something else in the middle. Define $Y_p,Y_p'$ as in (\ref{eq:def-yp}), and $Y_p''$ analogously. (The compatibility conditions we put on the orderings in the previous paragraph ensure that the two possible definitions of $Y_p'$ agree.) Now, consider the ribbon diagram in Lemma \ref{lem:tens-hom}, with the labels of the coloured strands being replaced by $Y_1',\ldots,Y_b'$. This is an embedded ribbon graph in $\mathbb{R}\times[0,1]\times\mathbb{R}$. On the other hand, consider the embedded ribbon graphs $(\tilde{\Gamma}_M\setminus V_\mathrm{out}) \cup L$ in $\mathbb{R}\times(-\infty,1]\times\mathbb{R}$ and $(\tilde{\Gamma}_{M'}\setminus V_\mathrm{in}') \cup L'$ in $\mathbb{R}\times[0,\infty)\times\mathbb{R}$. We may then glue these together to form an embedded ribbon graph in $\mathbb{R}^3$.

    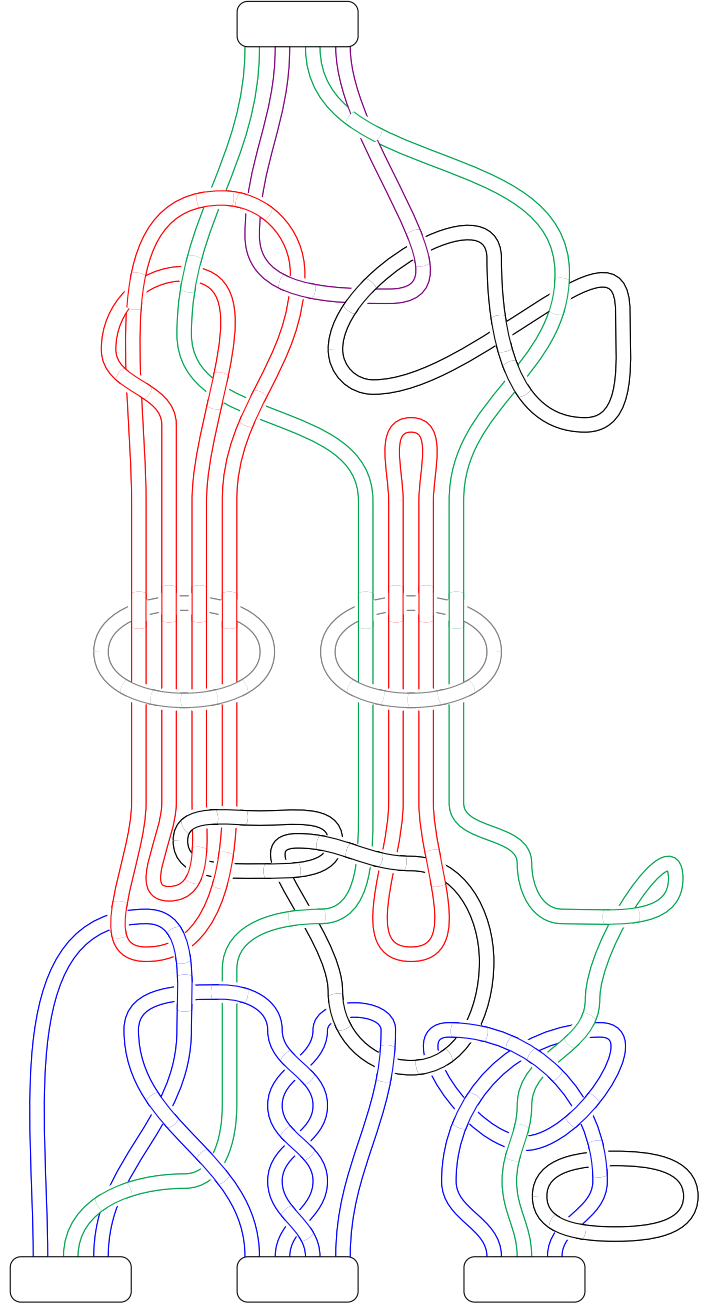
\begin{figure}[hbt!]
        \begin{tikzpicture}
            \begin{knot}[consider self intersections, end tolerance=2pt, clip radius=5pt, clip width=2]
                \strand[draw=blue, double=white, double distance=5pt] (-0.6,-3) to [out=up, in=down] (-2.2,0) to [out=up,in=left] (-1,0.5) to [out=right,in=up] (-0.3,0) to [out=down,in=up] (0.3,-1) to [out=down,in=up] (-0.3,-2) to [out=down,in=up] (0.2,-3);
                \strand[draw=blue, double=white, double distance=5pt] (0.6,-3) to [out=up, in=down] (1.2,0) to [out=up,in=up] (0.3,0) to [out=down,in=up] (-0.3,-1) to [out=down,in=up] (0.3,-2) to [out=down,in=up] (-0.2,-3);
                \strand[draw=blue, double=white, double distance=5pt] (-3.4,-3) to [out=up,in=left] (-2,1.5) to [out=right,in=up] (-1.5,0) to [out=down,in=up] (-2.6,-3);
                \strand[draw=blue, double=white, double distance=5pt] (2.6,-3) to [out=up,in=down] (2,-2) to [out=up,in=left] (4,0) to [out=right,in=right] (3,-1.5) to [out=left,in=left] (2,0) to [out=right,in=up] (4,-2) to [out=down,in=up] (3.4,-3);
                \strand[draw=red, double=white, double distance=5pt] (-2.1,7) to [out=down,in=up] (-2.1,3) to [out=down,in=left] (-2,1) to [out=right,in=down] (-0.9,3) to [out=up,in=down] (-0.9,7) to [out=up,in=down] (0,10) to [out=up,in=right] (-1,11) to [out=left,in=up] (-2.1,7);
                \strand[draw=red, double=white, double distance=5pt] (-1.7,7) to [out=down,in=up] (-1.7,3)  to [out=down,in=left] (-1.7,1.8) to [out=right,in=down] (-1.3,3) to [out=up,in=down] (-1.3,7) to [out=up,in=right] (-1.5,10) to [out=left,in=up] (-2.5,9) to [out=down,in=up] (-1.7,8) to [out=down,in=up] (-1.7,7);
                \strand[draw=red, double=white, double distance=5pt] (1.3,7) to [out=down,in=up] (1.3,3) to [out=down,in=left] (1.5,1) to [out=right,in=down] (1.7,3) to [out=up,in=down] (1.7,7) to [out=up,in=right] (1.5,8) to [out=left,in=up] (1.3,7);
                \strand[draw=black, double=white, double distance=5pt] (3.2,-2.2) to [out=down,in=left] (4.2,-2.7) to [out=right,in=down] (5.2,-2.2) to [out=up,in=right] (4.2,-1.7) to [out=left,in=up] (3.2,-2.2);
                \strand[draw=black, double=white, double distance=5pt] (0.5,0.5) to [out=down,in=left] (1.5,-0.5) to [out=right,in=down] (2.5,0.9) to [out=up,in=right] (1.5,2.2) to [out=left,in=right] (0,2.5) to [out=left,in=up] (0.5,0.5);
                \strand[draw=black, double=white, double distance=5pt] (-1.55,2.5) to [out=down,in=left] (-0.5,2.1) to [out=right,in=down] (0.5,2.5) to [out=up,in=right] (-0.5,2.8) to [out=left,in=up] (-1.55,2.5);
                \strand[draw=Green, double=white, double distance=5pt] (-3,-3) to [out=up,in=left] (-1.5,-2) to [out=right,in=down] (-0.9,-1) to [out=up,in=down] (-0.9,0.8) to [out=up,in=left] (0.3,1.5) to [out=right,in=down] (0.9,3) to [out=up,in=down] (0.9,7) to [out=up,in=down] (-1.5,9.2) to [out=up,in=down] (-0.6,13);
                \strand[draw=Green, double=white, double distance=5pt] (0.2,13) to [out=down,in=up] (3.5,10) to [out=down,in=up] (2.1,7) to [out=down,in=up] (2.1,3) to [out=down,in=up] (3,2.2) to [out=down,in=left] (3.5,1.5) to [out=right,in=down] (5,2) to [out=up,in=up] (3.9,0.4) to [out=down,in=up] (3,-1) to [out=down,in=up] (2.8,-2) to [out=down,in=up] (3,-3);
                \strand[draw=gray, double=white, double distance=5pt] (-2.6,5) to [out=down,in=down] (-0.4,5) to [out=up,in=up] (-2.6,5);
                \strand[draw=gray, double=white, double distance=5pt] (2.6,5) to [out=down,in=down] (0.4,5) to [out=up,in=up] (2.6,5);
                \strand[draw=violet, double=white, double distance=5pt] (-0.2,13) to [out=down,in=up] (-0.6,10.5) to [out=down,in=left] (1.2,9.7) to [out=right,in=down] (0.6,13);
                \strand[draw=black, double=white, double distance=5pt] (0.5,9) to [out=up,in=up] (2.6,10.2) to [out=down,in=left] (3.8,8) to [out=right,in=down] (4.3,9.5) to [out=up,in=right] (1,8.5) to [out=left,in=down] (0.5,9);
                \flipcrossings{3,6,9,10,12,14,15,17,18,22,23,26,27,30,33,35,38,39,42,43,45,47,51,52,53,56,58,61}
            \end{knot}
            \draw[rounded corners] (-0.8,-3.6) rectangle (0.8,-3);
            \draw[rounded corners] (-3.8,-3.6) rectangle (-2.2,-3);
            \draw[rounded corners] (2.2,-3.6) rectangle (3.8,-3);
            \draw[rounded corners] (-0.8,13) rectangle (0.8,13.6);
        \end{tikzpicture}
        \caption{Embedding of $|\tilde{\Gamma}_{M''}\cup L''|$.} \label{fig:gamma-glued}
    \end{figure}
    Figure \ref{fig:gamma-glued} shows an example. Here, $\tilde{\Gamma}_M\cup L$ is as in Figure \ref{fig:gamma-m}. This new graph consists of $\tilde{\Gamma}_{M''}$, along with some loops: these are made out of $L$, $L'$ (in black), the loops from the graph in Lemma \ref{lem:tens-hom} (in grey), as well as some new loops which arise from gluing the internal edges of $\pm\tilde{\Gamma}_{\Sigma'}$ to each other (in red). Let $L_\mathrm{vert}$ denote the set of the red loops and $L_\mathrm{hor}$ denote the set of the grey loops. Let $L'' = L \cup L' \cup L_\mathrm{vert} \cup L_\mathrm{hor}$ denote the set of all of these loops.
\end{constr}

A choice of colourings $\tilde{K},\tilde{K''}$ which extend $\tilde{\mathbf{K}}_{\mathbf{i},\mathbf{j}},\tilde{\mathbf{K}}_{\mathbf{i},\mathbf{j}}''$ defines a colouring $\tilde{K}_{M''}$ of $\tilde{\Gamma}_{M''}$, making $(\tilde{\Gamma}_M\setminus (V_\mathrm{in}\cup V_\mathrm{out}')) \cup L''$ a labelled ribbon graph with uncoloured loops. We now consider $F_*$ of this graph.

\begin{lem}\label{lem:vert-alg}
    The linear map $$\Phi'':= F_*((\tilde{\Gamma}_M\setminus (V_\mathrm{in}\cup V_\mathrm{out}')) \cup L'',\tilde{K}_{M''},\emptyset): \Hom\pa{\mathbbm{1},\bigotimes_{j=1}^a Y_j} \rightarrow \Hom\pa{\mathbbm{1},\bigotimes_{j=1}^b Y_j''}$$ along with the canonical inclusion $\varphi$ and canonical projection $\psi''$ satisfy $$((RT_{\mathcal{C},\sqrt{D}}(M',n')\circ RT_{\mathcal{C},\sqrt{D}}(M,n))_\mathbf{i}^\mathbf{j})_{\tilde{K}}^{\tilde{K}''} =  A\dim_\mathrm{int}(\tilde{K}'')\psi''\Phi''\varphi$$ where $$A := \pa{\frac{p_-}{\sqrt{D}}}^{\sigma(X_L)+\sigma(X_{L'})-n-n'} \sqrt{D}^{\chi(\ol{\Sigma''})/2-\# L''-2}.$$
\end{lem}
\begin{proof}
    Applying (\ref{eq:2vect-comp}) and Lemma \ref{lem:rt-formula}, the left-hand side is equal to
    \begin{equation}\label{eq:comp-sum}
        B\dim_\mathrm{int}(\tilde{K}'')\sum_{\tilde{K}'\in\tilde{\mathbb{K}}_{\mathbf{i},\mathbf{j}}'} \dim_\mathrm{int}(\tilde{K}') \psi''\Phi'\varphi'\psi'\Phi\varphi
    \end{equation}
    where $$B := \pa{\frac{p_1}{\sqrt{D}}}^{\sigma(X_L)+\sigma(X_{L'})-n-n'}\sqrt{D}^{\chi(\ol{\Sigma'})/2 + \chi(\ol{\Sigma''})/2-\# L-\# L'-2}.$$

    By Lemma \ref{lem:tens-hom}, $\varphi'\psi'$ is $D^{-b_0(\Sigma')}$ times $F_*$ applied to the central portion of Figure \ref{fig:gamma-glued}. Substituting this into (\ref{eq:comp-sum}) and using the functoriality of $F_*$, we obtain $$BD^{-b_0(\Sigma')}\dim_\mathrm{int}(\tilde{K}'')\sum_{\tilde{K}'\in\tilde{\mathbb{K}}_{\mathbf{i},\mathbf{j}}'} \dim_\mathrm{int}(\tilde{K}') \psi''F_*((\tilde{\Gamma}_M\setminus (V_\mathrm{in}\cup V_\mathrm{out}')), \tilde{\Gamma}_{M''}\sqcup \tilde{K'},\emptyset)\varphi.$$ By the definition of $F_*$ on ribbon graphs with uncoloured loops, this is exactly $$BD^{-b_0(\Sigma')}\dim_\mathrm{int}(\tilde{K}'')\psi''\Phi''\varphi.$$

    Now, we have $$\# L'' = \# L + \# L' + b_0(\ol{\Sigma'}) + \frac12 b_1(\ol{\Sigma'}) = \# L + \# L' - \frac{\chi(\ol{\Sigma'})}2 + 2b_0(\ol{\Sigma'})$$ and so $$A = BD^{-b_0(\ol{\Sigma'})} = \pa{\frac{p_1}{\sqrt{D}}}^{\sigma(X_L)+\sigma(X_{L'})-n-n'}\sqrt{D}^{\chi(\ol{\Sigma''})/2-\# L''-2}$$ as desired.
\end{proof}

We shall now interpret Figure \ref{fig:gamma-glued} geometrically. It turns out that performing surgery along $L''$ doesn't yield $\ol{M''}$ exactly, but instead $\ol{M''}\# (S^1\times S^2)$. Nevertheless, by using Lemma \ref{lem:add-1handle}, we may easily relate the invariants associated to $\ol{M''}$ and $\ol{M''}\# (S^1\times S^2)$.

\begin{prop}\label{prop:vert-glue}
    For the embedding of $|\tilde{\Gamma}_{M''}\cup L''|$ defined in Construction \ref{constr:vert-glue},
    \begin{enumerate}
        \item $|\tilde{\Gamma}_{M''}\cup L''|$ is a surgery presentation for $(\ol{M''}\#(S^1\times S^2), |\tilde{\Gamma}_{M''}|)$.
        \item The link trace $X_{L''}$ has the same signature as to $X_L \cup_{H_{\Sigma'}} X_{L'}$.
    \end{enumerate}
\end{prop}
\begin{proof}
    Consider the tuple $(X_L,\ol{M},-H_{\Sigma'},|\Gamma_M|)$. These are related in the following ways: $\ol{M}$ is the boundary of $X_L$, $-H_{\Sigma'}$ is a codimension zero submanifold of $\ol{M}$, and $|\Gamma_M|$ is an embedded surface in $\ol{M}$. Moreover, the intersection of $|\Gamma_M|$ with $H_{\Sigma'}$ is exactly an embedded $|\Gamma_{\Sigma'}|$. Equivalently, $|\Gamma_M \setminus \Gamma_{\Sigma'}|$ is embedded in $\ol{M} \setminus H_{\Sigma}$. We want to study the result of gluing $(X_L,\ol{M},H_{\Sigma'},|\Gamma_M\setminus \Gamma_{\Sigma'}|)$ along $H_{\Sigma'}$ to $(X_{L'},\ol{M'},-H_{\Sigma'},|\Gamma_{M'} \setminus \Gamma_{\Sigma'}|)$. The result will be a 4-manifold with boundary $\ol{M''} = \ol{M} \cup_{H_{\Sigma'}} \ol{M'}$ with an embedded $\Gamma_{M''}$.
    
    Recall that the embedded ribbon graph $\Gamma_{\Sigma'}$ in $H_{\Sigma'}$ admits a thickening to $|\Gamma_{\Sigma'}| \times [-1,1] = |\tilde{\Gamma}_{\Sigma'}| \times [-1,1]$ such that both in both $\ol{M}$ and $\ol{M'}$, the embedded $\pm H_{\Sigma'}$ is ambiently isotopic to $|\tilde{\Gamma}_{\Sigma'}| \times I$. We may thus equivalently study the result of gluing $(D^4,S^3,|\tilde{\Gamma}_{\Sigma'}| \times I,|\tilde{\Gamma}_M\setminus\tilde{\Gamma}_{\Sigma'} \cup L|)$ along $|\tilde{\Gamma}_{\Sigma'}| \times I$ to $(D^4,S^3,-|\tilde{\Gamma}_{\Sigma'}| \times I,|\tilde{\Gamma}_{M'}\setminus\tilde{\Gamma}_{\Sigma'} \cup L')|$, followed by attaching 2-handles along the framed links $L$ and $L'$.
    
    We shall perform the first gluing in stages. Recall that the surface $|\tilde{\Gamma}_{\Sigma'}|$ consists of 0-handles indexed by the set $V$ of vertices of $\tilde{\Gamma}_{\Sigma'}$ and 1-handles indexed by the set $E$ of edges of $\tilde{\Gamma}_{\Sigma'}$. Its thickening is correspondingly a disjoint union of handlebodies embedded in $S^3$. We work with the embeddings of $|\tilde{\Gamma}_M \cup L|$ and $|\tilde{\Gamma}_{M'} \cup L'|$ as defined in Construction~\ref{constr:emb}. Then, the 0-handles corresponding to the vertices of $\tilde{\Gamma}_{\Sigma'}$ are in a neighbourhood of $\mathbb{R}^2 \times \set{0}$ with their centres arranged in a line. Join each pair of adjacent 0-handles by attaching a new 1-handle whose core lies on this line, so the union of the 0-handles with the new 1-handles forms an embedded $D^3$ in $S^3$. Upon doing this to both copies of $S^3$, we may now glue $(D^4,S^3,D^3,|\tilde{\Gamma}_M \setminus V \cup L|)$ and $(D^4,S^3,-D^3,|\tilde{\Gamma}_{M'}\setminus V \cup L'|)$ together along these two $D^3$s. Figure \ref{fig:glue-d3s} below shows the result of this gluing locally around the two $D^3$s; the cyan ribbons on the right correspond to the belt spheres of the new 1-handles. Overall, the result is a $D^4$ with the surface $|\tilde{\Gamma}_{M''} \cup L \cup L' \cup L_\mathrm{vert}|$ embedded in the boundary $S^3$ exactly as drawn in Figure \ref{fig:gamma-glued}, but without the grey surgery link $L_{\mathrm{hor}}$.

    \begin{figure}[hbt!]
        \begin{tikzpicture}
            \shade[ball color = gray!40, opacity = 0.4] (0,1) circle (0.7);
            \draw[gray] (0,1) circle (0.7);
            \shade[ball color = gray!40, opacity = 0.4] (-2,1) circle (0.7);
            \draw[gray] (-2,1) circle (0.7);
            \shade[ball color = gray!40, opacity = 0.4] (2,1) circle (0.7);
            \draw[gray] (2,1) circle (0.7);
            \shade[ball color = gray!40, opacity = 0.4] (0,-1) circle (0.7);
            \draw[gray] (0,-1) circle (0.7);
            \shade[ball color = gray!40, opacity = 0.4] (-2,-1) circle (0.7);
            \draw[gray] (-2,-1) circle (0.7);
            \shade[ball color = gray!40, opacity = 0.4] (2,-1) circle (0.7);
            \draw[gray] (2,-1) circle (0.7);

            \draw[cyan, top color=cyan!20,bottom color=cyan, opacity=0.4] (-1.328,0.804) rectangle (-0.672,1.196);
            \draw[cyan, top color=cyan!20,bottom color=cyan, opacity=0.4] (0.672,0.804) rectangle (1.328,1.196);
            \draw[cyan, top color=cyan!20,bottom color=cyan, opacity=0.4] (-1.328,-1.196) rectangle (-0.672,-0.804);
            \draw[cyan, top color=cyan!20,bottom color=cyan, opacity=0.4] (0.672,-1.196) rectangle (1.328,-0.804);

            \draw[draw=red, double=white, double distance=5pt] (-2.4,1.4) .. controls (-2.8,1.7) .. (-3,2.5);
            \draw[draw=Green, double=white, double distance=5pt] (-2,1.6) -- (-2,2.5);
            \draw[draw=red, double=white, double distance=5pt] (-1.6,1.4) .. controls (-1.2,1.7) .. (-1,2.5);
            \draw[draw=red, double=white, double distance=5pt] (-2.4,-1.4) .. controls (-2.8,-1.7) .. (-3,-2.5);
            \draw[draw=Green, double=white, double distance=5pt] (-2,-1.6) -- (-2,-2.5);
            \draw[draw=red, double=white, double distance=5pt] (-1.6,-1.4) .. controls (-1.2,-1.7) .. (-1,-2.5);

            \draw[draw=Green, double=white, double distance=5pt] (0,1.6) -- (0,2.5);
            \draw[draw=Green, double=white, double distance=5pt] (0,-1.6) -- (0,-2.5);

            \draw[draw=red, double=white, double distance=5pt] (2.4,1.4) .. controls (2.8,1.7) .. (3,2.5);
            \draw[draw=red, double=white, double distance=5pt] (1.6,1.4) .. controls (1.2,1.7) .. (1,2.5);
            \draw[draw=red, double=white, double distance=5pt] (2.4,-1.4) .. controls (2.8,-1.7) .. (3,-2.5);
            \draw[draw=red, double=white, double distance=5pt] (1.6,-1.4) .. controls (1.2,-1.7) .. (1,-2.5);
            \draw[draw=red, double=white, double distance=5pt] (2.15,1.55) .. controls (2.3,2) .. (2.3,2.5);
            \draw[draw=red, double=white, double distance=5pt] (1.85,1.55) .. controls (1.7,2) .. (1.7,2.5);
            \draw[draw=red, double=white, double distance=5pt] (2.15,-1.55) .. controls (2.3,-2) .. (2.3,-2.5);
            \draw[draw=red, double=white, double distance=5pt] (1.85,-1.55) .. controls (1.7,-2) .. (1.7,-2.5);

            \node[] at (3.7,0) {$\rightsquigarrow$};

            \begin{knot}[clip radius=5pt, clip width=2]
                \strand[draw=red, double=white, double distance=5pt] (5.2,2) -- (5.2,-2);
                \strand[draw=Green, double=white, double distance=5pt] (5.5,2) -- (5.5,-2);
                \strand[draw=red, double=white, double distance=5pt] (5.8,2) -- (5.8,-2);
                \strand[draw=cyan, double=cyan!20, double distance=3pt] (4.8,0) to [out=up,in=up] (6.2,0) to [out=down,in=down] (4.8,0);
                \strand[draw=Green, double=white, double distance=5pt] (6.7,2) -- (6.7,-2);
                \strand[draw=cyan, double=cyan!20, double distance=3pt] (4.6,0) to [out=up,in=left] (5.85,0.8) to [out=right,in=up] (7.1,0) to [out=down,in=right] (5.85,-0.8) to [out=left,in=down] (4.6,0);
                \flipcrossings{2,4,6,8,10,12,14}
            \end{knot}
            \draw[draw=red, double=white, double distance=5pt] (7.6,2) -- (7.6,-2);
            \draw[draw=red, double=white, double distance=5pt] (7.9,2) -- (7.9,-2);
            \draw[draw=red, double=white, double distance=5pt] (8.2,2) -- (8.2,-2);
            \draw[draw=red, double=white, double distance=5pt] (8.5,2) -- (8.5,-2);

        \end{tikzpicture}
        \caption{Gluing $D^4$s along the embedded $D^3$s} \label{fig:glue-d3s}
    \end{figure}
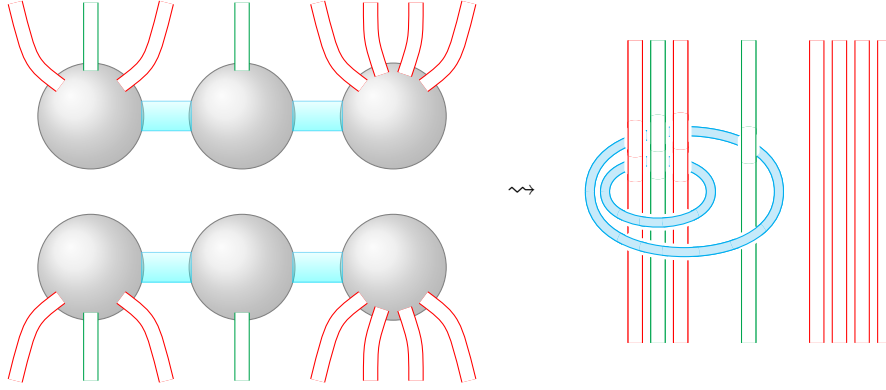

    Now, suppose we had instead glued the two $D^4$s along all the 0-handles $\bigsqcup_{v\in V} D^3$, without first attaching the extra 1-handles. In the gluing described in the previous paragraph, the belt spheres of the extra 1-handles are as drawn in the right side of Figure \ref{fig:glue-d3s}. More precisely, there are $\# V-1$ of these, with the $i^\mathrm{th}$ going around the first $i$ groups of vertical strands for each $i$. Starting from the previous glued 4-manifold, our new glued 4-manifold may be obtained by removing tubular neighbourhoods of $\# V-1$ boundary-parallel $D^2$s whose boundaries are given by the belt spheres above. In the language of Kirby calculus, this corresponds to adding $\# V-1$ dotted circles. By performing handle slides, this is the same as adding dotted circles along all but one the components of $L_\mathrm{hor}$, i.e. all but one of the grey circles in Figure \ref{fig:gamma-glued}. As before, the resulting 4-manifold has an embedded $|\tilde{\Gamma}_{M''} \cup L \cup L' \cup L_\mathrm{vert}|$ in its boundary.

    Next, we consider what we would have obtained by instead gluing the $D^4$s along the $|\tilde{\Gamma}_{\Sigma'}| \times I$, i.e. the 0-handles from the previous paragraph plus the 1-handles corresponding to the edges of $\tilde{\Gamma}_{\Sigma'}$. Starting from the 4-manifold in the previous paragraph, this exactly corresponds to gluing in 2-handles along tubular neighbourhoods of the $|L_{\mathrm{vert}}|$, with framing exactly as given by the embedding in Construction \ref{constr:vert-glue}. The resulting 4-manifold now has an embedded $|\tilde{\Gamma}_{M''} \cup L \cup L'|$ in its boundary.
    
    Finally, we can glue in 2-handles along $L$ and $L'$. We have thus obtained $X_L \cup_{H_{\Sigma'}} X_{L'}$ with the embedded $|\tilde{\Gamma}_{M''}|$ in its boundary, which is $\ol{M''}$.

    Now, note that in Kirby calculus, replacing dotted circles with trivially framed unknots does not change the boundary of the described 4-manifold, and hence we see that $(\ol{M''},|\tilde{\Gamma}_{M''}|)$ has a surgery presentation as in Figure \ref{fig:gamma-glued}, except with one of the components of $L_{\mathrm{hor}}$ (i.e. one of the grey unknots) deleted. Taking a connect sum with $S^2\times S^1$ is the same as taking the disjoint union of this with one trivially framed unknot. After handle sliding this across the other grey unknots, this results in the final grey unknot. Hence, we have proven (1).

    On the other hand, we have also shown that the Kirby diagram for $X_L \cup_{H_{\Sigma'}} X_{L'}$ has the 1-handle attachments encoded by dotted circles for all but one component of the $L_\mathrm{hor}$, and 2-handle attachments encoded by the framed link $L\cup L' \cup L_{\mathrm{vert}}$ corresponding to 2-handle addition. Note that attaching a 1-handle to a 4-manifold does not change its signature, since the vector spaces in Wall's Theorem (Theorem \ref{thm:wall}) are zero vector spaces, and $\sigma(D^4) = 0$. Applying this to $X_L \cup_{H_{\Sigma'}} X_{L'}$, this corresponds to adding an extra dotted circle to the Kirby diagram, which upon handle sliding gives exactly the final component of $L_\mathrm{hor}$. Now, gluing this to $-X_{L''}$ along their common boundary $\ol{M''}$ exactly yields the double of the link trace of $L\cup L'\cup L_{\mathrm{vert}}$. Computing the signature of this closed 4-manifold in two ways, we thus have $$\sigma(X_L \cup_{H_{\Sigma'}} X_{L'}) - \sigma(X_{L''}) = \sigma(X_{L\cup L'\cup L_{\mathrm{vert}}}) - \sigma(X_{L\cup L'\cup L_{\mathrm{vert}}}) = 0,$$ which gives us (2).
\end{proof}

By comparing the results of Lemma \ref{lem:vert-alg} and Proposition \ref{prop:vert-glue}, we may now conclude that $RT_{\mathcal{C},\sqrt{D}}$ is compatible with vertical composition.

\begin{prop}\label{prop:vert-comp}
    Let $(M,n)$ and $(M',n')$ be vertically composable 2-morphisms in $\Bord{sig}$. Let $(M'',n'')$ be their composition. Then, $$RT_{\mathcal{C},\sqrt{D}}(M',n')\circ RT_{\mathcal{C},\sqrt{D}}(M,n) = RT_{\mathcal{C},\sqrt{D}}(M'',n'').$$
\end{prop}
\begin{proof}
    It suffices to show
    \begin{equation}\label{eq:vert-comp-toshow}
        ((RT_{\mathcal{C},\sqrt{D}}(M',n')\circ RT_{\mathcal{C},\sqrt{D}}(M,n))_{\mathbf{i}}^{\mathbf{j}})_{\tilde{K}}^{\tilde{K}''} = ((RT_{\mathcal{C},\sqrt{D}}(M'',n''))_{\mathbf{i}}^{\mathbf{j}})_{\tilde{K}}^{\tilde{K}''}
    \end{equation}
    for all indices $\mathbf{i},\mathbf{j}$ and for all ribbon-colourings $\tilde{K}, \tilde{K}''$ of $\tilde{\Gamma}_\Sigma,\tilde{\Gamma}_{\Sigma''}$ extending $\tilde{\mathbf{K}}_{\mathbf{i},\mathbf{j}},\tilde{\mathbf{K}}_{\mathbf{i},\mathbf{j}}''$. By Lemma \ref{lem:add-1handle} followed by Lemma \ref{lem:rt-formula} and Proposition \ref{prop:vert-glue}(1), the right-hand side of (\ref{eq:vert-comp-toshow}) is given by
    \begin{equation*}
        \begin{split}
            (RT_{\mathcal{C},\sqrt{D}}(M'',n'')_{\mathbf{i}}^{\mathbf{j}})_{\tilde{K}}^{\tilde{K}''} &= \frac1{\sqrt{D}} (RT_{\mathcal{C},\sqrt{D}}(M''\# (S^1\times S^2),n'')_{\mathbf{i}}^{\mathbf{j}})_{\tilde{K}}^{\tilde{K}''} \\ &= \pa{\frac{p_-}{\sqrt{D}}}^{\sigma(X_{L''})-n''} \sqrt{D}^{\chi\pa{\ol{\parout M''}}/2-\# L''-2}\dim_\mathrm{int}(\tilde{K}')\psi''\Phi''\varphi
        \end{split}
    \end{equation*}
    where $L''$ is as defined in Construction \ref{constr:vert-glue}, and $\Phi''$ is as in Lemma \ref{lem:vert-alg}. Comparing this to the result of Lemma \ref{lem:vert-alg}, it remains to show that
    \begin{equation}\label{eq:two-csigs}
        \sigma(X_L)+\sigma(X_{L'})-n-n' = \sigma(X_{L''})-n''.
    \end{equation}

    Indeed, applying Wall's Theorem (Theorem \ref{thm:wall}) to the result of Proposition \ref{prop:vert-glue}(2), we obtain $$\sigma(X_{L''}) = \sigma(X_{L'}) + \sigma(X_L) - \sigma(V;A,B,C)$$ where $V,A,B,C$ are exactly as in (\ref{eq:csig}). Comparing this to the definition of $n''$ in Definition \ref{defn:comp}, we obtain (\ref{eq:two-csigs}), as desired.
\end{proof}

\subsection{Horizontal composition of 2-morphisms}

The approach is similar in this case. Given horizontally composable 2-morphisms $(M,n)$ and $(M',n')$, let their horizontal composition be $(M'',n'')$, so $n''=n+n'$. Again, by Lemma \ref{lem:add-1handle}, we may assume that $M,M'$ (and hence $M''$) are connected. In this subsection we revert to looking at the uncontracted graphs $\Gamma_M,\Gamma_{M'}$. We shall construct surgery presentations for $(\ol{M},|\Gamma_M|)$ and $(\ol{M'},|\Gamma_{M'}|)$, then use these construct a surgery presentation for $(\ol{M''},|\Gamma_{M''}|)$. We can then show that the corresponding linear maps defined in this way agree with the horizontal composition.

\begin{constr}\label{constr:hor-glue}
    Pick a total ordering on the components of $V_{\parOut M}$, and hence $V_{\parIn M'}$; suppose there are $r$ of these components. Recall that these are disjoint unions of $D^2\times[0,1]$ along each of whose cores there is exactly one edge of $\Gamma_M$ and $\Gamma_{M'}$, respectively. Let $E_\mathrm{vert}$ denote the ribbon subgraph of these that consists of these edges (with source and target taken to be $\infty$). In $\Gamma_M$, the tail of each edge in $E_\mathrm{vert}$ is a vertex in $\Gamma_{\parin M}$, while its head is a vertex in $\Gamma_{\parout M}$; for $\Gamma_{M'}$, this is reversed.

    By performing a suitable ambient isotopy, we may make the embeddings of $\Gamma_M\cup L$ and $\Gamma_{M'}\cup L'$ in $S^3$ satisfy the following conditions:
    \begin{itemize}
        \item $|\Gamma_M\cup L|$ is embedded in $(-\infty,0)\times\mathbb{R}\times\mathbb{R} \subset \mathbb{R}^3\subset S^3$, and its intersection with $(-\infty,-1]\times\mathbb{R}\times\mathbb{R}$ is exactly the subsurface given by $|E_\mathrm{vert}|$. Likewise, $|\Gamma_{M'}\cup L'|$ is embedded in $(0,\infty)\times\mathbb{R}\times\mathbb{R}\subset\mathbb{R}^3\subset S^3$ and its intersection with $[1,\infty)\times\mathbb{R}\times\mathbb{R}$ is $|E_\mathrm{vert}|$.
        \item The endpoints of the ribbons of $E_\mathrm{vert}$ which correspond to vertices in $\Gamma_{\parin M}$ and $\Gamma_{\parin M'}$ lie in $\set{\pm1}\times\set{0}\times\mathbb{R}$ while the ones corresponding to vertices of $\Gamma_{\parout M}$ and $\Gamma_{\parout M'}$ lie in $\set{\pm1}\times\set{1}\times\mathbb{R}$.
        \item For both embeddings of $|E_\mathrm{vert}|$, the order in which the ribbons are incident on each of the lines $\set{\pm1}\times\set{0}\times\mathbb{R}$ and $\set{\pm1}\times\set{1}\times\mathbb{R}$ corresponds to the total order chosen.
        \item For both embeddings of $|E_\mathrm{vert}|$, the projection onto $\mathbb{R}^2$ given by forgetting the first coordinate yields a ribbon graph corresponding to the identity morphism.
    \end{itemize}

    Now, take the these embeddings of $|(\Gamma_M\setminus E_{\mathrm{vert}})\cup L|$ and $|(\Gamma_{M'}\setminus E_{\mathrm{vert}})\cup L'|$, and glue them to $r$ copies of a rotated version of the left-hand side of the diagram in Lemma \ref{lem:BK}. This yields a new ribbon graph embedded in $\mathbb{R}^3\subset S^3$.

    \begin{figure}[hbt!]
        \begin{tikzpicture}
            \draw[draw=Green, double=white, double distance=5pt] (-9.7,1.7) .. controls (-9.4,0) .. (-9.7,-1.7);
            \draw[draw=Green, double=white, double distance=5pt] (-9.3,1.3) .. controls (-9,0) .. (-9.3,-1.3);
            \draw[draw=Green, double=white, double distance=5pt] (-8.9,0.9) .. controls (-8.6,0) .. (-8.9,-0.9);
            \draw[draw=Green, double=white, double distance=5pt] (-6.3,1.7) .. controls (-6.6,0) .. (-6.3,-1.7);
            \draw[draw=Green, double=white, double distance=5pt] (-6.7,1.3) .. controls (-7,0) .. (-6.7,-1.3);
            \draw[draw=Green, double=white, double distance=5pt] (-7.1,0.9) .. controls (-7.4,0) .. (-7.1,-0.9);
            \draw[black] (-10,3) -- (-8.2,0.3);
            \draw[black] (-10,-3) -- (-8.2,-0.3);
            \draw[black] (-6,3) -- (-7.8,0.3);
            \draw[black] (-6,-3) -- (-7.8,-0.3);
            \draw[black] (-10,3) -- (-10,-3);
            \draw[black] (-8.2,0.3) -- (-8.2,-0.3);
            \draw[black] (-6,3) -- (-6,-3);
            \draw[black] (-7.8,0.3) -- (-7.8,-0.3);
            \draw[black, dashed] (-10,2) -- (-8.2,0.2);
            \draw[black, dashed] (-10,-2) -- (-8.2,-0.2);
            \draw[black, dashed] (-6,2) -- (-7.8,0.2);
            \draw[black, dashed] (-6,-2) -- (-7.8,-0.2);
            \draw[Green, ->] (-9.5,-0.2) .. controls (-9.45,0) .. (-9.5,0.2);
            \draw[Green, ->] (-9.1,-0.2) .. controls (-9.04,0) .. (-9.1,0.2);
            \draw[Green, ->] (-8.7,-0.2) .. controls (-8.62,0) .. (-8.7,0.2);
            \draw[Green, <-] (-6.5,-0.2) .. controls (-6.55,0) .. (-6.5,0.2);
            \draw[Green, <-] (-6.9,-0.2) .. controls (-6.96,0) .. (-6.9,0.2);
            \draw[Green, <-] (-7.3,-0.2) .. controls (-7.38,0) .. (-7.3,0.2);

            \draw[black] (-2,3) -- (-0.2,0.3);
            \draw[black] (-2,-3) -- (-0.2,-0.3);
            \draw[black] (2,3) -- (0.2,0.3);
            \draw[black] (2,-3) -- (0.2,-0.3);
            \draw[black] (-2,3) -- (-2,-3);
            \draw[black] (-0.2,0.3) -- (-0.2,-0.3);
            \draw[black] (2,3) -- (2,-3);
            \draw[black] (0.2,0.3) -- (0.2,-0.3);
            \draw[black, dashed] (-0.95,0.95) -- (-0.2,0.2);
            \draw[black, dashed] (-0.95,-0.95) -- (-0.2,-0.2);
            \draw[black, dashed] (0.95,0.95) -- (0.2,0.2);
            \draw[black, dashed] (0.95,-0.95) -- (0.2,-0.2);
            \begin{knot}[consider self intersections, end tolerance=2pt, clip radius=5pt, clip width=2]
                \strand[draw=Green, double=white, double distance=5pt] (-1.7,1.7) .. controls (0,1.4) .. (1.7,1.7);
                \strand[draw=Green, double=white, double distance=5pt] (-1.3,1.3) .. controls (0,1) .. (1.3,1.3);
                \strand[draw=Green, double=white, double distance=5pt] (-0.9,0.9) .. controls (0,0.6) .. (0.9,0.9);
                \strand[draw=Green, double=white, double distance=5pt] (-1.7,-1.7) .. controls (0,-1.4) .. (1.7,-1.7);
                \strand[draw=Green, double=white, double distance=5pt] (-1.3,-1.3) .. controls (0,-1) .. (1.3,-1.3);
                \strand[draw=Green, double=white, double distance=5pt] (-0.9,-0.9) .. controls (0,-0.6) .. (0.9,-0.9);
                \strand[draw=gray, double=white, double distance=5pt] (0,0.9) to [out=left,in=left] (0,-0.9) to [out=right,in=right] (0,0.9);
                \strand[draw=gray, double=white, double distance=5pt] (0,1.32) to [out=left,in=left] (0,-1.32) to [out=right,in=right] (0,1.32);
                \strand[draw=gray, double=white, double distance=5pt] (0,1.74) to [out=left,in=left] (0,-1.74) to [out=right,in=right] (0,1.74);
                \flipcrossings{1,3,4,6,7,8,10,11,13,14,16,17,18,20}
            \end{knot}
            \draw[black, dashed] (-2,2) -- (-1,1);
            \draw[black, dashed] (-2,-2) -- (-1,-1);
            \draw[black, dashed] (2,2) -- (1,1);
            \draw[black, dashed] (2,-2) -- (1,-1);
            \draw[Green, ->] (0.05,-0.67) -- (0.45,-0.76);
            \draw[Green, ->] (0.15,-1.08) -- (0.55,-1.15);
            \draw[Green, ->] (0.25,-1.49) -- (0.65,-1.54);
            \draw[Green, <-] (0.05,0.67) -- (0.45,0.76);
            \draw[Green, <-] (0.15,1.08) -- (0.55,1.15);
            \draw[Green, <-] (0.25,1.49) -- (0.65,1.54);

            \node[] at (-4,0) {$\rightsquigarrow$};
        \end{tikzpicture}
        \caption{Replacing the $|E_\mathrm{vert}|$ in $[-1,1]\times\mathbb{R}^2$ to get an embedded $|\Gamma_{M''}\cup L''|$} \label{fig:hor-glue}
    \end{figure}
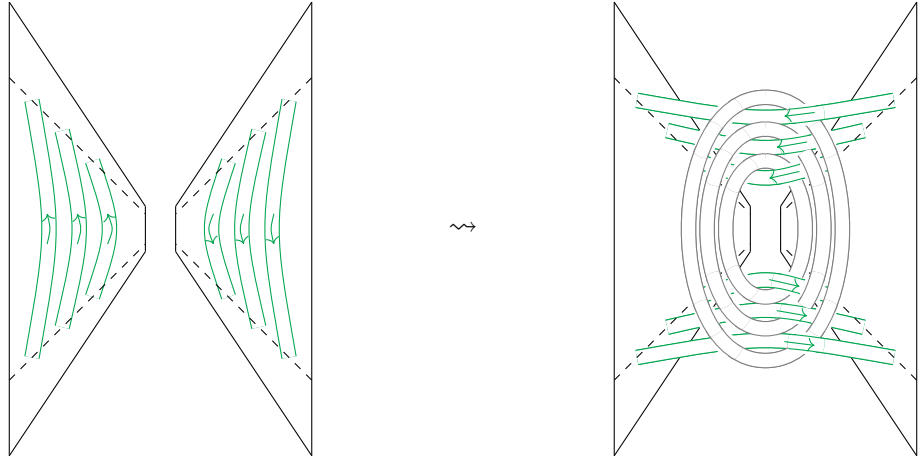
    Figure \ref{fig:hor-glue} shows a schematic of this process (for $r=3$): the diagram on the left shows the two embeddings of $\Gamma_M\cup L$ and $\Gamma_{M'}\cup {L'}$, with only the ribbons in $E_\mathrm{vert}$ drawn (in green). Then, the diagram on the right shows the result after these green ribbons are removed and replaced with $r$ copies of the left-hand side of the diagram in Lemma \ref{lem:BK}. The result is exactly an embedding of $\Gamma_{M''}$, along with loops corresponding to $L,L'$ as well as the $r$ new loops, which are drawn in grey. Let $L''$ be the union of all of these loops.
\end{constr}
\begin{rk}
    This construction is a generalisation of the ``special ribbon graph'' constructed in \cite{Tsumura}.
\end{rk}

\begin{lem}\label{lem:hor-glue}
    For the embedding of $|\Gamma_{M''}\cup L''|$ defined in Construction \ref{constr:hor-glue},
    \begin{enumerate}
        \item $|\Gamma_{M''}\cup L''|$ is a surgery presentation for $(\ol{M''}\#(S^1\times S^2), |\Gamma_{M''}|)$.
        \item The link trace $X_{L''}$ has the same signature as $X_L \cup_{V_{\parOut M}} X_{L'}$.
    \end{enumerate}
\end{lem}
\begin{proof}
    This is a special case of Proposition \ref{prop:vert-glue} (rotated by $90^\circ$) where each component of $\Sigma'$ has genus $0$.
\end{proof}

\begin{prop}\label{prop:hor-comp}
    Let $(M,n)$ and $(M',n')$ be horizontally composable 2-morphisms in $\Bord{sig}$. Let $(M'',n'')$ be their composition. Then, $$RT_{\mathcal{C},\sqrt{D}}(M',n')\star RT_{\mathcal{C},\sqrt{D}}(M,n) = RT_{\mathcal{C},\sqrt{D}}(M'',n'').$$
\end{prop}
\begin{proof}
    Let $\Sigma = \parin M''$ and $\Sigma' = \parout M''$. Throughout, $K,K'$ will denote ribbon-colourings of $\Gamma_\Sigma,\Gamma_{\Sigma'}$, and similar notation will be used for partial ribbon-colourings and the sets of ribbon-colourings which extend these. As in Construction \ref{constr:hor-glue}, let $r$ be the number of connected components of $V_{\parOut M} = V_{\parIn M'}$.

    Let the three matrices of linear maps be indexed by $\mathbf{i},\mathbf{j},\mathbf{k}$. We want to show that 
    \begin{equation} \label{eq:horcomp-toshow}
        RT_{\mathcal{C},\sqrt{D}}(M'',n'')_{\mathbf{i}}^{\mathbf{k}} = \sum_{\mathbf{j}\in\mathbb{X}^r} RT_{\mathcal{C},\sqrt{D}}(M,n)_{\mathbf{i}}^{\mathbf{j}} \otimes RT_{\mathcal{C},\sqrt{D}}(M',n')_{\mathbf{j}}^{\mathbf{k}}
    \end{equation}
    for each $\mathbf{i},\mathbf{k}$. Recall that the domain and codomain of these linear maps are given by direct sums of spaces of labels, indexed over possible ribbon-colourings of $\Gamma_{\Sigma}$ and $\Gamma_{\Sigma'}$ extending the partial ribbon-colourings $\mathbf{K}_{\mathbf{i},\mathbf{k}},\mathbf{K}_{\mathbf{i},\mathbf{k}}'$, which in turn are determined by colouring the external edges according to $\mathbf{i},\mathbf{k}$. In other words, these are linear maps $$\bigoplus_{K \in \mathbb{K}_{\mathbf{i},\mathbf{k}}} \mathcal{V}_{K}(\Gamma_{\Sigma}) \rightarrow \bigoplus_{K' \in \mathbb{K}_{\mathbf{i},\mathbf{k}}'} \mathcal{V}_{K'}(\Gamma_{\Sigma'}).$$

    Pick such ribbon-colourings $K,K'$; these uniquely determine a ribbon-colouring $K_{M''}$ of $\Gamma_{M''}$. In the definition of $(RT_{\mathcal{C},\sqrt{D}}(M'',n'')_{\mathbf{i}}^{\mathbf{k}})_K^{K'}$, we have to compute the invariant $\tau(\ol{M''},\Gamma_{M''},K_{M''},\mathbf{f})$ for each $\mathbf{f} \in \mathcal{V}_{K_{M''}}(\Gamma_{M''})$. In other words, we have to compute the operator invariant $F(\Gamma_{M''} \cup L'',K_{M''},\mathbf{f})$ where $|\Gamma_{M''}\cup L''|$ is a surgery presentation for $(\ol{M''},|\Gamma_{M''}|)$. By Lemma \ref{lem:add-1handle}, we may replace $\ol{M''}$ by its connect sum with $S^1\times S^2$ and divide the result by $\sqrt{D}$. By Lemma \ref{lem:hor-glue}, we may then use the embedding of $\Gamma_{M''}\cup L''$ as depicted on the right-hand side of Figure \ref{fig:hor-glue}. In other words, we have

    \begin{align*}
        \tau(\ol{M''},\Gamma_{M''}, K_{M''}, \mathbf{f}) &= \frac1{\sqrt{D}}\tau(\ol{M''}\#(S^1\times S^2),\Gamma_{M''}, K_{M''}, \mathbf{f}) \\ &= \pa{\frac{p_{-}}{\sqrt{D}}}^{\sigma\pa{X_{L''}}} \sqrt{D}^{-\# L''-2} F(\Gamma_{M''}\cup L'', K_{M''},\mathbf{f}).
    \end{align*}

    Let $E_\mathrm{centr},E_\mathrm{centr}'$ denote the central edges of $\Gamma_{\parin M''},\Gamma_{\parout M''}$ respectively; these are the green edges in the right-hand diagram of Figure \ref{fig:hor-glue}. These two sets are canonically identified with each other, as they are both in bijection with $E_\mathrm{vert}$. Then, by Lemma~\ref{lem:BK}, $(RT_{\mathcal{C},\sqrt{D}}(M'',n'')_{\mathbf{i}}^{\mathbf{k}})_K^{K'}$ gives a zero linear map unless the ribbon-colourings that $K,K'$ induce on $E_\mathrm{centr},E_\mathrm{centr}'$ agree. On the other hand, $\pa{RT_{\mathcal{C},\sqrt{D}}(M,n)_{\mathbf{i}}^{\mathbf{j}} \otimes RT_{\mathcal{C},\sqrt{D}}(M',n')_{\mathbf{j}}^{\mathbf{k}}}_K^{K'}$ is zero unless the ribbon-colourings $K|_{E_\mathrm{centr}},K'|_{E_\mathrm{centr}'}$ and \emph{also} $\mathbf{j}$ agree. Thus, if the ribbon-colourings $K|_{E_\mathrm{centr}},K'|_{E_\mathrm{centr}'}$ do not agree, then there is no possible $\mathbf{j}$ for which this is true, and so we have
    \begin{equation} \label{eq:vert-comp-intermediate}
        (RT_{\mathcal{C},\sqrt{D}}(M'',n'')_{\mathbf{i}}^{\mathbf{k}})_K^{K'} = \pa{\sum_{\mathbf{j}\in\mathbb{X}^r} RT_{\mathcal{C},\sqrt{D}}(M,n)_{\mathbf{i}}^{\mathbf{j}} \otimes RT_{\mathcal{C},\sqrt{D}}(M',n')_{\mathbf{j}}^{\mathbf{k}}}_K^{K'}
    \end{equation}
    as both sides are the zero linear map.
    
    Moreover, if they do agree, then the only nonzero term on the right-hand side of (\ref{eq:vert-comp-intermediate}) comes from the $\mathbf{j} = (j_1,\ldots,j_r)$ which corresponds to $K|_{E_\mathrm{centr}},K'|_{E_\mathrm{centr}'}$. Then, in calculating the operator invariant $F(\Gamma_{M''}\cup L'', K_{M''},\mathbf{f})$, we may apply Lemma \ref{lem:BK} to replace the diagram on the right-hand side of Figure \ref{fig:hor-glue} with the one on the left. This has the effect of multiplying the operator invariant by a factor of $\prod_{t=1}^r \frac{\dim j_t}D$.

    The new labelled embedded ribbon graph obtained in this way is exactly the disjoint union of $\Gamma_M \cup L$ and $\Gamma_{M'} \cup L'$, with $\Gamma_M$ ribbon-coloured according to $K,\mathbf{j}$ and $\Gamma_{M'}$ ribbon-coloured according to $K',\mathbf{j}$, and with coupons of both graphs labelled according to $\mathbf{f}$. Let these ribbon-colourings be $K_M,K_{M'}$ and the coupon labels be $\mathbf{f}_M,\mathbf{f}_{M'}$. Let $K_{\parout M} \in \mathbb{K}_{\mathbf{i},\mathbf{j}}$ be the corresponding ribbon-colouring of $\Gamma_{\parout M}$, and define $K_{\parout M'}$ similarly. We then have
    \begin{align*}
        &\dim_\mathrm{int}(K')F(\Gamma_{M''}\cup L'', K_{M''},\mathbf{f}) \\ &\quad= \dim_\mathrm{int}(K')\pa{\prod_{t=1}^r \frac{D}{\dim j_t}}F(\Gamma_M\cup L, K_M, \mathbf{f}_M)F(\Gamma_{M'}\cup L', K_{M'}, \mathbf{f}_{M'}) \\ &\quad= D^r \pa{\dim_{\mathrm{int}}(K_{\parout M}')F(\Gamma_M\cup L, K_M, \mathbf{f}_M)}\pa{\dim_{\mathrm{int}}(K_{\parout M'}')F(\Gamma_{M'}\cup L', K_{M'}, \mathbf{f}_{M'})}.
    \end{align*}

    Hence, we have
    \begin{align*}
        &\dim_\mathrm{int}(K') \tau(\ol{M''},\Gamma_{M''}, K_{M''}, \mathbf{f}) \\ &\quad= \frac1{\sqrt{D}}\dim_\mathrm{int}(K') \tau(\ol{M''}\#(S^1\times S^2),\Gamma_{M''}, K_{M''}, \mathbf{f}) \\ &\quad= \pa{\frac{p_{-}}{\sqrt{D}}}^{\sigma\pa{X_{L''}}} \sqrt{D}^{-\# L''-2} \dim_\mathrm{int}(K') F(\Gamma_{M''}\cup L'') \\ &\quad= \pa{\frac1{\sqrt{D}}}^r \times D^r \pa{\pa{\frac{p_-}{\sqrt{D}}}^{\sigma(X_L)}\sqrt{D}^{-\#L-1}\dim_{\mathrm{int}}(K_{\parout M}')F(\Gamma_M\cup L, K_M, \mathbf{f}_M)} \\ &\quad\quad\quad\quad\times \pa{\pa{\frac{p_-}{\sqrt{D}}}^{\sigma(X_{L'})}\sqrt{D}^{-\#L'-1}\dim_{\mathrm{int}}(K_{\parout M'}')F(\Gamma_{M'}\cup L', K_{M'}, \mathbf{f}_{M'})} \\ &\quad= \sqrt{D}^r \pa{\dim_{\mathrm{int}}(K_{\parout M}')\tau(\ol{M},\Gamma_M,K_M,\mathbf{f}_M)}\pa{\dim_{\mathrm{int}}(K_{\parout M'}')\tau(\ol{M'},\Gamma_{M'},K_{M'},\mathbf{f}_{M'})}.
    \end{align*}
    In the penultimate step, we used the fact that $\#L'' = \#L + \#L' + r$ as well as $\sigma(X_{L''}) = \sigma(X_{L}) + \sigma(X_{L'})$. The latter follows from Lemma \ref{lem:hor-glue}(2) and noting that $\parOut M = \parIn M$ has genus zero, so the term from Wall's invariant is zero. 

    On the other hand, by \cite[Lemma 4.5]{partA}, we have $$\frac{\chi(\ol{\parout M''})}2 + r = \frac{\chi(\ol{\parout M})}2 + \frac{\chi(\ol{\parout M'})}2$$ and hence, multiplying $\sqrt{D}^{\chi(\ol{\parout M''})/2}$ to both sides, we conclude that $$(RT_{\mathcal{C},\sqrt{D}}(M'',n'')_{\mathbf{i}}^{\mathbf{k}})_K^{K'} = \pa{RT_{\mathcal{C},\sqrt{D}}(M,n)_{\mathbf{i}}^{\mathbf{j}} \otimes RT_{\mathcal{C},\sqrt{D}}(M',n')_{\mathbf{j}}^{\mathbf{k}}}_K^{K'}$$ where $\mathbf{j}$ corresponds to the colouring that $K,K'$ induce on $E_\mathrm{centr},E_\mathrm{centr}'$.
    
    Summing over all $K \in \mathbb{K}_{\mathbf{i},\mathbf{k}},K' \in \mathbb{K}_{\mathbf{i},\mathbf{k}}'$, we have shown that (\ref{eq:horcomp-toshow}) holds, and the conclusion follows.
\end{proof}

\subsection{Compatibility with identity 2-morphisms}

Finally, we verify that $RT_{\mathcal{C},\sqrt{D}}$ sends identity 2-morphisms to identity 2-morphisms. By Propositions \ref{prop:vert-comp} and \ref{prop:hor-comp}, it suffices to check this for the identities of the four generating 1-morphisms. We shall do so for the pants; the rest are analogous.

\begin{lem}\label{lem:rt-id}
    For each generating 1-morphism $\tikztinypants,\tikztinycopants,\tikztinycup,\tikztinycap$, $RT_{\mathcal{C},\sqrt{D}}$ sends its identity 2-morphism to an identity 2-morphism in $\KV$.
\end{lem}
\begin{proof}
    Let $M$ be the underlying cobordism of $\id_{\tikztinypants}$, then $\ol{M} = S^3$, and $\Gamma_M$ is embedded as drawn in Figure \ref{fig:id}. Note that this exactly depicts the evaluation map for each $(RT_{\mathcal{C},\sqrt{D}}(\tikztinypants))_{i_1i_2}^j$. Thus, $\tau(\ol{M},\Gamma_M)$ is $\frac1{\sqrt{D}}$ times the evaluation map applied to the two labels. Hence, $RT_{\mathcal{C},\sqrt{D}}(M,0)$ is $\sqrt{D}^{\chi(S^2)/2}\times \frac1{\sqrt{D}} = 1$ times the identity.
    \begin{figure}[hbt!]
        \begin{tikzpicture}
            \draw[draw=blue, double=white, double distance=5pt] (-1,-1.6) to [out=up,in=left] (0,-0.6) to [out=right,in=up] (1,-1.6);
            \draw[draw=red, double=white, double distance=5pt] (-0.7,-2.2) to [out=down, in=left] (0,-2.9) to [out=right, in=down] (0.7,-2.2);
            \draw[draw=Green, double=white, double distance=5pt] (-1.3,-2.2) to [out=down, in=left] (0,-3.5) to [out=right, in=down] (1.3,-2.2);
            \draw[rounded corners] (-1.6,-2.2) rectangle (-0.4,-1.6);
            \draw[rounded corners] (0.4,-2.2) rectangle (1.6,-1.6);
            \draw[draw=blue, ->] (-0.2,-0.62) .. controls (0,-0.58) .. (0.2,-0.62);
            \draw[draw=Green, ->] (0.2,-3.48) .. controls (0,-3.52) .. (-0.2,-3.48);
            \draw[draw=red, ->] (0.2,-2.88) .. controls (0,-2.92) .. (-0.2,-2.88);
        \end{tikzpicture}
        \caption{Embedding of $\Gamma_M$ in $\ol{M} = S^3$ for $\id_{\tikztinypants}$.} \label{fig:id}
    \end{figure}

    The arguments for the remaining generators are identical: we have $\ol{M} = S^3, \ol{\parout M} = S^2$, and the embeddeding of $\Gamma_M$ exactly produces the diagram which computes the evaluation map.
\end{proof}

\subsection{Functoriality of $RT_{\mathcal{C},\sqrt{D}}$ and restriction to $\Bord{sig/2}$}

From the arguments above, we deduce
\begin{thm} \label{thm:rt-functor}
    Let $\mathcal{C}$ be a modular tensor category, and $\sqrt{D}$ a choice of square root of its global dimension. Then, $RT_{\mathcal{C},\sqrt{D}}$ defines a symmetric monoidal functor $\Bord{sig,dec} \rightarrow \KV$.
\end{thm}

Finally, we shall check that upon restriction to $\Bord{sig/2,dec}$, the choice of $\sqrt{D}$ does not matter, and so the choice of a modular tensor category $\mathcal{C}$ defines a functor $RT_\mathcal{C}: \Bord{sig/2,dec} \rightarrow \KV$.

\begin{prop} \label{prop:rt-well-def}
    On the 2-morphisms of $\Bord{sig/2,dec} \subset \Bord{sig,dec}$, $RT_{\mathcal{C},\sqrt{D}}$ and $RT_{\mathcal{C},-\sqrt{D}}$ define the same matrices of linear maps.
\end{prop}
\begin{proof}
    We compute the exponent of $\sqrt{D}$ in the definition of $RT_{\mathcal{C},\sqrt{D}}(M,n)$. It suffices to consider the case where $M$ is connected. The constants in (\ref{eq:def-rt}) contribute $\sqrt{D}^{n+\chi(\ol{\parout M})/2}$ while the $\tau(\ol{M},\Gamma_M)$ term contributes $\sqrt{D}^{-\sigma(X_L)-\# L-1}$ where $L$ is a surgery link for $\ol{M}$. So, we want to show that $$n + \frac{\chi(\ol{\parout M})}2 -\sigma(X_L)-\# L-1$$ is even. Substituting $n \equiv m(M) \pmod{2}$ where $m$ is given by the formula in Definition~\ref{defn:sig/2}, this reduces to showing that
    \begin{equation}\label{eq:sig/2-rt}
        b_1(\ol{M}) \equiv \sigma(X_L) + \# L \pmod2.
    \end{equation}

    Our result will now follow from applying \cite[Proposition 4.3]{partA} to a judicious choice of 2-morphisms in $\Bord{sig}$. Let $H$ be the union of disjoint tubular neighbourhoods of the components of $L$ in $S^3$, and let $\Sigma = \partial H$. Then, $M' := S^3\setminus \mathring{H}$ is a bordism from $\Sigma$ to $\emptyset$. Surgery on the link $L$ is given by gluing in some disjoint union of copies of $S^1\times D^2$ onto the boundary components of $M'$. These provide a bordism $H'$ from $\emptyset$ to $\Sigma$. We may consider $(H',0)$ and $(M',0)$ as morphisms from $(\emptyset,\emptyset)$ to $(\Sigma,H)$ and from $(\Sigma,H)$ to $(\emptyset,\emptyset)$, respectively. Then, by construction, we have $\ol{M'} = S^3$, and $\ol{H'}$ is a disjoint union of $\# L$ copies of $S^3$. The vertical composition of these two 2-morphisms is $(\ol{M}, - \sigma(V;A,B,C))$, where $V,A,B,C$ are as in Definition \ref{defn:comp} with $H'$ playing the role of $M$.
    
    On the other hand, by the definition of link trace, we have $X_L = D^4 \cup_H (D^4)^{\# L}$. Applying Wall's theorem (Theorem \ref{thm:wall}) to this union, we get $$\sigma(X_L) = \sigma(D^4) + \sigma((D^4)^{\sqcup \# L}) - \sigma(V;A,B,C) = - \sigma(V;A,B,C)$$ where $V,A,B,C$ are the exact same four vector spaces. We may thus apply \cite[Proposition 4.3]{partA} to the vertical composition of $H'$ and $M'$ to deduce that this Wall's invariant term satisfies $$-\sigma(V;A,B,C) \equiv m(\ol{M}) - m(H') - m(M') \pmod2.$$ By the definition of $m$ (Definition \ref{defn:sig/2}), we have:
    \begin{align*}
        m(\ol{M}) &= b_1(\ol{M}) + 1 + 0 + 0 = b_1(\ol{M}) + 1 \\
        m(H') &= 0 + \# L + 0 + 0 = \# L \\
        m(M') &= 0 + 1 + 1 - 1 = 1.
    \end{align*}
    Adding these up, we obtain (\ref{eq:sig/2-rt}).
\end{proof}

Thus, on the subbicategory $\Bord{sig/2,dec}$, both $RT_{\mathcal{C},\sqrt{D}}$ and $RT_{\mathcal{C},-\sqrt{D}}$ restrict to the same symmetric monoidal functor, which we denote simply as $RT_\mathcal{C}$. This concludes our proof of the following theorem:

\begin{thm} \label{thm:rt-functor2}
    Let $\mathcal{C}$ be a modular tensor category. Then, we have a symmetric monoidal functor $RT_\mathcal{C}: \Bord{sig/2,dec} \rightarrow \KV$.
\end{thm}

\section{The classification of representations of $\Bord{sig/2}$} \label{section:classify}

In separate work of the author \cite{partA}, the classification of symmetric monoidal functors $\Bord{sig/2}\rightarrow\Vect$ is proven by adapting the approach of Bartlett, Douglas, Schommer-Pries and Vicary \cite{BDSV4}. In particular, this method is dependent on a certain Cerf-theoretic result regarding the presentation of the oriented bordism bicategory; this was intended as the main result of \cite{BDSV1} and will be proven in upcoming work \cite{pres-bord} by building upon Cerf-theoretic results from \cite{haioun,filippos-thesis}. We record this below as Theorem~\ref{conj:pres-or}.

In this section, we will demonstrate how our construction of the extended Reshetikhin--Turaev TQFT allows us to replace the Cerf-theoretic portion of the Bartlett--Douglas--Schommer-Pries--Vicary approach (i.e. Theorem \ref{conj:pres-or}), giving us an alternate proof of the classification of linear representations of $\Bord{sig/2}$ (and similarly, of $\Bord{or}$ and $\Bord{sig}$).

\subsection{The generators-and-relations approach} \label{subsection:pres}

We first review the generators-and-rela\-tions approach taken in \cite{partA} to prove the classification of linear representation of $\Bord{sig/2}$. This is done via the construction of a certain bicategorical presentation. We reproduce the generators below; the relations are omitted as they are not relevant for our results below, though the ones of note are mentioned in the Remark following the definition.

\begin{defn}[{\cite[Definition 6.12]{partA}}] \smallbordisms
    The \emph{global half-signature presentation} $\mathcal{G}$ is the presentation with 
    \begin{itemize}
        \item Generating object:
            \begin{equation*}
                \begin{tz}
                    \node[Cyl, top, height scale=0]  at (0,0) {};
                \end{tz}
            \end{equation*}
        \item Generating 1-morphisms:
            \begin{equation*}
                \begin{tz}
                    \node[Pants, top, bot] (A) at (0,0) {};
                    \node[Copants, top, bot] (B) at (2,0) {};
                    \node[Cup, top] (C) at (4,0.1) {};
                    \node[Cap, bot] (D) at (6,-0.1) {};
                \end{tz}
            \end{equation*}
        \item Generating 2-morphisms:
            \begin{align*}
                &&
                \begin{tz}
                    \node[Pants, top, bot, wide] (A) at (0,0) {};
                    \node[Pants,  bot, anchor=belt] (B) at (A.leftleg) {};    
                    \node[Cyl, bot, anchor=top] at (A.rightleg) {}; 
                \end{tz}    
                &\rightleftdoublearrow{\alpha}{\alpha^{-1}}
                \begin{tz}
                    \node[Pants, top, bot, wide] (A) at (0,0) {};
                    \node[Pants,  bot, anchor=belt] (B) at (A.rightleg) {};    
                    \node[Cyl, bot, anchor=top] at (A.leftleg) {}; 
                \end{tz} 
                &
                \begin{tz}
                    \node[Pants, top, bot] (A) at (0,0) {};
                    \node[Cyl, bot, anchor=top] at (A.leftleg) {};
                    \node[Cup] at (A.rightleg) {};  
                \end{tz}
                &\rightleftdoublearrow{\rho}{\rho^{-1}}
                \begin{tz}
                    \node[Cyl, bot, top, tall] at (0,0) {};
                \end{tz}
                \rightleftdoublearrow{\lambda^{-1}}{\lambda}
                \begin{tz}
                    \node[Pants, top, bot] (A) at (0,0) {};
                    \node[Cyl, bot, anchor=top] at (A.rightleg) {};
                    \node[Cup] at (A.leftleg) {};   
                \end{tz}
            \\
                \begin{tz}
                    \node[Pants, top, bot] (A) at (0,0) {};
                \end{tz}
                &\rightleftdoublearrow{\beta}{\beta^{-1}}
                \begin{tz}
                    \node[Pants, top, bot] (A) at (0,0) {};
                    \node[BraidB, anchor=topleft, bot] at (A.leftleg) {};
                \end{tz}
                &
                \begin{tz}
                    \node[Cyl, top, bot] (A) at (0,0) {};
                \end{tz}
                &\rightleftdoublearrow{\theta}{\theta^{-1}}
                \begin{tz}
                    \node[Cyl, top, bot] (A) at (0,0) {};
                \end{tz}
                &
                \begin{tz}
                    \node[Copants, top, bot] (A) at (0,0) {};
                    \node[Pants, bot, anchor=belt] (B) at (A.belt) {};
                \end{tz} 
                &\longxdoubleto{\phi_1^{-1}}
                \begin{tz}
                    \node[Pants, top, bot] (A) at (0,0) {};
                    \node[Cyl, bot, anchor=top] (B) at (A.leftleg) {};
                    \node[Copants, bot, anchor=leftleg] (C) at (A.rightleg) {};
                    \node[Cyl, top, bot, anchor=bottom] (D) at (C.rightleg) {}; 
                \end{tz}
                &
                \begin{tz}
                    \node[Copants, top, bot] (A) at (0,0) {};
                    \node[Pants, bot, anchor=belt] (B) at (A.belt) {};
                \end{tz} 
                & \longxdoubleto{\phi_2^{-1}}
                \begin{tz}
                    \node[Pants, top, bot] (A) at (0,0) {};
                    \node[Cyl, bot, anchor=top] (B) at (A.rightleg) {};
                    \node[Copants, bot, anchor=rightleg] (C) at (A.leftleg) {};
                    \node[Cyl, top, bot, anchor=bottom] (D) at (C.leftleg) {}; 
                \end{tz}
            \\
                \begin{tz} 
                    \node[Cyl, tall, top, bot] (A) at (0,0) {};
                    \node[Cyl, tall, top, bot] (B) at (2*\cobwidth, 0) {};
                \end{tz}
                &\longxdoubleto{\eta}
                \begin{tz} 
                    \node[Pants, bot] (A) at (0,0) {};
                    \node[Copants, top, bot, anchor=belt] at (A.belt) {};
                \end{tz}
                &
                \begin{tz} 
                    \node[Pants, top, bot] (A) at (0,0) {};
                    \node[Copants, bot, anchor=leftleg] at (A.leftleg) {};
                \end{tz}
                &\longxdoubleto{\epsilon}
                \begin{tz} 
                    \node[Cyl, top, bot, tall] (A) at (0,0) {};
                \end{tz}
                &
                \begin{tz}
                    \draw[green] (0,0) rectangle (0.6, -0.6);  
                \end{tz}
                &\longxdoubleto{\nu}
                \begin{tz}
                    \node[Cap, bot] (A) at (0,0) {};
                    \node[Cup] at (0,0) {};
                \end{tz}
                &
                \begin{tz}
                    \node[Cup, top] (A) at (0,0) {};
                    \node[Cap, bot] (B) at (0,-2*\cobheight) {};
                \end{tz}
                &\longxdoubleto{\mu}
                \begin{tz}
                    \node[Cyl, top, bot, tall] (A) at (0,0) {};
                \end{tz}
            \\
                \begin{tz} 
                    \node[Pants, bot] (A) at (0,0) {};
                    \node[Copants, top, bot, anchor=belt] at (A.belt) {};
                \end{tz}
                &\longxdoubleto{\eta ^\dag}
                \begin{tz} 
                    \node[Cyl, tall, top, bot] (A) at (0,0) {};
                    \node[Cyl, tall, top, bot] (B) at (2*\cobwidth, 0) {};
                \end{tz}
                &
                \begin{tz} 
                    \node[Cyl, top, bot, tall] (A) at (0,0) {};
                \end{tz}
                &\longxdoubleto{\epsilon ^\dag}
                \begin{tz} 
                    \node[Pants, top, bot] (A) at (0,0) {};
                    \node[Copants, bot, anchor=leftleg] at (A.leftleg) {};
                \end{tz}
            &
                \begin{tz}
                    \node[Cap, bot] (A) at (0,0) {};
                    \node[Cup] at (0,0) {};
                \end{tz}
                &\longxdoubleto{\nu ^\dag}{}
                \begin{tz}
                    \draw[green] (0,0) rectangle (0.6, -0.6);  
                \end{tz}
                &
                \begin{tz}
                    \node[Cyl, top, bot, tall] (A) at (0,0) {};
                \end{tz}
                &\longxdoubleto{\mu ^\dag}
                \begin{tz}
                    \node[Cup, top] (A) at (0,0) {};
                    \node[Cap, bot] (B) at (0,-2*\cobheight) {};
                \end{tz}
            \\ &&
                \begin{tz}
                    \draw[green] (0,0) rectangle (\cobwidth, \cobwidth);
                \end{tz}
                &\rightleftdoublearrow{\zeta}{\zeta^{-1}}
                \begin{tz}
                    \draw[green] (0,0) rectangle (\cobwidth, \cobwidth);
                \end{tz}
            &
                \begin{tz}
                    \node[Cyl, top, bot] (A) at (0,0) {};
                \end{tz}
                &\rightleftdoublearrow{z}{z^{-1}}
                \begin{tz}
                    \node[Cyl, top, bot] (A) at (0,0) {};
                \end{tz}
            \end{align*}
    \end{itemize}
    subject to a number of relations, which we omit.
\end{defn}
\begin{rk}
    Each pair of generators of the form $\gamma,\gamma^{-1}$ has relations declaring them to be inverses of each other. The inverse Frobeniusators $\phi_1^{-1},\phi_2^{-1}$ have relations (the rigidity and additional rigidity relations) which express them as a composition of the other generators, and which declare them to be the inverses of compositions of the other generators which are denoted $\phi_1,\phi_2$. All of these make up the \emph{invertible generators}. The remaining generators are the \emph{non-invertible generators}.
\end{rk}

Following the corresponding construction for the oriented bordism bicategory given in \cite{BDSV2}, it is shown in \cite{partA} that there is a symmetric monoidal functor from the symmetric monoidal bicategory $\mathbf{F}(\mathcal{G})$ generated by the global half-signature presentation to the half-signature bordism bicategory.
\begin{prop}[{\cite[Proposition 6.16]{partA}}] \label{prop:modg}
    There exists a symmetric monoidal functor $|-|_\mathcal{G}: \mathbf{F}(\mathcal{G})\rightarrow \Bord{sig/2}$.
\end{prop}
\begin{rk}
    We refer the reader to \cite[Constructions 6.14-6.15]{partA} for the explicit construction of this functor. Roughly speaking, $\zeta,z$ correspond to shifting the signature by $2$, the remaining invertible 2-morphisms are sent to the obvious mapping cylinders of diffeomorphisms, and the non-invertible 2-morphisms are sent to various handle attachments. For each of the non-invertible generators, the underlying bordism $M$ has $\ol{M} = S^3$.
\end{rk}
\begin{rk}
    One should note that $|-|_\mathcal{G}$ sends the generating object and 1-morphisms to the standard circle and standard pants, copants, cup and cap respectively, and so this construction naturally factors through $\Bord{sig/2,dec}$. Henceforth, we abuse notation and regard $|-|_\mathcal{G}$ as a symmetric monoidal functor $\mathbf{F}(\mathcal{G})\rightarrow \Bord{sig/2,dec}$.
\end{rk}

In \cite[Theorem B]{partA}, it is further deduced from the corresponding Cerf-theoretic result for $\Bord{or}$ (Theorem \ref{conj:pres-or} below, which is the main theorem of the upcoming paper \cite{pres-bord}) that this $|-|_\mathcal{G}$ is an equivalence of symmetric monoidal bicategories. It is, however, far easier to show that $|-|_\mathcal{G}$ is essentially surjective, essentially full and full; we record a proof in Proposition \ref{prop:modg-full} for the sake of completeness. In the course of doing so, we also prove in Proposition \ref{prop:modo-full} that the same holds for the symmetric monoidal functor $|-|_\mathcal{O}$ described below.

There is one other presentation of note, which upcoming work \cite{pres-bord} will show presents the oriented bordism bicategory $\Bord{or}$. The following is equivalent to the definition given in \cite{BDSV4} in view of the remarks following \cite[(6.1)]{partA}.
\begin{defn}[{\cite[Definition 3.11]{BDSV4}}]
    The \emph{anomaly-free modular presentation} $\mathcal{O}$ is a 2-extension of $\mathcal{G}$ with the added relations that $z=\id$ and $\zeta = \id$.
\end{defn}
\begin{rk}
    In particular, there is a natural functor $\mathbf{F}(\mathcal{G}) \rightarrow \mathbf{F}(\mathcal{O})$ sending $z,\zeta$ to identity 2-morphisms and the remaining generators to their counterparts.
\end{rk}
\begin{thm}[{\cite{BDSV2,haioun,filippos-thesis,pres-bord}}]\label{conj:pres-or}
    There exists an equivalence of symmetric monoidal bicategories $|-|_\mathcal{O}: \mathbf{F}(\mathcal{O}) \rightarrow \Bord{or}$ between the symmetric monoidal bicategory generated by the anomaly-free modular presentation and the oriented bordism bicategory.
\end{thm}
\begin{rk}
    This is stated as \cite[Theorem 3.12]{BDSV4}, and can be shown by combining the intended main result of \cite{BDSV1} with the arguments of \cite{BDSV2}. Much of the proof of the former result has been done in \cite{haioun,filippos-thesis}; upcoming work of Bartlett, Douglas and Sytilidis \cite{pres-bord} will complete this.
\end{rk}
\begin{rk}
    It is, however, far easier to write down the images of the generators under $|-|_\mathcal{O}$ and check that this indeed defines a symmetric monoidal functor; see \cite[Definition 3]{BDSV2} and \cite[Construction 6.14]{partA}. As with $|-|_\mathcal{G}$, the generating object and 1-morphisms are sent to the manifolds that they depict. This thus factors through $\Bord{or,dec}$; we henceforth $|-|_\mathcal{O}:\mathbf{F}(\mathcal{O})\rightarrow\Bord{or,dec}$ to be the symmetric monoidal functor as described in \cite[Construction 6.14]{partA}.
\end{rk}

By construction, the functor $|-|_\mathcal{G}$ defined in \cite{partA} is compatible with the $|-|_\mathcal{O}$ defined in \cite{BDSV2}.
\begin{prop}[{\cite[Proposition 6.16]{partA}}] \label{prop:go-compat}
    The following diagram commutes:
    \[\begin{tikzcd}
        {\mathbf{F}(\mathcal{G})} && {\Bord{sig/2,dec}} \\
        \\
        {\mathbf{F}(\mathcal{O})} && {\Bord{or,dec}}
        \arrow["{|-|_\mathcal{G}}", from=1-1, to=1-3]
        \arrow[from=1-1, to=3-1]
        \arrow[from=1-3, to=3-3]
        \arrow["{|-|_\mathcal{O}}", from=3-1, to=3-3]
    \end{tikzcd}\]
    Here, the vertical arrows are the natural quotients.
\end{prop}

Following methods of \cite{BDSV4}, it is then shown in \cite{partA} that linear representations of $\mathbf{F}(\mathcal{G})$ are exactly classified by modular tensor categories.
\begin{thm}[{\cite[Theorem 7.16]{partA}}] \label{thm:rep-fg}
    Symmetric monoidal functors $\mathbf{F}(\mathcal{G})\rightarrow\Vect$ are classified by finite direct sums of modular tensor categories whose anomalies are equal.
\end{thm}

We now provide a sketch of both directions of this bijection. Starting from a symmetric monoidal functor $Z:\mathbf{F}(\mathcal{G})\rightarrow\Vect$, the images of $\tikztinycup,\tikztinypants$ as well as $\alpha,\rho,\lambda,\beta,\theta$ naturally endow $Z(\tikztinycirc)$ with the structure of a linear balanced braided monoidal category. Following arguments of \cite{BDSV4}, it was then shown in \cite[Corollary 7.10]{partA} that this must in fact be a finite direct sum of modular tensor categories. A simple computation then shows that each direct summand must have the same anomaly.

Conversely, given the data of a modular tensor category $\mathcal{C}$, one may directly construct a symmetric monoidal functor $Z_\mathcal{C}: \mathbf{F}(\mathcal{G}) \rightarrow \Vect$ by writing down the image of every generator and checking that the relations hold. This in fact factors through $\KV$. In the general case, one may then define $Z_\mathcal{C}$ by taking direct sums. The description of this in \cite{BDSV4} is given in terms of ``internal string diagrams''; we shall now present a reformulation of the symmetric monoidal functor $Z_\mathcal{C}: \mathbf{F}(\mathcal{G}) \rightarrow \KV$ in our language of ribbon graphs.

Let $\mathcal{C}$ be a modular tensor category with $s$ isomorphism classes of simple objects. Then, $Z_\mathcal{C}$ sends the generating object to $\tikztinycirc$ to the object in $\KV$ given by the integer $s$. To each 1-morphism of $\mathbf{F}(\mathcal{G})$, we may assign an associated ribbon graph as described in Example \ref{eg:internal-graph}. Then, $Z_\mathcal{C}$ sends the 1-morphism to the matrix of spaces of labels of this graph, exactly as in the definition of $RT_\mathcal{C}$. 

Now, the generating 2-morphisms are sent to matrices of linear maps between the spaces of labels of a source ribbon graph and a target ribbon graph. In most cases, it is simpler to apply the isomorphism $\bigcirc$ from Lemma \ref{lem:tree} to the source and target ribbon graphs, and then specify a matrix of linear maps between the spaces of labels of ribbon graphs which have been contracted along a spanning tree. With this in mind, the images of the generating 2-morphisms are given by the following three results.

\begin{prop}[{\cite[Propositions 4.3-4.7]{BDSV4}}] \label{prop:r-gen}
    The symmetric monoidal functor $Z_\mathcal{C}:\mathbf{F}(\mathcal{P}) \rightarrow \KV$ sends the generating 2-morphisms $\gamma$ of $\mathcal{R}$ to the following matrices of linear maps:
    \begin{itemize}
        \item For each of $\gamma=\alpha,\rho,\lambda,\phi_1,\phi_2$, upon contracting along the unique internal edge, the source and target ribbon graphs may be identified with each other. Then, $Z_\mathcal{C}(\gamma)$ act as matrices of identity linear maps on the spaces of labels.
        \item For $\gamma=\theta$, the only nonzero entries are given by the linear maps $$Z_\mathcal{C}(\theta)_i^i: \Hom(X_i,X_i) \rightarrow \Hom(X_i,X_i)$$ sending $\id_{X_i}$ to the twist $\theta_{X_i}$.
        \item For $\gamma=\beta$, each $$Z_\mathcal{C}(\beta)_{i_1i_2}^j: \Hom(X_{i_1}\otimes X_{i_2}, X_j) \rightarrow \Hom(X_{i_2}\otimes X_{i_1}, X_j)$$ is the linear map given by sending $f \mapsto f\circ\beta_{X_{i_2}\otimes X_{i_1}}^{-1}$.
        \item For $\gamma=\nu$, upon contracting the target ribbon graph along its one edge we are left with defining a linear map $k \rightarrow \Hom(\mathbbm{1},\mathbbm{1})$; $\nu$ is sent to the linear map sending $1$ to $\id_{\mathbbm{1}}$.
        \item For $\gamma=\mu$, the only nonzero entry is $$Z_\mathcal{C}(\mu)_1^1: \Hom(\mathbbm{1},\mathbbm{1})\otimes\Hom(\mathbbm{1},\mathbbm{1}) \rightarrow k$$ which sends $\id_{\mathbbm{1}}\otimes \id_{\mathbbm{1}}$ to $1$.
        \item For $\gamma=\epsilon$, the only nonzero entries are the linear maps $$Z_\mathcal{C}(\epsilon)_i^i: \bigoplus_{k_1,k_2}\Hom(X_i,X_{k_1}\otimes X_{k_2})\otimes\Hom(X_{k_1}\otimes X_{k_2},X_i) \rightarrow k$$ given by composition and identifying $\id_{X_i}$ with $1$.
        \item For $\gamma=\eta$, contract the target ribbon graph along its one internal edge. The only nonzero entries are the linear maps $Z_\mathcal{C}(\eta)_{ij}^{ij}: k \rightarrow \Hom(X_{i}\otimes X_{j}, X_{i}\otimes X_{j})$ sending $1$ to the identity morphism $\id_{X_{i}\otimes X_{j}}$.
    \end{itemize}
    For the unspecified entries, either the source or the target ribbon graph will have the zero vector space as its space of labels, and thus the linear map in this entry has to be the zero linear map.
\end{prop}
\begin{prop}[{\cite[Propositions 5.4-5.6]{BDSV4}}] \label{prop:noninv-dag}
    Let $p_+$ be as in Definition \ref{defn:cat-consts}. $Z_\mathcal{C}(\eta^\dag),Z_\mathcal{C}(\epsilon^\dag),Z_\mathcal{C}(\nu^\dag),Z_\mathcal{C}(\mu^\dag)$ act as follows:
    \begin{itemize}
        \item For $\nu^\dag$, contract the source ribbon graph along its one edge. Then, $$Z_\mathcal{C}(\nu^\dag):\Hom(\mathbbm{1},\mathbbm{1}) \rightarrow k$$ is the linear map sending $\id_{\mathbbm{1}}$ to $\frac1{p_+}$.
        \item For $\mu^\dag$, the only nonzero entry is $$Z_\mathcal{C}(\mu^\dag)_1^1: k\rightarrow \Hom(\mathbbm{1},\mathbbm{1})\otimes \Hom(\mathbbm{1},\mathbbm{1})$$ which sends $1$ to $p_+\id_{\mathbbm{1}}\otimes\id_{\mathbbm{1}}$.
        \item For $\epsilon^\dag$, $Z_\mathcal{C}(\epsilon^\dag)_i^j$ is only nonzero when $i=j$. Contract the target ribbon graph along the edge going through the left tube of the torus, and then by applying Lemma \ref{lem:cyc-shift}, we are left with defining linear maps $$Z_\mathcal{C}(\epsilon^\dag)_i^i: k \rightarrow \bigoplus_k \Hom(X_i\otimes X_k^*, X_i\otimes X_k^*).$$ These send $1$ to $\frac1{p_+}\sum_k\dim k\id_{X_i\otimes X_k^*}$.
        \item For $\eta^\dag$, contract the source ribbon graph along its internal edge. The only nonzero entries are the linear maps $$Z_\mathcal{C}(\eta^\dag)_{ij}^{ij}: \Hom(X_{i}\otimes X_{j},X_{i}\otimes X_{j})\rightarrow k$$ sending each $f$ to $p_+\frac{\tr(f)}{\dim i \dim j}$.
    \end{itemize}
    The remaining entries are necessarily the zero linear map as either the source or target space of labels is zero.
\end{prop}
\begin{lem}[{\cite[Lemma 7.14]{partA}}] \label{lem:z-anomaly}
    $Z_\mathcal{C}(\zeta)$ and $Z_\mathcal{C}(z)$ act as multiplication by the anomaly $\frac{p_+}{p_-}$.
\end{lem}

Putting together Proposition \ref{prop:modg} and (one direction of) Theorem \ref{thm:rep-fg}, we have thus produced another way to obtain a linear representation of $\Bord{sig/2}$ from the data of a modular tensor category $\mathcal{C}$. Now, we show that this in fact agrees with our construction of $RT_\mathcal{C}$.

\begin{prop}\label{prop:mtc-rep-mtc}
    Let $\mathcal{C}$ be a modular tensor category. Then we have an equivalence $RT_{\mathcal{C}}\circ |-|_\mathcal{G} \simeq Z_{\mathcal{C}}$ of functors valued in $\KV$.
\end{prop}
\begin{proof}
    It suffices to compare these two functors on the generators of $\mathcal{G}$. By construction, they agree at the object and 1-morphism levels: both send the generating object $\tikztinycirc$ to the object in $\KV$ given by $s$, the number of isomorphism classes of $\mathcal{C}$; and 1-morphisms to spaces of labels of the associated ribbon graphs.
    
    We now check the generating 2-morphisms one by one (with the exception of $\phi_1^{-1},\phi_2^{-1}$, which may be written in terms of the other generators). For each generator $\gamma$, let $(M,n)=|\gamma|_\mathcal{G}$ and $\Sigma = \parin M, \Sigma' = \parout M$. Then, $RT_{\mathcal{C}}(|\gamma|_\mathcal{G})_{\mathbf{i}}^\mathbf{j}$ and $Z_{\mathcal{C}}(\gamma)_{\mathbf{i}}^\mathbf{j}$ are both linear maps $\mathcal{V}_{\mathbf{K}_{\mathbf{i},\mathbf{j}}}(\Gamma_{\Sigma}) \rightarrow \mathcal{V}_{\mathbf{K}_{\mathbf{i},\mathbf{j}}'}(\Gamma_{\Sigma'})$. We seek to show that these are equal.

    The images of the respective generators $\gamma$ under $Z_\mathcal{C}$ are given in Proposition \ref{prop:r-gen}, Proposition \ref{prop:noninv-dag} and Lemma \ref{lem:z-anomaly}. For most of these, $Z_\mathcal{C}(\gamma)_{\mathbf{i}}^\mathbf{j}$ is defined on spaces of labels of ribbon graphs which have been contracted along spanning subtrees, i.e. as a linear map $\mathcal{V}_{\tilde{\mathbf{K}}_{\mathbf{i},\mathbf{j}}}(\tilde{\Gamma}_{\Sigma}) \rightarrow \mathcal{V}_{\tilde{\mathbf{K}}_{\mathbf{i},\mathbf{j}}'}(\tilde{\Gamma}_{\Sigma'})$. For these generators, we compute $RT_\mathcal{C}(|\gamma|_\mathcal{G})_{\mathbf{i}}^\mathbf{j}$ on these contracted ribbon graphs, as in Lemma \ref{lem:contract-tree-rt}.
    
    Now, recall that each $RT_\mathcal{C}(M,n)_\mathbf{i}^\mathbf{j}$ is defined in a dual sense by defining a linear map $(RT_\mathcal{C}(M,n)_\mathbf{i}^\mathbf{j})_{K,K'}: \mathcal{V}_K(\Gamma_\Sigma) \otimes \mathcal{V}_{K'}(-\Gamma_{\Sigma'}) \rightarrow k$ for each ribbon-colourings $K,K'$ of $\Gamma_\Sigma,\Gamma_{\Sigma'}$ extending the partial ribbon-colourings $\mathbf{K}_{\mathbf{i},\mathbf{j}},\mathbf{K}_{\mathbf{i},\mathbf{j}}'$ determined by $\mathbf{i},\mathbf{j}$. This is in turn given in (\ref{eq:def-rt}) by computing the invariant $\tau$ of the labelled embedded ribbon graph $\Gamma_M$ which is the result of gluing $\Gamma_\Sigma,\Gamma_{\Sigma'}$ together, and multiplying it by some scalar terms. Note that for every generator except for $\zeta$ (which we will handle separately), the underlying bordism $M$ has $\ol{M} = S^3$, and so $\Gamma_M$ is already embedded in $S^3$ and there is no surgery link $L$. Thus, the formula in (\ref{eq:def-rt}) may be rewritten as 
    \begin{equation} \label{eq:rt-gen}
        (RT_\mathcal{C}(M,n)_\mathbf{i}^\mathbf{j})_{K,K'}(\mathbf{f}\otimes\mathbf{g}) = \pa{\frac{\sqrt{D}}{p_-}}^n \sqrt{D}^{\chi(\ol{\Sigma'})/2 - 1} \dim_\mathrm{int}(K') F(\Gamma_M,K\cup K',\mathbf{f}\otimes\mathbf{g}).
    \end{equation}
    for each $\mathbf{f}\otimes\mathbf{g} \in \mathcal{V}_K(\Gamma_\Sigma) \otimes \mathcal{V}_{K'}(-\Gamma_{\Sigma'}) \cong \mathcal{V}_{K_M}(\Gamma_M)$. In other words, it suffices to compute the operator invariant of the labelled embedded ribbon graph $\Gamma_M$ in $S^3$. Likewise, if we are instead working with contracted ribbon graphs, by Lemma \ref{lem:contract-tree-rt}, the corresponding formula is given by replacing $K,K',\Gamma_M$ in (\ref{eq:rt-gen}) with $\tilde{K},\tilde{K}',\tilde{\Gamma}_M$ respectively.

    Taking direct sums over all $K,K'$, we thus obtain the linear map $$\left\langle RT_\mathcal{C}(M,n)_\mathbf{i}^\mathbf{j}(-), - \right\rangle_{\Gamma_{\Sigma'}}: \mathcal{V}_{\mathbf{K}_{\mathbf{i},\mathbf{j}}}(\Gamma_\Sigma) \otimes \mathcal{V}_{\mathbf{K}_{\mathbf{i},\mathbf{j}}'}(\Gamma_{\Sigma'}) \rightarrow k.$$ If we are working with contracted ribbon graphs, then replace $\mathbf{K}_{\mathbf{i},\mathbf{j}},\mathbf{K}_{\mathbf{i},\mathbf{j}}',\Gamma_{\Sigma'}$ respectively with $\tilde{\mathbf{K}}_{\mathbf{i},\mathbf{j}},\tilde{\mathbf{K}}_{\mathbf{i},\mathbf{j}}',\tilde{\Gamma}_{\Sigma'}$. It suffices to check that the linear map $RT_\mathcal{C}(M,n)_\mathbf{i}^\mathbf{j}$ defined in this way agrees with $Z_\mathcal{C}(\gamma)_\mathbf{i}^\mathbf{j}$. Equivalently, it suffices to show that for $\mathbf{f} \in \mathcal{V}_{\mathbf{K}_{\mathbf{i},\mathbf{j}}}(\Gamma_\Sigma)$ and $\mathbf{g} \in \mathcal{V}_{\mathbf{K}_{\mathbf{i},\mathbf{j}}'}(-\Gamma_{\Sigma'})$, we have
    \begin{equation} \label{eq:rtz-dual}
        \left\langle Z_\mathcal{C}(\gamma)_\mathbf{i}^\mathbf{j}(\mathbf{f}), \mathbf{g} \right\rangle_{\Gamma_{\Sigma'}} = \left\langle RT_\mathcal{C}(M,n)_\mathbf{i}^\mathbf{j}(\mathbf{f}), \mathbf{g} \right\rangle_{\Gamma_{\Sigma'}}
    \end{equation}
    (or, when relevant, the corresponding version with tildes).

    For most of the generators (all $\gamma\ne\epsilon,\epsilon^\dag$), after contracting the ribbon graphs, $\tilde{\Gamma}_\Sigma,\tilde{\Gamma}_{\Sigma'}$ have no internal edges, and hence the $\tilde{\mathbf{K}}_{\mathbf{i},\mathbf{j}},\tilde{\mathbf{K}}_{\mathbf{i},\mathbf{j}}'$ are in fact ribbon-colourings. In these cases, by (\ref{eq:rt-nodual}), the right-hand side of (\ref{eq:rtz-dual}) may be written as the right-hand side of (\ref{eq:rt-gen}) with $\tilde{K} = \tilde{\mathbf{K}}_{\mathbf{i},\mathbf{j}}, \tilde{K}'= \tilde{\mathbf{K}}_{\mathbf{i},\mathbf{j}}'$.

    Before we proceed, we first simplify the scalar term in (\ref{eq:rt-gen}). We compute the values of $\pa{\frac{\sqrt{D}}{p_-}} \sqrt{D}^{\chi(\ol{\Sigma'})/2 - 1}$ for the generators of $\mathcal{G}$:
    \begin{center}
        \renewcommand{\arraystretch}{1.5}
        \begin{tabular}{c | c c | c}
            Generators & $n$ & $\chi(\Sigma')$ & $\pa{\frac{\sqrt{D}}{p_-}} \sqrt{D}^{\chi(\Sigma')/2 - 1}$ \\
            \hline
            $\alpha,\rho,\lambda,\beta,\theta,\eta,\epsilon,\mu,\nu$ & $0$ & $2$ & $1$ \\
            $\eta^\dag,\mu^\dag$ & $1$ & $4$ & $\frac{D}{p_-}$ \\
            $\epsilon^\dag,\nu^\dag$ & $-1$ & $0$ & $\frac{p_-}D$ \\
            $z$ & $2$ & $2$ & $\frac{D}{p_-^2}$
        \end{tabular}
    \end{center}
    By \cite[Corollary 3.1.10]{BK}, we have $p_+ = \frac{D}{p_-}$, and so the constants for $\eta^\dag,\epsilon^\dag,\mu^\dag,\nu^\dag,z$ here exactly agree with the constants in the definition of $Z_\mathcal{C}$ (Proposition \ref{prop:noninv-dag} and Lemma \ref{lem:z-anomaly}). In particular, $RT_\mathcal{C}(|z|_\mathcal{G})$ is simply multiplication by $\frac{p_+}{p_-}$, which agrees with $Z_\mathcal{C}(z)$. For the remaining generators, we have to compute the operator invariant of the labelled embedded ribbon graph $\Gamma_M$.

    We may now verify that (\ref{eq:rtz-dual}) holds for each generating 2-morphism $\gamma$. This will be done in Lemmas \ref{lem:inv1} and \ref{lem:inv2} for the invertible generators $\gamma = \alpha,\lambda,\rho,\beta,\theta$, in Lemma \ref{lem:eps} for $\gamma=\epsilon,\mu,\mu^\dag$, and in Lemma \ref{lem:noninv} for the remaining non-invertible 2-morphisms.

    Finally, in the case $\gamma=\zeta$, we have $M=\emptyset$ and $n=2$. Then as $\chi(\ol{\Sigma'}) = 0$ and $\tau(\emptyset,\emptyset) = 1$, we simply have that $RT_\mathcal{C}(|\zeta|_\mathcal{G})$ acts as multiplication by $\pa{\frac{\sqrt{D}}{p_-}}^2 = \frac{p_+}{p_-},$ which indeed agrees with $Z_\mathcal{C}(\zeta)$.
\end{proof}

It remains to check that (\ref{eq:rtz-dual}) holds for the remaining generating 2-morphisms. Throughout, $M$ refers to the underlying bordism of the 2-morphism in consideration.

\begin{lem}\label{lem:inv1}
    The equation (\ref{eq:rtz-dual}) holds for $\gamma=\alpha,\lambda,\rho$.
\end{lem}
\begin{proof}
    Upon contracting along the unique internal edges, $Z_\mathcal{C}(\alpha)_{i_1i_2i_3}^j$ is simply the identity linear map on $\Hom(X_{i_1}\otimes X_{i_2} \otimes X_{i_3}, X_j)$. On the other hand, the embedding of the contracted ribbon graph $\tilde{\Gamma}_M$ is exactly the ribbon diagram that encodes the evaluation map of $\Hom(X_{i_1}\otimes X_{i_2} \otimes X_{i_3}, X_j)$ and its dual; this is analogous to the proof of Lemma \ref{lem:rt-id}. Hence, $RT_{\mathcal{C}}(|\alpha|_\mathcal{G}) = Z_{\mathcal{C}}(\alpha)$.
    
    The arguments for $\gamma=\rho,\lambda$ are similar.
\end{proof}
For all subsequent ribbon diagrams depicting embeddings of $\Gamma_M$ or $\tilde{\Gamma}_M$, the left halves of the diagrams correspond to $\Gamma_\Sigma$ (or $\tilde{\Gamma}_\Sigma$) and the right halves correspond to $-\Gamma_{\Sigma'}$ (or $-\tilde{\Gamma}_{\Sigma'}$).
\begin{lem} \label{lem:inv2}
    The equation (\ref{eq:rtz-dual}) holds for $\gamma=\beta,\theta$.
\end{lem}
\begin{proof}
    The embedded ribbon graphs $\Gamma_M$ are as drawn in Figure \ref{fig:beta-theta}. We see that these are exactly the graphs that encode $\left\langle Z_\mathcal{C}(\beta)_{\textcolour{red}{i_1}\textcolour{Green}{i_2}}^{\textcolour{blue}{j}}(f),g\right\rangle_{\Gamma_{\Sigma'}}$ and $\left\langle Z_\mathcal{C}(\theta)_{\textcolour{blue}{i}}^{\textcolour{blue}{i}}(\id_{X_{\textcolour{blue}{i}}}),\id_{X_{\textcolour{blue}{i}}}\right\rangle_{\Gamma_{\Sigma'}}$ respectively and so (\ref{eq:rtz-dual}) holds for $\gamma=\beta,\theta$.
    \begin{figure}[hbt!]
        \begin{tikzpicture}
            \draw[draw=blue, double=white, double distance=5pt] (-1,-1.6) to [out=up,in=left] (0,-0.6) to [out=right,in=up] (1,-1.6);
            \begin{knot}[clip width=2]
                \strand[draw=red, double=white, double distance=5pt] (-0.7,-2) to [out=down, in=up] (-1.3,-3) to [out=down, in=left] (0,-4.3) to [out=right, in=down] (1.3,-2);
                \strand[draw=Green, double=white, double distance=5pt] (-1.3,-2) to [out=down, in=up] (-0.7,-3) to [out=down, in=left] (0,-3.7) to [out=right, in=down] (0.7,-2);
                \flipcrossings{1}
            \end{knot}
            \draw[draw=blue, ->] (-0.2,-0.62) .. controls (0,-0.58) .. (0.2,-0.62);
            \draw[draw=red, ->] (0.2,-4.28) .. controls (0,-4.32) .. (-0.2,-4.28);
            \draw[draw=Green, ->] (0.2,-3.68) .. controls (0,-3.72) .. (-0.2,-3.68);
            \draw[rounded corners] (-1.6,-2) rectangle (-0.4,-1.6) node[pos=.5] {$f$};
            \draw[rounded corners] (0.4,-2) rectangle (1.6,-1.6) node[pos=.5] {$g$};
            \node[] at (0,-0.3) {$\textcolour{blue}{j}$};
            \node[] at (0,-4) {$\textcolour{Green}{i_2}$};
            \node[] at (0,-4.6) {$\textcolour{red}{i_1}$};

            \begin{knot}[clip width=2, consider self intersections, end tolerance=2pt]
                \strand[draw=blue, double=white, double distance=5pt] (5,-0.9) to [out=left, in=up] (4,-1.9) to [out=down, in=left] (5,-2.9) to [out=right, in=right] (5,-1.9) to [out=left, in=up] (4,-2.9) to [out=down, in=left] (5,-3.9) to [out=right, in=down] (6.5,-2.4) to [out=up, in=right] (5,-0.9);
                \flipcrossings{1}
            \end{knot}
            \draw[draw=blue, ->] (4,-1.9) .. controls (4,-1.7) .. (4.1,-1.5);
            \node[] at (3.7,-1.7) {$\textcolour{blue}{j}$};
        \end{tikzpicture}
        \caption{Embeddings of $\Gamma_M$ for the 2-morphisms $\beta$ and $\theta$.} \label{fig:beta-theta}
    \end{figure}
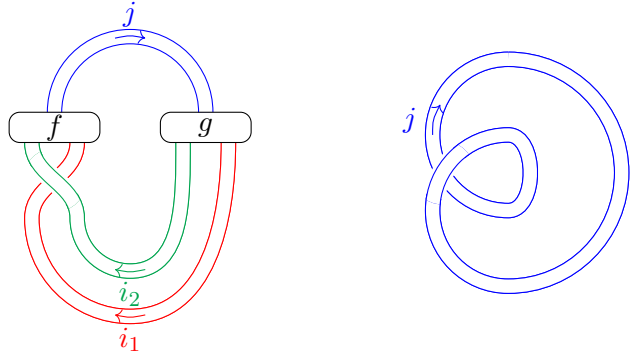
\end{proof}
For the non-invertible generators, we may likewise write both sides of (\ref{eq:rtz-dual}) in terms of the operator invariant of the same ribbon diagram.
\begin{lem}\label{lem:eps}
    The equation (\ref{eq:rtz-dual}) holds for $\gamma=\epsilon,\mu,\mu^\dag$.
\end{lem}
\begin{proof}
    let $f_1,f_2$ be the labels of the vertices of $\Gamma_\Sigma$. Then, we have
    \begin{equation*}
        \begin{tikzpicture}
            \draw[draw=blue, double=white, double distance=5pt] (3,0.7) to [out=up,in=left] (3.6,1) to [out=right,in=right] (3.6,-1) to [out=left,in=down] (3,-0.7);
            \draw[draw=red, double=white, double distance=5pt] (2.7,-0.3) -- (2.7,0.3);
            \draw[draw=Green, double=white, double distance=5pt] (3.3,-0.3) -- (3.3,0.3);
            \draw[draw=red, ->] (2.7,-0.2) -- (2.7,0.2);
            \draw[draw=Green, ->] (3.3,-0.2) -- (3.3,0.2);
            \draw[draw=blue, ->] (4.18,0.2) .. controls (4.22,0) .. (4.18,-0.2);
            \draw[rounded corners] (2.4,-0.7) rectangle (3.6,-0.3) node[pos=.5] {$f_2$};
            \draw[rounded corners] (2.4,0.3) rectangle (3.6,0.7) node[pos=.5] {$f_1$};
            \node[] at (4.42,0) {$\textcolour{blue}{i}$};

            \node[] at (-1.05,0) {$\left\langle Z_\mathcal{C}(\epsilon)_{\textcolour{blue}{i}}^{\textcolour{blue}{i}} (f_1\otimes f_2), \id_{X_{\textcolour{blue}{i}}} \right\rangle_{\Gamma_{\Sigma'}} = \tr(f_1\circ f_2) = $};
            \node[] at (7.45,0) {$= \left\langle RT_\mathcal{C}(|\epsilon|_\mathcal{G})_{\textcolour{blue}{i}}^{\textcolour{blue}{i}} (f_1\otimes f_2), \id_{X_{\textcolour{blue}{i}}} \right\rangle_{\Gamma_{\Sigma'}}$.};
        \end{tikzpicture}
    \end{equation*}
    When $i\ne j$, $Z_\mathcal{C}(\epsilon)_i^j = RT_\mathcal{C}(|\epsilon|_\mathcal{G})_i^j$ as they are both the zero linear map (since their codomain is the zero vector space). Thus, we have shown (\ref{eq:rtz-dual}) for $\gamma=\epsilon$.
    
    The computations for $\gamma=\mu,\mu^\dag$ are similar (delete the red and green ribbons from the diagram above).
\end{proof}
For the remaining non-invertible generators $\gamma$, we are instead working with the contracted ribbon graph $\tilde{\Gamma}_M$. Other than that, the computations are similar to those in Lemma \ref{lem:eps}
\begin{lem}\label{lem:noninv}
    The equation (\ref{eq:rtz-dual}) holds for $\gamma=\nu,\nu^\dag,\eta,\eta^\dag,\epsilon^\dag$.
\end{lem}
\begin{proof}
    The cases $\gamma=\nu,\nu^\dag$ are obvious as $\tilde{\Gamma}_M$ consists of a single vertex and no edges.
    
    For $\gamma=\eta$, the only nonzero case is $\mathbf{i} = \mathbf{j}$ as otherwise, the domain of both linear maps is the zero vector space. For any $g\in\Hom(X_{\textcolour{red}{j}}^* \otimes X_{\textcolour{blue}{i}}^*, X_{\textcolour{red}{j}}^* \otimes X_{\textcolour{blue}{i}}^*)$ we have
    \begin{equation*}
        \begin{tikzpicture}
            \draw[draw=blue, double=white, double distance=5pt] (4.3,0.2) to [out=up,in=right] (3.4,1) to [out=left,in=up] (2.5,0) to [out=down,in=left] (3.4,-1) to [out=right,in=down] (4.3,-0.2);
            \draw[draw=red, double=white, double distance=5pt] (3.7,0.2) to [out=up,in=right] (3.4,0.5) to [out=left,in=left] (3.4,-0.5) to [out=right,in=down] (3.7,-0.2);
            \draw[draw=blue, ->] (2.52,-0.2) .. controls (2.48,0) .. (2.52,0.2);
            \draw[draw=red, ->] (3.12,-0.2) .. controls (3.08,0) .. (3.12,0.2);
            \draw[rounded corners] (3.4,-0.2) rectangle (4.6,0.2) node[pos=.5] {$g$};
            \node[] at (2.28,0) {$\textcolour{blue}{i}$};
            \node[] at (2.88,0) {$\textcolour{red}{j}$};
            
            \node[] at (-0.9,0) {$\left\langle Z_\mathcal{C}(\eta)_{\textcolour{blue}{i}\textcolour{red}{j}}^{\textcolour{blue}{i}\textcolour{red}{j}}(\id_{X_{\textcolour{blue}{i}}\otimes X_{\textcolour{red}{j}}}), g\right\rangle_{\tilde{\Gamma}_{\Sigma'}} = \tr(g) =$};
            \node[] at (7.35,0) {$= \left\langle RT_\mathcal{C}(|\eta|_\mathcal{G})_{\textcolour{blue}{i}\textcolour{red}{j}}^{\textcolour{blue}{i}\textcolour{red}{j}}(\id_{X_{\textcolour{blue}{i}}\otimes X_{\textcolour{red}{j}}}), g\right\rangle_{\tilde{\Gamma}_{\Sigma'}}$.};
        \end{tikzpicture}
    \end{equation*}
    Hence, we have shown (\ref{eq:rtz-dual}) for $\gamma=\eta$.

    For $\gamma=\eta^\dag$, the only nonzero case is when $\mathbf{i} = \mathbf{j}$, as otherwise both codomains are the zero vector space. Now for any $f \in \Hom(X_{\textcolour{blue}{i}}\otimes X_{\textcolour{red}{j}}, X_{\textcolour{blue}{i}}\otimes X_{\textcolour{red}{j}})$ we have
    \begin{equation*}
        \begin{tikzpicture}
            \draw[draw=blue, double=white, double distance=5pt] (-4.3,0.2) to [out=up,in=left] (-3.4,1) to [out=right,in=up] (-2.5,0) to [out=down,in=right] (-3.4,-1) to [out=left,in=down] (-4.3,-0.2);
            \draw[draw=red, double=white, double distance=5pt] (-3.7,0.2) to [out=up,in=left] (-3.4,0.5) to [out=right,in=right] (-3.4,-0.5) to [out=left,in=down] (-3.7,-0.2);
            \draw[draw=blue, ->] (-2.52,0.2) .. controls (-2.48,0) .. (-2.52,-0.2);
            \draw[draw=red, ->] (-3.12,0.2) .. controls (-3.08,0) .. (-3.12,-0.2);
            \draw[rounded corners] (-4.6,-0.2) rectangle (-3.4,0.2) node[pos=.5] {$f$};
            \node[] at (-2.28,0) {$\textcolour{blue}{i}$};
            \node[] at (-2.88,0) {$\textcolour{red}{j}$};

            \node[] at (-3.9,1.7) {$\frac1{p_+}\left\langle Z_\mathcal{C}(\eta^\dag)_{\textcolour{blue}{i}\textcolour{red}{j}}^{\textcolour{blue}{i}\textcolour{red}{j}}(f), \id_{X_{\textcolour{blue}{i}}\otimes X_{\textcolour{red}{j}}} \right\rangle_{\tilde{\Gamma}_{\Sigma'}} = \frac1{\dim\textcolour{blue}{i}\dim\textcolour{red}{j}}\tr(f)\tr(\id_{X_{\textcolour{blue}{i}}})\tr(\id_{X_{\textcolour{red}{j}}}) = \tr(f)$};
            \node[] at (-4.8,0) {$=$};
            \node[] at (1,0) {$= \frac1{p_+}\left\langle RT_\mathcal{C}(|\eta^\dag|_\mathcal{G})_{\textcolour{blue}{i}\textcolour{red}{j}}^{\textcolour{blue}{i}\textcolour{red}{j}}(f), \id_{X_{\textcolour{blue}{i}}\otimes X_{\textcolour{red}{j}}} \right\rangle_{\tilde{\Gamma}_{\Sigma'}}$.};
        \end{tikzpicture}
    \end{equation*}
    Thus, we have shown (\ref{eq:rtz-dual}) for $\gamma=\eta^\dag$.

    Lastly, for $\epsilon^\dag$, recall that the definition of $Z_\mathcal{C}(\epsilon^\dag)$ involved contracting along the edge going through the left leg of the pants and copants. Explicitly, this corresponds to applying Lemmas \ref{lem:cyc-shift} and \ref{lem:tree} to get the isomorphism
    \begin{align*}
        &\bigoplus_{j,k}\Hom(X_i,X_j\otimes X_k) \otimes \Hom(X_j\otimes X_k, X_i) \\ &\quad\cong \bigoplus_{j,k}\Hom(X_i \otimes X_k^*,X_j) \otimes \Hom(X_j, X_i \otimes X_k^*) \cong \bigoplus_{k} \Hom(X_i \otimes X_k^*, X_i \otimes X_k^*).
    \end{align*}
    Now, for any $g\in \bigoplus_{\textcolour{Green}{k}}\Hom(X_{\textcolour{Green}{k}} \otimes X_{\textcolour{blue}{i}}^*, X_{\textcolour{Green}{k}} \otimes X_{\textcolour{blue}{i}}^*)$, say $g=\sum_{\textcolour{Green}{k}} g_{\textcolour{Green}{k}}$, we have
    \begin{equation*}
        \begin{tikzpicture}
            \draw[draw=blue, double=white, double distance=5pt] (-4.3,0.2) to [out=up,in=right] (-5.2,1) to [out=left,in=up] (-6.1,0) to [out=down,in=left] (-5.2,-1) to [out=right,in=down] (-4.3,-0.2);
            \draw[draw=Green, double=white, double distance=5pt] (-3.7,0.2) to [out=up,in=left] (-3.4,0.5) to [out=right,in=right] (-3.4,-0.5) to [out=left,in=down] (-3.7,-0.2);
            \draw[draw=blue, ->] (-6.08,-0.2) .. controls (-6.12,0) .. (-6.08,0.2);
            \draw[draw=Green, ->] (-3.12,0.2) .. controls (-3.08,0) .. (-3.12,-0.2);
            \draw[rounded corners] (-4.6,-0.2) rectangle (-3.4,0.2) node[pos=.5] {$g_{\textcolour{Green}{k}}$};
            \node[] at (-5.92,0) {$\textcolour{blue}{i}$};
            \node[] at (-2.88,0) {$\textcolour{Green}{k}$};

            \node[] at (-9,1.5) {$p_+\left\langle Z_\mathcal{C}(\epsilon^\dag)_{\textcolour{blue}{i}}^{\textcolour{blue}{i}}(\id_{X_{\textcolour{blue}{i}}}), g\right\rangle_{\tilde{\Gamma}_{\Sigma'}} = \displaystyle\sum_{\textcolour{Green}{k}} \dim {\textcolour{Green}{k}} \tr(g)$};
            \node[] at (-7.45,-0.2) {$= \displaystyle\sum_{\textcolour{Green}{k}} \dim {\textcolour{Green}{k}}$};
            \node[] at (-0.15,0) {$=p_+\left\langle RT_\mathcal{C}(|\epsilon^\dag|_\mathcal{G})_{\textcolour{blue}{i}}^{\textcolour{blue}{i}}(\id_{X_{\textcolour{blue}{i}}}), g\right\rangle_{\tilde{\Gamma}_{\Sigma'}}$.};
        \end{tikzpicture}
    \end{equation*}
    In the last step, the $\dim {\textcolour{Green}{k}}$ is exactly the $\dim_\mathrm{int} \tilde{K}'$ term in (\ref{eq:rt-gen}). Hence, we have shown (\ref{eq:rtz-dual}) for $\gamma=\epsilon^\dag$.
\end{proof}

This concludes our proof of Proposition \ref{prop:mtc-rep-mtc}.

\subsection{Fullness of $|-|_\mathcal{G}$}

The proof of the classification of linear representations of $\Bord{sig/2}$ given in \cite{partA} relies on the fact that the symmetric monoidal functor $|-|_\mathcal{G}:\mathbf{F}(\mathcal{G})\rightarrow\Bord{sig/2}$ is an equivalence. This in turn (along with the classification results of \cite{BDSV4}) depends on the corresponding result for $|-|_\mathcal{O}$ (Theorem \ref{conj:pres-or}), the proof of which will be completed in \cite{pres-bord}. In our alternate approach, we are replacing this with the statement that $|-|_\mathcal{O}$ (and hence $|-|_\mathcal{G}$) is essentially surjective, essentially full and full. This is substantially easier to prove; in this subsection, we provide direct proofs of these. In the next subsection, we will then be able to formally deduce in Proposition \ref{prop:pb-inj} that precomposition with $|-|_\mathcal{G}$ induces an injection on isomorphism classes of linear representations.

Firstly, it is clear that $|-|_\mathcal{O}:\mathbf{F}(\mathcal{O})\rightarrow\Bord{or,dec}$ and $|-|_\mathcal{G}:\mathbf{F}(\mathcal{G})\rightarrow\Bord{sig/2,dec}$ are essentially surjective and essentially full. (If we replaced $\mathbf{F}(\mathcal{O})$ and $\mathbf{F}(\mathcal{G})$ with the free quasistrict symmetric monoidal bicategories $\mathbf{F}_{\mathrm{qs}}(\mathcal{O})$ and $\mathbf{F}_{\mathrm{qs}}(\mathcal{G})$ in the sense of \cite[\textsection 2.12]{csp-phd} and \cite[Definition 62]{BDSV2}, then these would in fact induce bijections at the level of objects and 1-morphisms.) The rest of this subsection is devoted to proving that $|-|_\mathcal{O}$ is full, which is the content of Proposition \ref{prop:modo-full}. By applying Proposition \ref{prop:go-compat}, we can then deduce in Proposition \ref{prop:modg-full} that $|-|_\mathcal{G}$ is full as well.

\begin{lem} \label{lem:standardise}
    For any two isomorphic 1-morphisms in $\Bord{or,dec}$, there exists an invertible 2-morphism between them which lies in the image of $|-|_\mathcal{O}$.
\end{lem}
\begin{proof}
    As shown in \cite[\textsection 1.4]{Kock}, $\mathbf{Bord}_{1,2}^\mathrm{or}$ is the symmetric monoidal category generated by a single object $\tikztinycirc$ as well as generating morphisms $\tikztinypants,\tikztinycopants,\tikztinycup,\tikztinycap$, subject to a certain set of relations. Given 1-morphisms in $\Bord{or,dec}$ (which are formal compositions of disjoint unions of the same four generating 1-morphisms), they are isomorphic if and only if their counterparts in $\mathbf{Bord}_{1,2}^\mathrm{or}$ may be related by the relations in \cite[\textsection 1.4]{Kock}. It then suffices to show that all these relations may arise from invertible 2-morphisms in the image of $|-|_\mathcal{O}$.

    Indeed, most of these relations correspond to the images of the invertible generators $\alpha^\pm,\rho^\pm,\lambda^\pm,\phi_1^\pm,\phi_2^\pm,\beta^\pm$ of $\mathcal{O}$. The only remaining relations are those which encode the comonoid structure of $\tikztinycopants,\tikztinycap$. These in turn exactly correspond to the compositions $\check{\alpha},\check{\rho},\check{\lambda}$ in \cite[Definition 12]{BDSV2}, which are written in terms of generators in $\mathcal{R}$.
\end{proof}

\begin{lem} \label{lem:all-diffeos}
    All 2-morphisms in $\Bord{or}$ which are given by mapping cylinders of diffeomorphisms lie in the image of $|-|_\mathcal{O}$.
\end{lem}
\begin{proof}
    It suffices to show that for each 1-morphism and for any element of its mapping class group, the corresponding mapping cylinder lies in the image of $|-|_\mathcal{O}$. By composing with the invertible 2-morphisms in Lemma \ref{lem:standardise}, it suffices to show this for (at least) one 1-morphism in each isomorphism class. Note that each 1-morphism is isomorphic to a disjoint union of 1-morphisms of the following form: a composition of some number of pants, followed by alternating pairs of copants and pants, and then a composition of some number of copants. In the case where the source or target is the empty 1-manifold, then we may instead allow a single cup at the bottom or a single cap at the top. As in \cite[\textsection 1.4.16]{Kock}, call this a \emph{normal form}. An example is depicted in Figure \ref{fig:dehn-gen}.

    \begin{figure}[hbt!]
        \begin{tikzpicture}
            \node [Pants, verywide, bot] (A) at (0,0) {};
            \node [Pants, anchor=belt, bot] (B) at (A.leftleg) {};
            \node [Cyl, anchor=top, bot] (C) at (A.rightleg) {};
            \node[Pants, anchor=belt, bot, red, bothred] (D) at (C.bottom) {};
            \node[Cyl, anchor=top, bot, red] (E) at (B.leftleg) {};
            \node[Cyl, anchor=top, bot, red] (F) at (B.rightleg) {};
            \node[Copants, anchor=belt, bot, red] (G) at (A.belt) {};
            \node[Pants, anchor=leftleg, bot, red] (H) at (G.leftleg) {};
            \node[Copants, anchor=belt, bot, red] (I) at (H.belt) {};
            \node[Pants, anchor=leftleg, bot, red] (J) at (I.leftleg) {};
            \node[Copants, anchor=belt, bot, red] (K) at (J.belt) {};
            \node[Pants, anchor=leftleg, bot, red] (L) at (K.leftleg) {};
            \node[Cap, bot] (M) at (L.belt) {};
            \node[BotGreen, rotate=270, xscale=2.4] at (A.belt) {};
            \node[BotGreen, rotate=270, xscale=2.4] at (J.belt) {};
            \node[BotGreen, rotate=270, xscale=2.4] at (H.belt) {};
            \draw[draw=blue, on layer=foreground] (H.leftleg) to [out=up,in=left] (H.center) to [out=right,in=up] (H.rightleg) to [out=down,in=right] (G.center) to [out=left,in=down] (H.leftleg);
            \draw[draw=blue, on layer=foreground] (J.leftleg) to [out=up,in=left] (J.center) to [out=right,in=up] (J.rightleg) to [out=down,in=right] (I.center) to [out=left,in=down] (J.leftleg);
            \draw[draw=blue, on layer=foreground] (L.leftleg) to [out=up,in=left] (L.center) to [out=right,in=up] (L.rightleg) to [out=down,in=right] (K.center) to [out=left,in=down] (L.leftleg);
        \end{tikzpicture}
        \caption{Almost-generating set of Dehn twists for normal form.} \label{fig:dehn-gen}
    \end{figure}
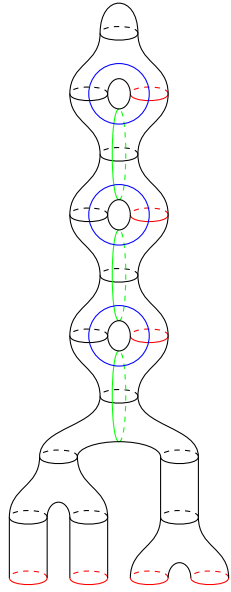
    Now, the mapping class group for 1-morphisms of this form is generated by certain Dehn twists \cite{mcg-gen}, almost all of which can be easily drawn. The Dehn twists around the red curves are given by $|\theta|_\mathcal{O}$. The Dehn twists around the green and blue curves correspond respectively to the images of the compositions $\II,A$ in Definition \cite[Definition 7.7]{partA}. The remaining generators, which are not labelled in Figure \ref{fig:dehn-gen}, correspond to conjugating an application of $\II$ by associators. For example, one such Dehn twist is given by the following composition:

    \begin{equation*}
        \smallbordisms
        \begin{tz}
            \node[Pants, bot, wide] (A) at (0,0) {};
            \node[Pants,  bot, anchor=belt] (B) at (A.leftleg) {};    
            \node[Cyl, bot, anchor=top] (C) at (A.rightleg) {}; 
            \node[Copants, top, bot, anchor=belt] (D) at (A.belt) {};
            \selectpart[green]{(A-belt) (B-leftleg) (C-bottom)};
        \end{tz}
        \longxdoubleto{\alpha}
        \begin{tz}
            \node[Pants, top, bot, wide] (A) at (0,0) {};
            \node[Pants,  bot, anchor=belt] (B) at (A.rightleg) {};    
            \node[Cyl, bot, anchor=top] (C) at (A.leftleg) {}; 
            \node[Copants, top, bot, anchor=belt] (D) at (A.belt) {};
            \selectpart[green]{(D-leftleg) (A-leftleg) (A-rightleg)};
        \end{tz}
        \longxdoubleto{\II}
        \begin{tz}
            \node[Pants, top, bot, wide] (A) at (0,0) {};
            \node[Pants,  bot, anchor=belt] (B) at (A.rightleg) {};    
            \node[Cyl, bot, anchor=top] (C) at (A.leftleg) {}; 
            \node[Copants, top, bot, anchor=belt] (D) at (A.belt) {};
            \selectpart[green]{(A-belt) (C-bottom) (B-rightleg)};
        \end{tz}
        \longxdoubleto{\alpha^{-1}}
        \begin{tz}
            \node[Pants, bot, wide] (A) at (0,0) {};
            \node[Pants,  bot, anchor=belt] (B) at (A.leftleg) {};    
            \node[Cyl, bot, anchor=top] (C) at (A.rightleg) {}; 
            \node[Copants, top, bot, anchor=belt] (D) at (A.belt) {};
        \end{tz}
    \end{equation*}
    If there are copants instead of a single cap at the top, then we may similarly conjugate the topmost $\II$ by the composition $\check{\alpha}$ in \cite[Definition 12]{BDSV2}.
    
    Hence, all such mapping cylinders indeed lie in the image of $|-|_\mathcal{O}$.
\end{proof}

\begin{prop} \label{prop:modo-full}
    $|-|_\mathcal{O}$ is full on 2-morphisms.
\end{prop}
\begin{proof}
    This is equivalent to showing that the images of the generators of $\mathcal{O}$ under $|-|_\mathcal{O}$ generate all 2-morphisms of $\Bord{sig/2}$ under vertical and horizontal composition, as well as disjoint unions. Note that the non-invertible generators of $\mathcal{O}$ account for all possible types of handle addition, and hence, by composing with diffeomorphisms, all 2-morphisms in $\Bord{or}$ which arise from handle additions are in the image of $|-|_\mathcal{O}$.

    Hence, it remains to show that each 2-morphism $M$ in $\Bord{or}$ is the vertical composition of some number of handle additions. Indeed, consider the function $f: \partial M \rightarrow [0,1]$ that is $0$ on $\parin M$, $1$ on $\parout M$, and on $\parIn M = \parin\parin M \times [0,1]$ and $\parOut M = \parout\parin M \times [0,1]$ is projection onto the second coordinate. By \cite[Lemma 2.1]{morse-cobord}, this extends to a Morse function $F: M \rightarrow [0,1]$. By construction, this does not have any critical points on the vertical boundary of $M$, and so all the critical points lie in the interior. This provides a decomposition of $M$ into handle additions, as desired.
\end{proof}

\begin{prop} \label{prop:modg-full}
    $|-|_\mathcal{G}$ is full on 2-morphisms.
\end{prop}
\begin{proof}
    Consider the composition $\mathbf{F}(\mathcal{G})\rightarrow\Bord{sig/2}\rightarrow\Bord{or}$. By \cite[Proposition 6.16]{partA}, this factors through $\mathbf{F}(\mathcal{O})$. The natural functor from $\mathbf{F}(\mathcal{G})$ to $\mathbf{F}(\mathcal{O})$ is full, and by Proposition \ref{prop:modo-full}, so is $|-|_\mathcal{O}$. Thus, their composition is full. Hence, for each 2-morphism in $\Bord{or}$, the corresponding $\mathbb{Z}$-family of 2-morphisms in $\Bord{sig/2}$ has at least one 2-morphism in the image of $|-|_\mathcal{G}$.

    On the other hand, note that taking the disjoint union with $|\zeta^n|_\mathcal{G}$ shifts the signature by $2n$ for each $n\in\mathbb{Z}$, and thus all 2-morphisms of $\Bord{sig/2}$ lie in the image of $|-|_\mathcal{G}$.
\end{proof}

\subsection{The classification of once-extended 3D field theories} \label{subsection:classify}

We may now combine the results of the previous two subsections to obtain a proof of the classification of $\Bord{sig/2}$ that is independent of Theorem \ref{conj:pres-or}.

\begin{defn}
    Let $\mathbf{MTC}$ denote the set of equivalence classes of finite direct sums of modular tensor categories whose anomalies are equal.
\end{defn}
\begin{defn}
    For a symmetric monoidal bicategory $\mathbf{C}$, let $\mathrm{Rep}(\mathbf{C})$ denote the set of equivalence classes of symmetric monoidal functors $\mathbf{C}\rightarrow \Vect$, where $\Vect$ is as defined in Definition \ref{defn:rep}. 
\end{defn}

There is a function $|-|_\mathcal{G}^*: \mathrm{Rep}(\Bord{sig/2})\rightarrow\mathrm{Rep}(\mathbf{F}(\mathcal{H}))$ given by precomposition with $|-|_\mathcal{G}$. We may deduce directly from Proposition \ref{prop:modg-full} that this function is an injection.
\begin{prop}\label{prop:pb-inj}
    The function $|-|_\mathcal{G}^*: \mathrm{Rep}(\Bord{sig/2})\rightarrow\mathrm{Rep}(\mathbf{F}(\mathcal{H}))$ is injective.
\end{prop}
\begin{proof}
    Consider the 2-extension $\mathcal{G}'$ of $\mathcal{G}$ constructed as follows: for every two 2-morphisms which $|-|_\mathcal{G}$ sends to the same 2-morphism in $\Bord{sig/2}$, we add a relation declaring these to be equal. Then, we have a symmetric monoidal functor $q:\mathbf{F}(\mathcal{G})\rightarrow\mathbf{F}(\mathcal{G'})$. By construction, $|-|_\mathcal{G}$ respects each new relation, so it factors as $|-|_{\mathcal{G}'}\circ q$ for some symmetric monoidal functor $|-|_{\mathcal{G}'}: \mathbf{F}(\mathcal{G}')\rightarrow \Bord{sig/2}$. On the other hand, $|-|_{\mathcal{G}'}$ is by construction an equivalence of symmetric monoidal bicategories as it is essentially surjective, essentially full and fully faithful. It thus suffices to show that $q^*: \mathrm{Rep}(\mathbf{F}(\mathcal{G}')) \rightarrow \mathrm{Rep}(\mathbf{F}(\mathcal{G}))$ is injective.
    
    Indeed, by \cite[Proposition 2.74]{csp-phd}, the bicategory of symmetric monoidal functors $\mathbf{F}(\mathcal{G}')\rightarrow \Vect$ is the full subbicategory of the bicategory of symmetric monoidal functors $\mathbf{F}(\mathcal{G})\rightarrow \Vect$ consisting of the functors which satisfy the added relations. This then induces an inclusion at the level of equivalence classes of representations.
\end{proof}

Finally, we obtain a Cerf theory-independent proof of the classification of once-extended 3-dimensional TQFTs.
\begin{thm}[{cf. \cite[Theorem C]{partA}}]\label{thm:rep-mtc}
    Symmetric monoidal functors $\Bord{sig/2}\rightarrow \Vect$ are classified by finite direct sums of modular tensor categories whose anomalies are equal.
\end{thm}
\begin{proof}
    Thus far, we have the following maps of sets:
    \begin{equation} \label{eq:rep-mtc-diagram}
        \begin{tikzcd}
            {\mathrm{Rep}(\Bord{sig/2})} && \\
            && {\mathbf{MTC}} \\
            {\mathrm{Rep}(\mathbf{F}(\mathcal{G}))}
            \arrow["{|-|_\mathcal{G}^*}"', from=1-1, to=3-1]
            \arrow["{RT_{(-)}}"', from=2-3, to=1-1]
            \arrow["{Z_{(-)}}"{pos=0.4}, shift left, from=2-3, to=3-1]
            \arrow["{(-)(\tikztinycirc)}"{pos=0.4}, shift left, from=3-1, to=2-3]
        \end{tikzcd}
    \end{equation}
    By Theorem~\ref{thm:rep-fg}, the two functions between $\mathrm{Rep}(\mathbf{F}(\mathcal{G}))$ and $\mathbf{MTC}$ are inverses to each other, hence providing a bijection. We wish to show that the remaining two maps are also (compatible) bijections.

    By Proposition \ref{prop:mtc-rep-mtc}, $|-|_\mathcal{G}^* \circ RT_{(-)} = Z_{(-)}$ is a bijection, and so $|-|_\mathcal{G}^*$ is surjective. On the other hand, by Proposition \ref{prop:pb-inj}, $|-|_\mathcal{G}^*$ is injective, and so it provides a bijection, as needed.
\end{proof}
\begin{rk}
    The sets in (\ref{eq:rep-mtc-diagram}) are equivalence classes of objects in certain bicategories. Namely, these are the bicategories of representations of $\Bord{sig/2}$ and $\mathbf{F}(\mathcal{G})$, and the bicategory of finite direct sums of modular tensor categories with equal anomalies. One would expect the set-theoretic maps to arise as equivalences of bicategories. Proposition \ref{prop:pb-inj} in fact holds in a bicategorical setting, but the other maps in (\ref{eq:rep-mtc-diagram}) would also have to be upgraded to bicategorical versions, which is beyond the scope of this paper.
\end{rk}

Using the methods in the proof of Theorem \ref{thm:rep-mtc}, we may also prove an analogous result on the classification of representations of $\Bord{or}$. Thus, we may recover proofs of the main results of \cite{BDSV4} which are independent of Theorem \ref{conj:pres-or}.

\begin{thm}[{cf. \cite[Theorem 2]{BDSV4}}] \label{thm:rep-bordor}
    Symmetric monoidal functors of $\Bord{or}\rightarrow \Vect$ are classified by finite direct sums of modular tensor categories with anomaly $1$.
\end{thm}
\begin{proof}
    Given a modular tensor category $\mathcal{C}$ with anomaly $1$, $RT_\mathcal{C}(|\zeta|_\mathcal{G})$ acts as the identity, and thus $RT_\mathcal{C}: \Bord{sig/2} \rightarrow \Vect$ factors through $\Bord{or}$. By taking direct sums, we may extend this to the general case.

    Now, we may follow the proof of Theorem \ref{thm:rep-mtc}, but replacing $\mathcal{G}$ with $\mathcal{O}$. The representations of $\mathbf{F}(\mathcal{O})$ are exactly the representations $Z_\mathcal{C}$ of $\mathbf{F}(\mathcal{G})$ which send the generators $z$ and $\zeta$ to the identity. By \cite[Lemma 7.14]{partA}, this happens exactly when $\mathcal{C}$ is a finite direct sum of modular tensor categories with anomaly $1$. By \cite[Proposition 2.74]{csp-phd}, the bicategory of representations of $\mathbf{F}(\mathcal{O})$ is the full subbicategory of the representations of $\mathbf{F}(\mathcal{G})$ consisting of these $Z_\mathcal{C}$. Thus, passing to equivalence classes of representations, we see that representations of $\mathbf{F}(\mathcal{O})$ are exactly classified by finite direct sums of modular tensor categories with anomaly $1$.
    
    Next, note that Proposition \ref{prop:mtc-rep-mtc} implies the version for $\mathcal{O}$, i.e. that the constructions $RT_{(-)}$ and $Z_{(-)}$ are compatible with $|-|_\mathcal{O}$. Additionally, by Proposition \ref{prop:modo-full}, $|-|_\mathcal{O}$ is full, and so the same argument as Proposition \ref{prop:pb-inj} shows that pulling back by $|-|_\mathcal{O}$ induces an injection on sets of equivalence classes of representations. Then, following the proof of Theorem \ref{thm:rep-mtc}, the desired result follows.
\end{proof}

We may also obtain a similar classification result for the central extensions of $\Bord{or}$ that arise as multiples of $\Bord{sig/2}$. Let $t$ be a positive integer. Then we may construct a symmetric monoidal extension of $\Bord{or}$ that is ``$t$ times'' of $\Bord{sig/2}$. At the level of cocycles as described in \cite[\textsection 2]{partA}, this corresponds to multiplying each term by $t$. However, we may also construct such an extension explicitly.

\begin{defn}
    For a positive integer $t$, let $\mathbf{Bord}_{1,2,3}^{t\mathrm{sig/2}}$ denote the bicategory with the same objects and 1-morphisms as $\Bord{sig}$, but with 2-morphisms being equivalence classes of pairs $(M,n)$ where $n$ is of the form $m(M) + \frac2tu$ for any $u\in\mathbb{Z}$. (Here, $m(M)$ is as in Definition \ref{defn:sig/2}.) The composition of 1-morphisms, vertical and horizontal composition of 2-morphisms, and monoidal structure are all as in $\Bord{sig}$.
\end{defn}
\begin{rk}
    This is a symmetric monoidal extension of $\Bord{or}$ by $\mathbb{Z}$ in the sense of \cite[\textsection 2]{partA}. By restricting $u$ to be a multiple of $t$, this contains $\Bord{sig/2}$ as an index $t$ symmetric monoidal subbicategory.
\end{rk}

\begin{thm} \label{thm:rep-bordext}
    For any positive integer $t$, symmetric monoidal functors $\mathbf{Bord}_{1,2,3}^{t\mathrm{sig/2}}\rightarrow\Vect$ are classified by finite direct sums of modular tensor categories whose anomalies are equal, along with a single choice of $t^\mathrm{th}$ root of this anomaly.
\end{thm}
\begin{proof}
    By Theorem \ref{thm:rep-mtc}, any $Z:\mathbf{Bord}_{1,2,3}^{t\mathrm{sig/2}}\rightarrow \Vect$ must restrict to a direct sum of $RT_{\mathcal{C}_i}$ on $\Bord{sig/2}$ where $\mathcal{C} = Z(\tikztinycirc) = \bigoplus_i \mathcal{C}_i$ is a finite direct sum of modular tensor categories whose anomalies are equal. Let $\zeta_t$ be the 2-morphism in $\mathbf{Bord}_{1,2,3}^{t\mathrm{sig/2}}$ that encodes the $\mathbb{Z}$-action. Then, we have $\zeta_t^t = |\zeta|_\mathcal{G}$, and so $\zeta_t$ acts as some $t^\mathrm{th}$ root $r$ of the common anomaly. By considering $Z(\tikztinycirc)$ and $\zeta_t$, we have obtained a finite direct sum of modular tensor categories with equal anomalies $\mathcal{C}=\bigoplus_i \mathcal{C}_i$ and a $t^\mathrm{th}$ root $r$ of this anomaly. Given an equivalence of representations $Z\simeq Z'$, restricting to $\Bord{sig/2}$ gives an equivalence of ribbon linear categories $\mathcal{C}\simeq \mathcal{C}'$ by Theorem \ref{thm:rep-mtc}. Additionally, $Z(\zeta)$ and $Z'(\zeta)$ must both act as multiplication by the same scalar, and so we also have $r=r'$.

    Conversely, given a such a $\mathcal{C}$ and a $t^\mathrm{th}$ root $r$ of the anomaly, we may extend $RT_\mathcal{C}$ from $\Bord{sig/2}$ to $\mathbf{Bord}_{1,2,3}^{t\mathrm{sig/2}}$ by letting $\zeta_t$ act as multiplication by $r$. This clearly defines a symmetric monoidal functor $\mathbf{Bord}_{1,2,3}^{t\mathrm{sig/2}}\rightarrow \Vect$, by the compatibility of the $\mathbb{Z}$-action with composition in both directions and the monoidal structure. Moreover, if we had some equivalent ribbon linear category $\mathcal{C}'\simeq \mathcal{C}$ (which would necessarily be a finite direct sum of modular tensor categories with the same anomaly) and chose the same root $r$ of the anomaly, then the two representations of $\mathbf{Bord}_{1,2,3}^{t\mathrm{sig/2}}$ constructed in this way would be equivalent when restricted to $\Bord{sig/2}$. As $\zeta_t$ acts as multiplication by the same scalar in both representations, this extends to an equivalence of representations of $\mathbf{Bord}_{1,2,3}^{t\mathrm{sig/2}}$.
\end{proof}
\begin{rk}
    In particular, the cases $t=2,6$ are exactly \cite[Theorem 3]{BDSV4} and \cite[Theorem 4]{BDSV4} respectively, but proven independently of Theorem \ref{conj:pres-or}.
\end{rk}
\begin{rk}
    With some care for the minus signs, one may also prove an analogous statement for negative $t$.
\end{rk}

Finally, we may also classify the linear representations of the componentwise versions of $\Bord{sig/2}$ and $\Bord{sig}$.
\begin{defn}[{\cite{BDSV3,partA}}]\label{defn:bordcsig}
    The \emph{componentwise signature bordism bicategory} $\Bord{csig}$ is the symmetric monoidal bicategory with 
    \begin{itemize}
        \item Objects and 1-morphisms: same as $\Bord{or}$
        \item 2-morphisms: equivalence classes of $(M,\mathbf{n})$, where $M = \sqcup_{i=1}^c M_i$, each $M_i$ connected, and $\mathbf{n} = (n_1,\ldots,n_c)$ is an $c$-tuple of integers.
    \end{itemize}
    When a vertical or horizontal composition involves multiple connected components, the signature term assigned to the combined connected component is given by the sum of all the signature terms, plus the error term given by Wall's invariant.

    The \emph{componentwise half-signature bordism bicategory} $\Bord{csig/2}$ is the symmetric monoidal subbicategory of $\Bord{csig}$ consisting only of the 2-morphisms $(M,\mathbf{n})$ such that $n_i \equiv m(M_i) \pmod{2}$ for each $i$.
\end{defn}

The classification of linear representations of $\Bord{csig/2}$ as given in \cite{partA} was dependent on Theorem \ref{conj:pres-or}. Following the arguments above, we can now give a proof that is independent of this result.
\begin{thm}[{cf. \cite[Theorem E]{partA}}] \label{thm:rep-bordcsig/2}
    Symmetric monoidal functors $\Bord{csig/2}\rightarrow\Vect$ are classified by finite direct sums of modular tensor categories.
\end{thm}
\begin{proof}
    The proof is analogous to that of Theorem \ref{thm:rep-mtc}, but with $\mathcal{G}$ replaced by $\mathcal{H}$. By \cite[Theorem 7.15]{partA}, linear representations of $\mathbf{F}(\mathcal{H})$ are classified by finite direct sums of modular tensor categories.
    
    As noted at the end of \cite[\textsection 6]{partA}, an analogous construction to \cite[Construction 6.15]{partA} yields a symmetric monoidal functor $|-|_{\mathcal{H}}: \mathbf{F}(\mathcal{H}) \rightarrow \Bord{csig/2}$. This is essentially surjective and essentially full by construction, and the same arguments as Propositions \ref{prop:modg-full} and \ref{prop:pb-inj} let us deduce that from Proposition \ref{prop:modo-full} that $|-|_\mathcal{H}$ is full and so pulling back by $|-|_\mathcal{H}$ induces an injection on the sets of equivalence classes of representations.

    On the other hand, note that the definition of $RT_\mathcal{C}: \Bord{sig/2} \rightarrow \Vect$ is given componentwise, and so readily extends to a symmetric monoidal functor $\Bord{csig/2} \rightarrow \Vect$. By the same argument as in Proposition \ref{prop:mtc-rep-mtc}, this construction is compatible with $|-|_\mathcal{H}$ and the $Z_{(-)}$ constructed in \cite[Theorem 7.15]{partA}.
    
    Thus, following the proof of Theorem \ref{thm:rep-mtc}, our desired result follows.
\end{proof}

Replacing $\mathcal{H}$ with the corresponding presentation $\mathcal{M}$ for $\Bord{csig}$ (see \cite[Definition 3.9]{BDSV4}) also gives a proof of the classification of linear representations of componentwise signature bordism bicategory, again independently of Theorem \ref{conj:pres-or}.
\begin{thm}[{cf. \cite[Theorem 1]{BDSV4}}]
    Symmetric monoidal functors $\Bord{csig}\rightarrow\Vect$ are classified by finite direct sums of modular tensor categories, along with a choice of square root of anomaly of each direct summand.
\end{thm}
\begin{proof}
    As constructed in \cite{BDSV3}, there is a symmetric monoidal functor $|-|_\mathcal{M}: \mathbf{F}(\mathcal{M})\rightarrow \Bord{csig}$ (see also the Remark following \cite[Construction 6.15]{partA}). As in the proof of Theorem \ref{thm:rep-bordcsig/2}, this is essentially surjective and essentially full by construction, and follwoing the proofs of Propositions \ref{prop:modg-full} and \ref{prop:pb-inj}, we can show that it is full and so pulling back by $|-|_\mathcal{M}$ induces an injection on the sets of equivalence classes of representations.

    On the other hand, it is shown in \cite{BDSV4} that linear representations of $\mathbf{F}(\mathcal{M})$ are classified by finite direct sums of modular tensor categories with a choice of square root of anomaly of each direct summand. Following the arguments of Theorem \ref{thm:rep-bordcsig/2}, we obtain our desired result.
\end{proof}

Finally, we note that in all the proofs of the various classification results given in this section, all symmetric monoidal functors constructed from modular tensor categories (and possibly extra data) arise via the extended Reshetikhin--Turaev construction. In essence, we have proven Theorem \ref{thmA:all-rt}.

\printbibliography
\end{document}